\documentclass[a4paper, 12pt]{report}

\renewcommand{\contentsname}

\usepackage{a4wide}
\usepackage{amsmath}
\usepackage{amssymb}
\usepackage{amsthm}
\usepackage{latexsym}
\usepackage{graphicx}
\usepackage[english]{babel}
\usepackage{makeidx}
\usepackage{setspace}

\newtheorem{theorem}{Theorem}[section]
\newtheorem{lemma}[theorem]{Lemma}
\newtheorem{corollary}[theorem]{Corollary}
\newtheorem{proposition}[theorem]{Proposition}

\theoremstyle{definition}
\newtheorem{definition}[theorem]{Definition}
\newtheorem{example}[theorem]{Example}

\theoremstyle{remark}
\newtheorem{remark}[theorem]{Remark}

\begin{document}

\selectlanguage{english}
\frenchspacing

\onehalfspacing

\Huge
\begin{center}
\textbf{Classifications of Cohen-Macaulay modules - The base ring associated to a transversal polymatroid}
\vspace{20 pt}
\[\]
\[\]

Ph.D. Thesis
\large
\vspace{20 pt}

Alin \c{S}tefan

\vspace{20 pt}

Adviser: Professor dr. Dorin Popescu

\vspace{200 pt}

Institute of Mathematics\\
of Romanian Academy \\

\vspace{20 pt}
September 2008
\end{center}

\newpage

\huge
\begin{center}
\[\]
\[\]
\[\]
\[\]
\emph{For my son Dorin-Andrei}
\end{center}

\normalsize

\newpage
\section*{Abstract}
\[\]

In this thesis, we focus on the study  of the base rings associated
to some transversal polymatroids. A transversal polymatroid is a special
kind of discrete polymatroid. Discrete polymatroids were introduced by 
Herzog and Hibi \cite{HH} in 2002.

The thesis is structured in four chapters. Chapter $1$ starts with a short
excursion into convex geometry. Next, we look at some properties of affine
semigroup rings. We recall some basic definitions and known facts about 
semigroup rings. Next we give a brief introduction to matroids and 
discrete polymatroids. Finally, we remind some properties of the base rings
associated to discrete polymatroids. These properties will be needed in 
the next chapters of the thesis.

Chapter two is devoted to the study of the canonical module of the base ring
associated to a transversal polymatroid. We determine the facets of the 
polyhedral 
cone generated by the exponent set of monomials defining the base 
ring. This allows us to describe the canonical module in terms of the
relative interior of the cone. Also, this would allow one to compute
the $a-\rm{invariant}$  of the base ring. Since 
the base ring associated to a discrete polymatroid is normal it follows 
that Ehrhart function is equal with Hilbert function 
and knowing the $a-\rm{invariant}$ we can very easy get its 
Hilbert series. We end this chapter with the following 
open problem\\
{\bf{Open Problem:}}
Let $n\geq 4,$ $A_{i}\subset [n]$ for any $1\leq i \leq n$ and 
$K[{\bf{\mathcal{A}}}]$ be the base ring associated to 
the transversal polymatroid presented by
${\bf{\mathcal{A}}}=\{A_{1},\ldots ,A_{n}\}.$
If the Hilbert series is: 
\[H_{K[{\bf{\mathcal{A}}}]}(t) \ = \ \frac{1+h_{1} \ t + \ldots + 
h_{n-r} \ t^{n-r} }{(1-t)^{n}},\] then we have the following:\\
$1)$ If $r=1,$ then $type(K[{\bf{\mathcal{A}}}])=1+h_{n-2}-h_{1}.$\\
$2)$ If $2\leq r\leq n,$ then $type(K[{\bf{\mathcal{A}}}])=h_{n-r}.$

In chapter three we study intersections of Gorenstein base rings. These 
are also Gorenstein rings and we are interested when the intersections of
Gorenstein base rings are the base rings associated to some transversal 
polymatroids. More precisely, we give 
necessary and sufficient conditions for the intersection of two base 
rings to be still a base ring of a transversal polymatroid. 

In chapter four we study when the transversal polymatroids presented by \\
${\bf{\mathcal{A}}}=\{A_1, A_2,\ldots, A_m\}$ with  $|A_{i}|=2$ 
have the base ring $K[{\bf{\mathcal{A}}}]$ Gorenstein. Using Worpitzky
identity, we prove that the numerator of the Hilbert series has the coefficients
Eulerian numbers and from \cite{BE} the Hilbert series is unimodal.

We acknowledge the support provided by the Computer Algebra Systems 
{\it{NORMALIZ}} \ \cite{BK} and {\it{SINGULAR}} \ \cite{GPS} for the
extensive experiments which helped us to obtain some of the results
in this thesis.

\[\]
\[\]
\[\]
\[\]
\[\]
\[\]
\[\]
\[\]
\[\]
\[\]
\[\]

\section*{Acknowledgements.}
\[\]

I wish to express my gratitude to my advisor Dorin Popescu for his
warm-hearted support and for his unresting willingness to discuss all my
questions and problems. I would like to thank for financial support
to University of Genova and CNCSIS program during one month visit to
University of Genova. I wish to express my thanks to Aldo Conca for
helpful discussions and hospitality. 
I would like to thank  \c{S}erban B\u{a}rc\u{a}nescu, 
Mihai Cipu, Viviana Ene, Ezra Miller for several useful suggestions and discussions.
Also, I am deeply indepted to my wife M\u{a}d\u{a}lina and my 
parents for them continuous encouragement and moral support during the 
preparation of this thesis. 

\newpage
\huge \noindent \textbf{Contents.} \vspace{30 pt} \normalsize
\[\]
\[\]
\[\]

\contentsline {chapter}{\numberline {}Abstract.}{3}
\contentsline {chapter}{\numberline {}Acknowledgements.}{4}
\contentsline {chapter}{\numberline {}Contents.}{5}
\contentsline {chapter}{\numberline {1}Background.}{6}
\contentsline {section}{\numberline {1.1}A short excursion into convex geometry.}{6}
\contentsline {section}{\numberline {1.2}Affine semigroup rings.}{9}
\contentsline {section}{\numberline {1.3}Discrete polymatroids.}{11}
\contentsline {section}{\numberline {1.4}The Ehrhart ring and the base ring of a discrete polymatroid.}{18}
\contentsline {chapter}{\numberline {2}The type of the base ring associated to a transversal polymatroid}{21}
\contentsline {section}{\numberline {2.1} Cones of dimension $n$ with $n+1$ facets.}{21}
\contentsline {section}{\numberline {2.2}The type of base ring associated to transversal polymatroids 
with the cone of dimension $n$ with $n+1$ facets.}{26}
\contentsline {section}{\numberline {2.3}Ehrhart function.}{32}
\contentsline {chapter}{\numberline {3}Intersections of base rings associated to transversal polymatroids}{36}
\contentsline {section}{\numberline {3.1}Intersection of cones of dimension $n$ with $n+1$ facets.}{36}
\contentsline {section}{\numberline {3.2}When is $K[A\cap B]$ the base ring associated to some transversal polymatroid?}{40}
\contentsline {chapter}{\numberline {4}A remark on the Hilbert series of transversal polymatroids}{53}
\contentsline {section}{\numberline {4.1}Segre product and the base ring associated to a transversal polymatroid.}{53}
\contentsline {section}{\numberline {4.2}Hilbert series.}{56}
\contentsline {section}{\numberline {}Bibliography.}{63}

\chapter{Background}

\section{A short excursion into convex geometry.}
\[\]

An \emph{affine space} in $\mathbb{R}^{n}$ is a translation of a linear
subspace of $\mathbb{R}^{n}.$ Let $A\subset \mathbb{R}^{n}$ and
\emph{aff}$(A)$ be the affine space generated by $A.$ Recall that 
\emph{aff}$(A)$ is the set of all \emph{affine combinations} 
of points in $A$:
\[\emph{aff}(A)=\{a_{1}p_{1}+\ldots + a_{r}p_{r} \ | \ p_{i}\in A, 
a_{1}+\ldots + a_{r}=1, a_{i}\in \mathbb{R}\}.\]
There is a unique linear 
subspace $V$ of $\mathbb{R}^{n}$ such that 
\[\emph{aff}(A)=x_{0}+ V,\] for some
$x_{0}\in \mathbb{R}^{n}.$ The dimension of \emph{aff}$(A)$ is 
$\dim(\emph{aff}(A))=\dim_{\mathbb{R}}(V).$

If $0\neq a\in \mathbb{R}^{n},$ then $H_{a}$ will denote the
hyperplane of $\mathbb{R}^{n}$ through the origin with normal vector
$a$, that is,\[H_{a}=\{x\in \mathbb{R}^{n}\ | \ \langle x,a\rangle
 =0\},\] where
$\langle ,\rangle$ is the usual inner product in $\mathbb{R}^{n}.$ The
 two closed
halfspaces bounded by $H_{a}$ are: \[H^{+}_{a}=\{x\in
\mathbb{R}^{n}\ | \ \langle x,a\rangle \geq 0\} \ \mbox{and} \
  H^{-}_{a}=\{x\in
\mathbb{R}^{n}\ | \ \langle x,a\rangle \leq 0\}.\]

Recall that a $polyhedral\ cone\ Q \subset \mathbb{R}^{n}$ is the
intersection of a finite number of closed subspaces of the form
$H^{+}_{a}.$ If $Q=H_{a_{1}}^{+}\cap \ldots \cap H_{a_{m}}^{+}$ is a polyhedral
cone, then \emph{aff}$(Q)$ is the intersection of those hyperplanes $H_{a_{i}}$ , 
$i= 1,\ldots,m$, that contain $Q$ (see {\cite[Proposition 1.2.] {BG}}).
The $dimension \ of \ Q$ is the dimension of 
\emph{aff}$(Q)$, $\dim(Q)=\dim(\emph{aff}(Q))$.\\
If $A=\{\gamma_{1},\ldots, \ \gamma_{r}\}$ is a
finite set of points in $\mathbb{R}^{n}$ the \emph{cone} generated by
$A$, denoted by $\mathbb{R}_{+}{A},$ 
respectively the \emph{convex hull} of $A$, denoted by \emph{conv}$(A)$,  are defined as
\[\mathbb{R}_{+}{A}=\{\sum_{i=1}^{r}a_{i}\gamma_{i}\ | \ a_{i}\in
\mathbb{R}_{+}\ \mbox{for all} \ 1\leq i \leq n\}\]
respectively
\[conv(A)=\{\sum_{i=1}^{r}a_{i}\gamma_{i}\ | \sum_{i=1}^{r} a_{i}=1, a_{i}\in
\mathbb{R}_{+}\ \mbox{for all} \ 1\leq i \leq n\},\]
where $\mathbb{R}_{+}$  denotes the set of nonnegative real numbers.
An important fact is that $Q$ is a polyhedral cone in
$\mathbb{R}^{n}$ if and only if there exists a finite set $A\subset
\mathbb{R}^{n}$ such that $Q=\mathbb{R}_{+}{A}$ (see {\cite{BG} or
\cite[Theorem 4.1.1.]{W}}).
If $U$ is  a $\mathbb{Q}-$vector subspace of $\mathbb{R}^{n}$ such that 
$\dim_{\mathbb{Q}}U=n$, then we say that a cone is \emph{rational} if is 
generated by a subset of $U$.

Next we give some important definitions and results (see \cite{B},
 \cite{BH}, \cite{BG}, \cite{MS}, \cite{V}).

\begin{definition}
A \emph{proper face} of a polyhedral cone $Q$ is a subset $F\subset \ Q$
 such
that there is a supporting hyperplane $H_{a}$ satisfying:

$1)$ $F=Q\cap H_{a}$,

$2)$ $Q\nsubseteq H_{a}$ and $Q\subset H^{+}_{a}$.
\end{definition}

The $dimension$ of a proper face $F$ of a polyhedral cone 
$Q$ is $\dim(F)=\dim(\emph{aff}(F)).$

\begin{definition}
A cone $C$ is  \emph{pointed} if $0$ is a face of $C.$ \ Equivalently
we can require that $x\in C$ \ and $-x\in C$ \ $\Rightarrow$ \ $x=0.$
\end{definition}

\begin{definition}
The 1-dimensional faces of a pointed cone are called $extremal \
rays.$
\end{definition}

\begin{definition}
A proper face $F$ of a polyhedral cone $Q\subset \ \mathbb{R}^{n}$
is called a $facet$ of $Q$ if $\dim(F)=\dim(Q)-1.$
\end{definition}

\begin{definition}
If a polyhedral cone $Q$ is written as \[Q=H^{+}_{a_{1}}\cap \ldots
\cap H^{+}_{a_{r}}\] such that no  $H^{+}_{a_{i}}$ can be
omitted, then we say that this is an \emph{irreducible representation}
 of
$Q.$
\end{definition}

\begin{theorem}
Let $Q\subset \mathbb{R}^{n},$ \ $Q\neq \mathbb{R}^{n}$, be a
polyhedral cone with $\dim(Q)=n.$ Then the halfspaces 
$H^{+}_{a_{1}}, \ldots , H^{+}_{a_{m}}$
in an irreducible representation 
$Q=H^{+}_{a_{1}}\cap \ldots \cap H^{+}_{a_{m}}$
are uniquely determined. In fact, the sets 
$F_{i}=Q \cap H_{a_{i}}, \  i = 1,\ldots, n,$ are the facets of
$Q$.
\end{theorem}
\begin{proof}
See {\cite[Theorem 1.6.] {BG}}
\end{proof}

The following two results are quite useful to determine the facets
of a polyhedral cone.

\begin{proposition}
Let $A$ be a finite set of points in $\mathbb{Z}^{n}$. If $F$ 
is a \emph{nonzero face} of $\mathbb{R}_{+}{A},$ then 
$F=\mathbb{R}_{+}{B}$ for some $B \subset A$.
\end{proposition}

\begin{proof}
Let $F=\mathbb{R}_{+}{A}\cap H_{a}$ with $\mathbb{R}_{+}{A}\subset 
H_{a}^{+}.$ Then $F$ is equal to the cone generated by the set 
$B=\{x \in A \ | \ \langle x,a\rangle = 0\}$.
\end{proof}

\begin{corollary}
Let $A$ be a finite set of points in $\mathbb{Z}^{n}$ and 
$F$ a face of $\mathbb{R}_{+}{A}$.\\
 \ $i)$ If $\dim \ F=1$ and $A\subset \mathbb{N}^{n},$ then 
$F=\mathbb{R}_{+}\alpha$ for some $\alpha \in A.$\\
$ii)$ If $\dim \ \mathbb{R}_{+}A=n$ and $F$ is a facet defined by 
the supporting hyperplane $H_{a}$, then $H_{a}$ is generated by
a linearly independent subset of $A.$
\end{corollary}

\begin{definition}
Let $Q$ be a polyhedral cone in $\mathbb{R}^{n}$ \ with $\dim \ Q=n$
 and such that $Q\neq\mathbb{R}^{n}.$  Let 
\[Q=H^{+}_{a_{1}}\cap\ldots \cap H^{+}_{a_{r}}
\] 
be the irreducible representation of
$Q.$  If $a_{i}=(a_{i1},\ldots,a_{in}),$  then we call
\[H_{a_{i}}(x):=a_{i1}x_{1}+\ldots + a_{in}x_{n}=0,\ \ i\in [r],
\]
the \emph{equations of the cone} $Q.$
\end{definition}

\begin{definition}
The \emph{relative interior} $ri(Q)$ of a polyhedral cone is the 
interior of $Q$ with respect to the embedding of $Q$ into its affine
space \emph{aff}($Q$), in which $Q$ is full-dimensional.
\end{definition}

The following result gives us the description of the relative interior
of a polyhedral cone when we know its irreducible representation.

\begin{theorem}
Let $Q\subset \mathbb{R}^{n},$ \ $Q\neq \mathbb{R}^{n}$, be a
 polyhedral cone with $\dim(Q)=n$  and let 
\[(*) \ Q=H^{+}_{a_{1}}\cap \ldots \cap H^{+}_{a_{m}}
\]  
be an irreducible representation of $Q$  with $H^{+}_{a_{1}}, \ldots ,
 H^{+}_{a_{n}}$ pairwise distinct, 
where $a_{i} \in \mathbb{R}^{n} \setminus \{0\}$  for all $i.$  Set
 $F_{i}=Q \cap H_{a_{i}}$ \ for $i \in [r]$.  Then:\\
$a)$ \ $ri(Q)=\{x \in \mathbb{R}^{n} \ | \ \langle x, a_{1}\rangle \ >
 0, \ldots ,\langle x, a_{r}\rangle \ > 0 \}$,  
where $ri(Q)$ \ is the relative interior of $Q,$ \ which in this case
 is just the interior.\\
$b)$ \ Each facet $F$ \ of $Q$ \ is of the form $F=F_{i}$ \ for some
 $i.$\\
$c)$ \ Each $F_{i}$ \ is a facet of $Q$.
\end{theorem}
\begin{proof}
See  \cite[Theorem  8.2.15]{B} and \cite[Theorem 3.2.1]{W}.
\end{proof}

\section{Affine semigroup rings.}
\[\]

An \emph{affine semigroup} $C$ is a finitely generated additive semigroup
which for some $n\in \mathbb{N}$ is isomorphic to a subsemigroup of 
$\mathbb{Z}^{n}$ containing 0. For instance, $(\mathbb{N}\cup \{0\})^{n}$ is
an affine semigroup for all $n\in \mathbb{N}.$ Just like $\mathbb{Z}$ being
generated by $\mathbb{N}\cup \{0\}$ as a group, for every affine semigroup 
$C$ there exists a unique up to a canonical isomorphism finitely generated 
abelian group, denoted $\mathbb{Z}C$, that contains $C$ and is generated by
$C$ as a group. Its rank is the \emph{dimension} of $C$. 
Moreover, $\mathbb{R}C$ denotes $\mathbb{Z}C \otimes_{\mathbb{Z}} \mathbb{R},$
and one considers $C$ to be a subset of $\mathbb{R}C$ via the canonical map
$\mathbb{Z}C \rightarrow \mathbb{Z}C \otimes_{\mathbb{Z}} \mathbb{R}.$ $C$ 
is said to be \emph{positive}, if 0 is the only invertible element in $C$.

Let $C$ be an affine semigroup, and let $K$ be a field. The $K-$ vector space
\[K[C]:= \bigoplus_{a \in C}Kx^{a}\]
becomes a $K-$ algebra by setting $x^{a}\cdot x^{b}:=x^{a+b}$ for all $a, b \in C.$
It is the \emph{affine semigroup ring} associated to $C$ over $K$. We say 
that $K[C]$ is \emph{positive} affine semigroup ring in case $C$ is positive.
Since $C$ is a finitely generated semigroup, $K[C]$ is a finitely generated
$K-$ algebra and thus Noetherian. Since $K[C]$ is a subring of $K[\mathbb{Z}C]$ 
and $K[\mathbb{Z}C]$ is isomorphic to the ring $K[x_{1}^{\pm},\ldots,x_{d}^{\pm}]$
of Laurent polynomials, where $d$ is the dimension of $C$, one obtains that
$K[C]$ is an integral domain.

Assume that $C$ is positive. A decomposition $R=\bigoplus_{n\geq 0}R_{n}$ of
$R=K[C]$ is called an \emph{admissible grading}, if the following conditions
are fulfilled:
\begin{enumerate}
 \item[a)] $R_{0}=K$.
 \item[b)] For all $n\geq 0$, $R_{n}$ is a $K-$ vector space that is generated by 
finitely many elements of the form $x^{a}$, $a\in C.$
 \item[c)] For all $m, n \geq 0$, $R_{m}\cdot R_{n}\subset R_{m+n}$.
\end{enumerate}
The existence of an admissible grading is guaranteed by the following result.
\begin{proposition}
If $C$ is a positive affine semigroup, then there exist $r\in \mathbb{N}$
and a group homomorphism $\phi \ : \mathbb{Z}C \rightarrow \mathbb{Z}^{r}$,
such that $\phi(C)\subseteq (\mathbb{N}\cup \{0\})^{r}$.
\end{proposition}
\begin{proof}
See \cite[Proposition 6.1.5]{BH}.
\end{proof}

An affine semigroup $C$ is called \emph{normal} if it satisfies the 
following condition: if $mz \in C$ for some $z \in \mathbb{Z}C$ and 
$m\in \mathbb{N}$, then $z\in C$. 
We have the following important results:
\begin{proposition}$(Gordan's \ lemma)$\\
$a)$ If $C$ is a normal semigroup, then $C=\mathbb{Z}C\cap \mathbb{R}_{+}C$.\\
$b)$ Let $G$ be a finitely generated subgroup of $\mathbb{Q}^{n}$ and $D$ a
finitely generated rational cone in $\mathbb{R}^{n}$. Then $C=G\cap D$ is a
normal semigroup.
\end{proposition}

\begin{theorem}
Let $C$ be an affine semigroup, and let $K[C]$ be the associated affine
semigroup ring over a field $K$. Then $C$ is normal if and only if 
$K[C]$ is normal.
\end{theorem}
\begin{proof}
One sees immediately that $C$ must be normal if $K[C]$ is a normal domain:
if $x^{z}$ is an element of the field of fractions of $K[C]$ and if 
$(x^{z})^{m}\in K[C]$ and $K[C]$ is normal, then $x^{z}\in K[C].$ 
For the proof of the converse, see \cite[Theorem 6.1.4]{BH}.
\end{proof}
The following famous theorem is due to Hochster, 
see \cite[Theorem 6.3.5]{BH} for a proof.
\begin{theorem}
Let $R=K[C]$ be an affine semigroup ring. If $R$ is normal, then it is
Cohen-Macaulay.
\end{theorem}

Let $R=K[x]=K[x_{1},\ldots,x_{n}]$ be a polynomial ring over a field $K$
in the indeterminates $x_{1},\ldots,x_{n}$ and $F=\{f_{1},\ldots,f_{q}\}$
a finite set of distinct monomials in $R$ such that $f_{i}\neq 1$ for all $i$.
For $c\in \mathbb{N}^{n}$ we set 
$x^{c}=x_{1}^{c_{1}}\cdot\cdot\cdot x_{n}^{c_{n}}.$ The \emph{monomial subring}
spanned by $F$ is the $K-$ subalgebra \[K[F]=K[f_{1},\ldots,f_{q}]\subset R.\]
The exponent vector of $f_{i}=x^{\alpha_{i}}$ is denoted by log($f_{i}$)=
$\alpha_{i}$ and log($F$) denotes the set of exponent vectors of the monomials
in $F$.

Note that $K[F]$ is equal to the affine semigroup ring 
\[K[C]=K[\{x^{\alpha} \ | \ \alpha \in C\}],\] where 
$C=\mathbb{N}\log(f_{1})+\ldots+\mathbb{N}\log(f_{q})$ is the subsemigroup of
$\mathbb{N}^{n}$ generated by log($F$). Thus as $K-$ vector space $K[F]$ is 
generated by the set of monomials of the form $x^{\alpha}$, with 
$\alpha \in C$. An important feature of $K[F]$  is that it is a graded subring
of $R$ with the standard grading $K[F]_{i}=K[F]\cap R_{i}$.

Next we give an important result of Danilov and Stanley which characterizes the
canonical module in terms of the relative interior of a cone.

\begin{theorem}{\bf(Danilov, Stanley)}
Let $R=K[x_{1},\ldots,x_{n}]$ be a polynomial ring over a field $K$ and
 $F$ a finite set of monomials in $R.$
If $K[F]$ is normal, then the canonical module $\omega_{K[F]}$ of
 $K[F],$ with respect to
standard grading, can be expressed as an ideal of $K[F]$ generated by
 monomials
\[\omega_{K[F]}=(\{x^{a}| \ a\in {\mathbb{N}}A\cap
 ri({\mathbb{R_{+}}}A)\}),\] where $A=\log (F)$ and
$ri({\mathbb{R_{+}}}A)$ denotes the relative interior of
${\mathbb{R_{+}}}A.$
\end{theorem}

The formula above represents the canonical module of $K[F]$ \ as an
 ideal of $K[F]$ \ generated by monomials. For a comprehensive treatment of
 the Danilov-Stanley formula see \cite{BH},  \cite{MS} or \cite{V} . \\

\section{Discrete polymatroids.}
\[\]

Matroid theory is one of the most fascinating research areas in
combinatorics. The discrete polymatroid is a multiset analogue of the matroid. 
Based on the polyhedral theory on integral polymatroids developed in late
1960's and early 1970's, in the present section the combinatorics and algebra
on discrete polymatroids will be studied.

Let $[n]=\{1,2,\ldots, n\}$ and $2^{[n]}$ the set
of all subsets of $[n]$. For a subset $A\subset [n]$ write $| A |$ for
the cardinality of $A$. The following definition of the \emph{matroid}
is originated in Whitney (1935).
\begin{definition}
A \emph{matroid} on the ground set $[n]$ 
is a nonempty subset ${\bf{\mathcal{M}}}\subset 2^{[n]}$ satisfying:
\begin{enumerate}
\item[($M_{1}$)] if $F_{1}\in \bf{\mathcal{M}}$ and 
$F_{2}\subset F_{1}$, then $F_{2}\in \bf{\mathcal{M}}$;
\item[($M_{2}$)] if $F_{1}, F_{2} \in \bf{\mathcal{M}}$ and 
$| F_{1} | < | F_{2} |$, then there is $x \in F_{2}\setminus F_{1}$
such that $F_{1}\cup \{x\}\in \bf{\mathcal{M}}$.
\end{enumerate}
\end{definition}

The members of ${\bf{\mathcal{M}}}$ are the \emph{independent sets}
of ${\bf{\mathcal{M}}}$. A \emph{base} of ${\bf{\mathcal{M}}}$ 
is a maximal independent set of ${\bf{\mathcal{M}}}$. It follows from 
$(M_{2})$ that if $B_{1}$ and $B_{2}$ are bases of ${\bf{\mathcal{M}}}$,
then $| \ B_{1} \ |=| \ B_{2} \ |$. The set of bases 
of ${\bf{\mathcal{M}}}$ possesses the $''exchange \ property''$ following:
\begin{enumerate}
\item[($E$)] If $B_{1}$ and $B_{2}$ are bases of ${\bf{\mathcal{M}}}$ and if
$x\in B_{1}\setminus B_{2}$, then there is $y\in B_{2}\setminus B_{1}$
such that $(B_{1}\setminus \{x\})\cup \{y\}$ is a base of ${\bf{\mathcal{M}}}$.
\end{enumerate}

Moreover, the set of bases of ${\bf{\mathcal{M}}}$ possesses the following
$''symmetric \ \ exchange \ property''$:

\begin{enumerate}
\item[($SE$)] If $B_{1}$ and $B_{2}$ are bases of ${\bf{\mathcal{M}}}$ and if
$x\in B_{1}\setminus B_{2}$, then there is $y\in B_{2}\setminus B_{1}$
such that both $(B_{1}\setminus \{x\})\cup \{y\}$ and $(B_{2}\setminus \{y\})\cup \{x\}$ 
are bases of ${\bf{\mathcal{M}}}$.
\end{enumerate}

Alternatively, we can give another definition of matroid in terms of 
its set of bases. Given a nonempty set ${\bf{\mathcal{B}}}\subset 2^{[n]}$, 
there exists a matroid ${\bf{\mathcal{M}}}$ on the ground set $[n]$ 
with ${\bf{\mathcal{B}}}$ its set of bases if and only if ${\bf{\mathcal{B}}}$
possesses the exchange property $(E)$. If we denote the canonical basis
vectors of $\mathbb{R}^{n}$ by $e_{1},e_{2},\ldots, e_{n}$, then a matroid on $[n]$
can be regarded as a set of $(0,1)-$vectors $\sum_{k\in F}e_{k}$ 
for each $F\subset [n]$. Now we give three important examples of matroids.
\begin{example}
{\bf{Vector Matroid:}}
Let $V$ be a vector space and $E$ be a nonempty finite subset of $V$. We define
the matroid ${\bf{\mathcal{M}}}$ on the ground set $E$ by taking the 
independent sets of ${\bf{\mathcal{M}}}$ to be the sets of linearly
independent elements in $E$. With linear algebra arguments one can check
that the axioms of the matroid are fulfilled.

{\bf{Cycle Matroid:}} Let $G$ be a finite graph, with $V$ its set of vertices
and $E$ its set of edges.  Consider a set of edges independent if and only if
it does not contain a simple cycle. Then the set of all these independent 
sets defines a matroid on the ground set $E$.

{\bf{Uniform Matroid:}} Let $r$ and $n$ be nonnegative integers with $r$ no 
larger than $n$. Let $E$ be a set of cardinality $n$ and let 
${\bf{\mathcal{M}}}$ be the collection of all subsets of $E$ of cardinality
$r$ or less. Then ${\bf{\mathcal{M}}}$ is a matroid, called the uniform matroid
of rank $r$ on $n$ elements, and it is denoted by $U_{r, n}$.
\end{example}

Let $e_{1},e_{2},\ldots, e_{n}$ denote the canonical basis vectors of
$\mathbb{R}^{n}$ and $\mathbb{R}_{+}^{n}$ denote the set of vectors
$u=(u_{1},\ldots,u_{n})\in \mathbb{R}^{n}$ with each $u_{i}\geq 0$. Also, let 
$\mathbb{Z}_{+}^{n}=\mathbb{R}_{+}^{n} \cap \mathbb{Z}^{n}$. 
If $u=(u_{1},\ldots,u_{n})$ and $v=(v_{1},\ldots,v_{n})$ are two vectors
belonging to $\mathbb{R}_{+}^{n},$ then we write $u\leq v$ if all 
components $v_{i}-u_{i}$ of $v-u$ are nonnegative,  and write $u<v$ 
if $u\leq v$ and $u\neq v$. We say that $u$ is a \emph{subvector} of
$v$ if $u\leq v$. In addition, 
we set \[u\vee v=(\max\{u_{1},v_{1}\},\ldots,\max\{u_{n},v_{n}\}),\] \[u\wedge v=(\min\{u_{1},v_{1}\},\ldots,\min\{u_{n},v_{n}\}).\]
Thus, $u\wedge v \leq u \leq u\vee v$ and $u\wedge v \leq v \leq u\vee v$.
The \emph{modulus} of 
$u=(u_{1},\ldots,u_{n}) \in \mathbb{R}_{+}^{n}$ is 
$| \ u \ |=u_{1}+\ldots+u_{n}$ and for a subset $F\subset [n]$, we set
\[u(F)=\sum_{k\in F}u_{k}.\]

Next we present the concept of \emph{polymatroid} and its associated 
\emph{rank function}. The concept of \emph{polymatroid} originated in
Edmonds (\cite{E}), and for further properties the reader can consult 
(\cite{O}, \cite{W1}).
\begin{definition}
A \emph{polymatroid} on the ground set $[n]$ is a nonempty compact
subset $\mathcal{P}\subset \mathbb{R}_{+}^{n}$, the set of independent 
vectors, such that
\begin{enumerate}
\item[($\mathcal{P}_{1}$)] every subvector of an independent vector is independent;
\item[($\mathcal{P}_{2}$)] if $u, v\in \mathcal{P}$ with $| \ v \ |>| \ u \ |$, then
there is a vector $w\in \mathcal{P}$ such that \[u<w\leq u\vee v.\]
\end{enumerate}
\end{definition}

A \emph{base} of a polymatroid $\mathcal{P}\subset \mathbb{R}_{+}^{n}$ is
a maximal independent vector of $\mathcal{P}$, i.e. an independent
vector $u\in \mathcal{P}$ with $u<v$ for no $v\in \mathcal{P}$. It follows from 
($\mathcal{P}_{2}$) that every base of $\mathcal{P}$ has the same
modulus rank($\mathcal{P}$), the \emph{rank} of $\mathcal{P}$.

Now we give an equivalent description of a polymatroid. 
Let $\mathcal{P}\subset \mathbb{R}_{+}^{n}$ be a polymatroid on the ground
set $[n]$. The \emph{ground set rank function} of $\mathcal{P}$ is a function
$\rho : 2^{[n]}\rightarrow \mathbb{R}_{+}$ defined by setting
\[\rho(F)=\max\{v(F)\ : \ v\in \mathcal{P}\}\] for all nonempty $F\in [n]$ 
together with $\rho(\emptyset)=0$. Then we have
\begin{proposition}$(\cite{W1})$
$a)$ Let $\mathcal{P}\subset \mathbb{R}_{+}^{n}$ be a polymatroid
on the ground set $[n]$ and $\rho$ its ground set rank function. Then $\rho$
is nondecreasing, i.e., if $F_{1}\subset F_{2}\subset [n]$, then 
$\rho(F_{1})\leq \rho(F_{2})$, and is submodular, i.e., 
\[\rho(F_{1})+\rho(F_{2})\geq \rho(F_{1}\cup F_{2})+ \rho(F_{1}\cap F_{2})\]
for all $F_{1}, F_{2}\subset [n]$. Moreover, $\mathcal{P}$ coincides 
with the compact set
\[\{x\in \mathbb{R}_{+}^{n} \ | \ x(A)\leq \rho(A), A\subset [n]\}.\]
$b)$ Conversely, given a nondecresing and submodular function 
$\rho : 2^{[n]}\rightarrow \mathbb{R}_{+}$ with $\rho(\emptyset)=0$, 
the compact set $\{x\in \mathbb{R}_{+}^{n} \ | \ x(A)\leq \rho(A), A\subset [n]\}$ 
is a polymatroid on the ground set $[n]$  with $\rho$ its ground set rank function.
\end{proposition}

From Proposition 1.3.4.$a)$ it follows that a polymatroid 
$\mathcal{P}\subset \mathbb{R}_{+}^{n}$ on the ground set $[n]$ 
is a convex polytope in $\mathbb{R}_{+}^{n}$. A polymatroid is \emph{integral}
if and only if its ground set rank function is integer valued. For a 
detailed material on convex polytopes see \cite{H}, \cite{Z}.

Now we introduce \emph{discrete polymatroids}. They may be viewed as 
generalizations of matroids. 
\begin{definition}$(\cite{HH})$
Let $\mathcal{P}$ be a nonempty finite set of integer vectors in 
$\mathbb{R}_{+}^{n}$, which contains with each $u\in \mathcal{P}$
all its integral subvectors. The set $\mathcal{P}$ is called 
\emph{discrete polymatroid} on the ground set $[n]$ if for all 
$u, v\in \mathcal{P}$ with $| \ v \ |>| \ u \ |$, 
there is a vector $w\in \mathcal{P}$ such that \[u<w\leq u\vee v.\]
\end{definition}

A \emph{base} of ${\bf{\mathcal{P}}}$ is a vector 
$u\in {\bf{\mathcal{P}}}$ such that $u<v$ for no 
$v\in {\bf{\mathcal{P}}}$. We denote by $B({\bf{\mathcal{P}}})$ the set of bases
of a discrete polymatroid ${\bf{\mathcal{P}}}$. It follows from the definition
that any two bases of ${\bf{\mathcal{P}}}$ have the same modulus. This common
number is called the \emph{rank} of ${\bf{\mathcal{P}}}$.

Discrete polymatroids can be  characterized in terms of their set of bases
as follows.
\begin{theorem}$(\cite{HH})$
Let $\mathcal{P}$ be a nonempty finite set of integer vectors in 
$\mathbb{R}_{+}^{n}$, which contains with each $u\in {\bf{\mathcal{P}}}$
all its integral subvectors, and let $B({\bf{\mathcal{P}}})$ 
be the set of vectors $u\in {\bf{\mathcal{P}}}$ with $u<v$ for no 
$v\in {\bf{\mathcal{P}}}$. The following conditions are equivalent:\\
$a)$ ${\bf{\mathcal{P}}}$ is a discrete polymatroid;\\
$b)$ if $u, v\in \mathcal{P}$ with $| \ v \ |>| \ u \ |$, 
there is an integer $i$ such that $u+e_{i}\in \mathcal{P}$ and $u+e_{i}\leq u\vee v$;\\
$c)$ \begin{enumerate}
\item[$i)$] all $u \in B({\bf{\mathcal{P}}})$ have the same modulus,
\item[$ii)$] if $u, v\in B({\bf{\mathcal{P}}})$ with $u_{i}>v_{i}$ 
there is $j\in [n]$ with $u_{j}<v_{j}$ such that 
$u-e_{i}+e_{j}\in B({\bf{\mathcal{P}}})$.
\end{enumerate}
\end{theorem}

Condition $c).ii)$ from the theorem is also called the \emph{exchange property}.
An important consequence of this theorem is that it gives a way to construct
discrete polymatroids. According to condition $c)$, it is enough to give a set 
of integer vector of the same modulus, which satisfy the \emph{exchange property}
and then, by taking all its integral subvectors,  we obtain a discrete polymatroid.

The following result, which is obtained from Theorem 1.3.6. and the definition of
matroid, shows that it makes sense to view the discrete polymatroids 
as generalizations of matroids.
\begin{corollary}
$(\cite{HH})$ Let $B$ be a nonempty finite set of integer vectors in 
$\mathbb{R}_{+}^{n}$. The following conditions are equivalent:
\begin{enumerate}
\item[$i)$] $B$  is the set of bases of a matroid;
\item[$ii)$] $B$ is the set of bases of a discrete polymatroid, and for all 
$u\in B$ one has $u_{k}\leq 1$ for $k\in [n]$.
\end{enumerate}
\end{corollary}

Just as in the case of matroids, we have the \emph{symmetric exchange property}:
\begin{theorem}$(\cite{HH})$
If $u=(u_{1},\ldots,u_{n})$ and $v=(v_{1},\ldots,v_{n})$ are bases of a discrete
polymatroid $\mathcal{P}\subset \mathbb{Z}_{+}^{n}$, then for each $i\in [n]$
with $u_{i}>v_{i}$ there is $j\in [n]$ with $u_{j}<v_{j}$ such that both 
$u-e_{i}+e_{j}$ and $u-e_{j}+e_{i}$ are bases of ${\bf{\mathcal{P}}}$.
\end{theorem} 

In the proof it is used the \emph{rank function} of a discrete polymatroid, which
we define next. Let $\mathcal{P}\subset \mathbb{Z}_{+}^{n}$ be a discrete polymatroid
and $B(\mathcal{P})$ its set of bases. We define the \emph{rank function} of 
the discrete polymatroid $\mathcal{P}$ to be function 
$\rho_{\mathcal{P}} : 2^{[n]}\rightarrow \mathbb{Z}_{+}$, by setting
\[\rho_{\mathcal{P}}(F)=\max\{u(F) \ | \ F\in \mathcal{P}\}\] for all 
$\emptyset \neq F \subset [n]$, together with $\rho_{\mathcal{P}}(\emptyset)=0$. 
It is clear that $\rho_{\mathcal{P}}$ is a nondecreasing function and from (\cite{HH})
we have that $\rho_{\mathcal{P}}$ is submodular. Conversely, given a nondecreasing 
and submodular function $\rho : 2^{[n]}\rightarrow \mathbb{Z}_{+}$, 
the set of $u\in \mathbb{Z}_{+}^{n}$ satisfying 
\[u(F)\leq \rho(F), \ \rm{for} \ \rm{all} \  \it{F} \in 2^{[n]}, \  (*)\] 
is a discrete polymatroid, whose
rank function is $\rho_{\mathcal{P}}=\rho$. In connection to the rank function $\rho$
of a discrete polymatroid $\mathcal{P}$ we distinguish two important types of sets.
A set $\emptyset \neq F \subset [n]$ is \emph{$\rho$-closed} if any subset 
$G\subset [n]$ properly containing $F$ satisfies $\rho(F)<\rho(G)$, and
$\emptyset\neq F\subset [n]$ is \emph{$\rho$-separable} if there exist 
two nonempty subsets $F_{1}$ and $F_{2}$ with $F_{1}\cap F_{2}=\emptyset$ and 
$F_{1}\cup F_{2}=F$ such that $\rho(F)=\rho(F_{1})+\rho(F_{2}).$ A nonempty subset
$F$ in $[n]$ is \emph{$\rho$-inseparable} if it is not \emph{$\rho$-separable}. 
The following example is intended to give a better view to the construction
$(*)$ and the definition above.
\begin{example}$(\cite{V1})$
Let us consider the function 
$\rho_{\mathcal{P}} : 2^{[3]}\rightarrow \mathbb{Z}_{+}$ defined by 
$\rho(\emptyset)=0, \ \rho(\{1\})=1, \ \rho(\{2\})=2, \ \rho(\{3\})=2, \ 
\rho(\{1,2\})=3, \ \rho(\{1, 3\})=2, \ \rho(\{2, 3\})=4, \ \rho(\{1, 2, 3\})=4.$ One
can easily check that $\rho$ is nondecreasing and submodular. The bases
are the integer solutions $(u_{1}, u_{2}, u_{3})$ of the inequations
\[u_{1}\leq 1, u_{2}\leq 2, u_{3}\leq 2, u_{1}+u_{2}\leq 3, u_{1}+u_{3}\leq 2, 
u_{2}+u_{3}\leq 4,\]
together with \[u_{1}+u_{2}+u_{3}=4,\]
i.e., the vectors (1, 2, 1) and (0, 2, 2). Taking all subintegral vectors of (1, 2, 1)
and (0, 2, 2) we obtain the discrete polymatroid $\mathcal{P}$
\[\mathcal{P}=\{(0, 0, 0), (1, 0, 0), (0, 1, 0), (0, 0, 1), (1, 1, 0), (1, 0, 1), 
(0, 1, 1), (0, 2, 0), (0, 0, 2), (0, 1, 2),\]\[ (0, 2, 1), (1, 1, 1), (1, 2, 0), (0, 2, 2), (1, 2, 1)\}.\]
The $\rho-closed$ subsets of $[3]$ are: $\{1\}, \{2\}, \{1, 2\} \ \rm{and} \ \{1, 3\}.$
The $\rho-inseparable$ subsets of $[3]$ are: $\{1\}, \{2\}, \{3\} \ \rm{and} \ \{1, 3\}.$
\end{example}

The following result makes the connection between discrete polymatroids 
and integral polymatroids.
\begin{theorem}$(\cite{HH})$
A nonempty finite set $\mathcal{P}\subset \mathbb{Z}_{+}^{n}$ is a discrete polymatroid 
if and only if $conv(\mathcal{P})\subset \mathbb{R}_{+}^{n}$ is an integral 
polymatroid with $conv(\mathcal{P})\cap \mathbb{Z}_{+}^{n}=\mathcal{P}$.
\end{theorem}

Now we present two techniques from \cite{HH} to construct discrete polymatroids.
The first one shows that a nondecreasing and submodular function defined 
on a sublattice of $2^{[n]}$ produces a discrete polymatroid. A sublattice
of $2^{[n]}$ is a collection $\mathcal{L}$ of subsets of $[n]$ with
$\emptyset\in \mathcal{L}$ and $[n]\in \mathcal{L}$ such that for all
$F, G \in \mathcal{L}$ both $F\cap G \ \rm{and} \ F\cup G$ belong to $\mathcal{L}$.
\[\]

\begin{theorem}$(\cite{HH})$
Let $\mathcal{L}$ be a sublattice of $2^{[n]}$ and 
$\mu : \mathcal{L}\rightarrow \mathbb{R}_{+}$ an integer valued nondecreasing and
submodular function with $\mu(\emptyset)=0$. Then 
\[\mathcal{P}_{(\mathcal{L}, \mu)}=\{u\in \mathbb{Z}_{+}^{n} \ | \ u(F)\leq \mu(F), F\in \mathcal{L}\}\]
is a discrete polymatroid.
\end{theorem}
\begin{example}$(\cite{HH})$
Let $\mathcal{L}$ be a chain of length $n$ of $2^{[n]}$, say
\[\mathcal{L}=\{\emptyset, \{n\}, \{n-1, n\}, \ldots, \{1,\ldots,n\}\}\subset 2^{[n]}.\]
Given nonnegative integers $a_{1},\ldots, a_{n},$ define 
$\mu : \mathcal{L}\rightarrow \mathbb{R}_{+}$ by
\[\mu(\{i, i+1, \ldots, n\})=a_{i}+a_{i+1}+\ldots +a_{n}, \ \ 1\leq i \leq n,\] 
together with $\mu(\emptyset)=0$. Then the discrete polymatroid 
$\mathcal{P}_{(\mathcal{L}, \mu)}\subset \mathbb{Z}_{+}^{n}$ is 
\[\mathcal{P}_{(\mathcal{L}, \mu)}=\{u\in \mathbb{Z}_{+}^{n} \ | \ 
\sum_{k=i}^{n}u_{k}\leq \sum_{k=i}^{n}a_{k}, \ \ 1\leq k \leq n\}.\]
\end{example}

For the second result about construction of discrete polymatroids, first
we need to fix some notation. Let ${\bf{\mathcal{A}}}=\{A_{1},\ldots,A_{d}\}$
be a family of nonempty subsets of $[n]$. It is not required that 
$A_{i}\neq A_{j}$ if $i\neq j$. Let 
\[{\bf{\mathcal{B}}}_{{\bf{\mathcal{A}}}}=\{e_{i_{1}}+\ldots +e_{i_{d}} \ | \
 i_{k}\in A_{k} , \ 1\leq k \leq d\}\subset \mathbb{Z}_{+}^{n}\]
and define the integer valued nondecreasing function 
$\rho_{{\bf{\mathcal{A}}}} : 2^{[n]}\rightarrow \mathbb{R}_{+}$ by setting
\[\rho_{{\bf{\mathcal{A}}}}(X)=| \{k \ | \ A_{k}\cap X\neq \emptyset\} 
|, \ X\subset [n].\]

Now we can state
\begin{theorem}$(\cite{HH})$
The function $\rho_{{\bf{\mathcal{A}}}}$ is submodular and 
${\bf{\mathcal{B}}}_{{\bf{\mathcal{A}}}}$ is the set of bases of the discrete
polymatroid ${\bf{\mathcal{P}}}_{{\bf{\mathcal{A}}}}\subset \mathbb{Z}_{+}^{n}$
arising from $\rho_{{\bf{\mathcal{A}}}}$.
\end{theorem}

The discrete polymatroid ${\bf{\mathcal{P}}}_{{\bf{\mathcal{A}}}}\subset \mathbb{Z}_{+}^{n}$
is called the \emph{transversal polymatroid} presented by ${\bf{\mathcal{A}}}$. The \emph{rank} of
${\bf{\mathcal{P}}}_{{\bf{\mathcal{A}}}}$ is $rank({\bf{\mathcal{P}}}_{{\bf{\mathcal{A}}}})=d$.
The following examples, given by Herzog and Hibi, show that 
the discrete polymatroid considered in Example 1.3.12. 
is a transversal polymatroid and that not all discrete polymatroids are transversal.
\begin{example}$(\cite{HH})$
$a)$ Let $r_{1},\ldots,r_{d}\in [n]$ and set $A_{k}=[r_{k}]$, $1\leq k \leq d$. Let 
$\min(X)$ denote the smallest integer belonging to $X$, where 
$\emptyset \neq X \subset [n]$. If ${\bf{\mathcal{A}}}=\{A_{1},\ldots,A_{d}\}$, then
\[\rho_{{\bf{\mathcal{A}}}}(X)=\rho_{{\bf{\mathcal{A}}}}(\{\min(X)\})=
|\{k \ | \ \min(X)\leq r_{k}\}|.\] 
If $\emptyset \neq X \subset [n]$ is $\rho_{{\bf{\mathcal{A}}}}-\rm{closed}$, then
$X=\{\min(X), \min(X)+1,\ldots, n)\}$. Let \[a_{i}=|\{k \ | \ r_{k}=i\}|, \ 
1\leq i\leq n. \] Thus \[\rho_{{\bf{\mathcal{A}}}}(\{i, i+1, \ldots, n\})=
a_{i}+a_{i+1}+\ldots+a_{n}, \ 1\leq i \leq n.\]
Then the discrete polymatroid 
$\mathcal{P}_{{\bf{\mathcal{A}}}}\subset \mathbb{Z}_{+}^{n}$ is 
\[\mathcal{P}_{{\bf{\mathcal{A}}}}=\{u\in \mathbb{Z}_{+}^{n} \ | \ 
\sum_{k=i}^{n}u_{k}\leq \sum_{k=i}^{n}a_{k}, \ \ 1\leq k \leq n\}.\]
Thus $\mathcal{P}_{{\bf{\mathcal{A}}}}$ coincides with the discrete
polymatroid $\mathcal{P}_{(\mathcal{L}, \mu)}$ in Example 1.3.12.

$b)$ Let $\mathcal{P}\subset \mathbb{Z}_{+}^{4}$ denote the discrete polymatroid
of rank 3 consisting of those $u\in \mathbb{Z}_{+}^{4}$ with $u_{i}\leq 2$ for
$1\leq i\leq 4$ and with $|u|\leq 3$. Then $\mathcal{P}$ is not transversal. Suppose, 
on the contrary, that $\mathcal{P}$ is a transversal polymatroid presented by 
${\bf{\mathcal{A}}}=\{A_{1},A_{2},A_{3}\}$ with each $A_{k}\subset [4]$. Since
$(2, 1, 0, 0), (2, 0, 1, 0), (2, 0, 0, 1) \in \mathcal{P}$ and 
$(3, 0, 0, 0)\notin \mathcal{P}$ we may assume that $1\in A_{1}, 1\in A_{2}$ and 
$A_{3}=\{2, 3, 4\}$. Since $(1, 2, 0, 0), (0, 2, 1, 0), (0, 2, 0, 1)\in \mathcal{P}$ 
and $(0, 3, 0, 0)\notin \mathcal{P}$ we may assume that $2\in A_{1}$ and $A_{2}=\{1, 3, 4\}$.
Since $(0, 0, 2, 1)\in \mathcal{P}$ and $(0, 0, 3, 0)\notin \mathcal{P}$, one has
$4\in A_{1}$. Hence $(0, 0, 0, 3)\in \mathcal{P}$, a contradiction.
\end{example}

\section{The Ehrhart ring and the base ring of a discrete polymatroid.}
\[\]

Let $K$ be a field and let $x_{1},\ldots,x_{n}$ and $s$ be 
indeterminates over $K$. If $u=(u_{1},\ldots,u_{n})\in \mathbb{Z}_{+}^{n}$, 
then we denote by $x^{u}$ the monomial $x_{1}^{u_{1}}\cdot\cdot\cdot x_{n}^{u_{n}}$.
Let $P$ be a discrete polymatroid of rank $d$ on the ground set $[n]$ with set of
bases $B$. The toric ring $K[B]$ generated over $K$ by the monomials $x^{u}$, where
$u\in B$, is called the \emph{base ring} of $P$. Since $P$ is the set of integer 
vectors of an integral polymatroid $\mathcal{P}$ (see Theorem 1.3.10), we may study the
\emph{Ehrhart ring of $\mathcal{P}$}. For this, one considers the cone 
$\mathcal{C}\subset \mathbb{R}^{n+1}$ with $\mathcal{C}=
\mathbb{R}_{+}\{(p, 1) \ | \ p \in \mathcal{P}\}$. Then $Q=\mathcal{C}\cap \mathbb{Z}^{n+1}$ 
is a subsemigroup of $\mathbb{Z}^{n+1}$, and the Ehrhart ring of $\mathcal{P}$ is defined
to be the toric ring $K[\mathcal{P}]\subset K[x_{1},\ldots,x_{n}, s ]$ generated over
$K$ by the monomials $x^{u}s^{i}$, $(u, i)\in Q$. By Gordan's Lemma 
(\cite{BH}, Proposition 6.1.2), $K[\mathcal{P}]$ is normal. Notice that $K[\mathcal{P}]$
is naturally graded if we assign to $x^{u}s^{i}$ the degree $i$. We denote by $K[P]$ the
$K-$subalgebra of $K[\mathcal{P}]$ which is generated over $K$ by the elements of degree 1
in $K[\mathcal{P}]$. Since $P=\mathcal{P}\cap \mathbb{Z}^{n}$ it follows that 
$K[P]=K[x^{u}s \ | \ u\in P]$. Observe that $K[B]$ may be identified with the subalgebra
$K[x^{u}s \ | \ u\in B]$ of $K[P]$.

The base ring $K[B]$ was introduced in 1977 by N. White, in the particular case when
$B$ is the set of bases of a matroid, and he showed that for every matroid, the ring
$K[B]$ is normal (see \cite{W3} ) and thus Cohen-Macaulay. It is natural to ask 
whether the same holds for the base ring of any discrete polymatroid. Herzog and Hibi
showed this.
\begin{theorem}$(\cite{HH})$
$K[P]=K[\mathcal{P}]$. In particular, $K[P]$ is normal.
\end{theorem}
As corollary they  also obtain
\begin{corollary}$(\cite{HH})$
$K[B]$ is normal.
\end{corollary}

It is natural to ask, since both $K[P]$ and $K[B]$ are Cohen-Macaulay, when these 
rings are Gorenstein. Next we give some suggestive examples.
\begin{example}$(\cite{HH})$
$a)$ Let $P\subset \mathbb{Z}_{+}^{3}$ be the discrete polymatroid consisting of all 
integer vectors $u\in \mathbb{Z}_{+}^{3}$ with $|u|\leq 3$. Then the base ring $K[B]$ 
is the Veronese subring $K[x, y, z]^{(3)}$,  which is Gorenstein. On the other hand, 
since the Hilbert series of the Ehrhart ring $K[P]$ is $(1+16t+10t^{2})/(1-t)^{4}$, it
follows that $K[P]$ is not Gorenstein.\\
$b)$ Let $P\subset \mathbb{Z}_{+}^{3}$ be the discrete polymatroid consisting of all 
integer vectors $u\in \mathbb{Z}_{+}^{3}$ with $|u|\leq 4$. Then the base ring 
$K[B]=K[x, y, z]^{(4)}$ is not Gorenstein. On the other hand, 
since the Hilbert series of the Ehrhart ring $K[P]$ is $(1+31t+31t^{2}+t^{3})/(1-t)^{4}$, it
follows that $K[P]$ is  Gorenstein.\\
$c)$ Let $P\subset \mathbb{Z}_{+}^{2}$ be the discrete polymatroid with 
$B=\{(1, 2), (2, 1)\}$ its set of bases. Then both $K[P]$ and $K[B]$ are Gorenstein.
\end{example}

However in \cite{HH} Herzog and Hibi give a combinatorial criterion for $K[P]$ to 
be Gorenstein.
\begin{theorem}$(\cite{HH})$
Let $P\subset \mathbb{Z}_{+}^{n}$ be a discrete polymatroid and suppose that the
canonical basis vectors $e_{1},\ldots, e_{n}$ of $\mathbb{R}^{n}$ belong to $P$. Let
$\rho$ denote the ground set rank function of the integral polymatroid 
$\mathcal{P}=conv(P)\subset \mathbb{R}^{n}$. Then the Ehrhart ring $K[P]$ of $P$ is
Gorenstein if and only if there exists an integer $\delta \geq 1$ such that 
\[\rho(F)=\frac{1}{\delta}(|F|+1)\] for all $\rho-closed$ and $\rho-inseparable$ subsets
$F$ on $[n]$.
\end{theorem}

We now turn to the problem when the base ring of a discrete polymatroid
is Gorenstein. A complete answer is not given so far. However, there are some
particular classes for which there is a complete description. For example, in
\cite{HN} there is a classification of the Gorenstein rings belonging to the
class of algebras of Veronese type. Herzog and Hibi, in \cite{HH}[Theorem 7.6.]
find a characterization for the base ring of a generic discrete polymatroid
to be Gorenstein.

We will end this chapter with the last phrase from the paper of Herzog and Hibi,
\emph{Discrete Polymatroids}:

$''$ It would, of course, be of interest to classify all transversal polymatroids
with Gorenstein base rings. $''$

\chapter{The type of the base ring associated to a transversal polymatroid}
\[\]

In this chapter we determine the facets of the polyhedral cone generated
 by the exponent set of the monomials defining the base ring associated
 to a transversal polymatroid presented by 
 ${\bf{\mathcal{A}}}=\{A_{1},\ldots,A_{n}\}$. The
 importance of knowing those facets
 comes from the fact that the canonical module of the base ring can be
 expressed in terms of the relative interior of the cone. Since the base ring
 of a transversal polymatroid  is normal  using,  
 Danilov-Stanley theorem we can find all minimal generators of 
 the canonical module 
 $\omega_{K[\mathcal{A}]}$  as an ideal of $K[\mathcal{A}]$
 generated by monomials. So, 
 we can compute the \emph{type} of $K[\mathcal{A}]$. Also, this would allow
 us to compute the \emph{a-\rm{invariant}} of those base rings. The results
 presented were discovered by extensive computer algebra experiments
performed with {\it{Normaliz}} \ \cite{BK}.

\section{Cones of dimension $n$ with $n+1$ facets.}
\[\]
Let $n\in \mathbb{N},$ \ $n \geq 3,$ $\sigma \in S_{n},$ \
$\sigma=(1,2,\ldots, n)$ \ the cycle of length $n,$ \ $[n]:=\{1,
2,\ldots, n\}$ \ and $\{e_{i}\}_{1 \leq i \leq n}$ \ be the
canonical base of $\mathbb{R}^{n}.$  For a vector $x\in
 \mathbb{R}^{n}$, $x=(x_{1},\ldots,x_{n})$,  we will denote  $| \ x \ | := x_{1}+
 \ldots + x_{n}.$ If $x^{a}$ is a monomial in $K[x_{1}, \ldots, x_{n}]$ we
 set $\log (x^{a})=a$. Given a set $A$ of monomials, the $log \ set \ of \
 A,$ denoted $\log (A),$ consists of all  $\log (x^{a})$ with $x^{a}\in
 A.$
 \[\]
We consider the following set of
integer vectors of $\mathbb{N}^{n}$:

\[\begin{array}{cccccccccccccccccc}
   \ \ \ \ \ \   \downarrow i^{th} column  \ \ \ \ \ \ \ \ \ \ \ \ \ \ \ 
     \ \ \ \ \ \ \ \ \ 
  \end{array}
\]

$\nu^{j}_{\sigma^{0}[i] } := \left( \begin{array}{cccccccccccc}
    -j, & -j, & \ldots & ,-j, & (n-j), & \ldots & ,(n-j)  \\
   \end{array}
   \right),
$

\[\begin{array}{ccccccccccccccccccccccccccccccccccccccccccccccccccccc}
  &  & \ \ \ \ \ \ \    \downarrow (i+1)^{st} column \ \ \ \ \ \ &  &
  \end{array}
\]

$\nu^{j}_{\sigma^{1}[i] } := \left( \begin{array}{ccccccc}
     (n-j), & -j, & \ldots & ,-j, & (n-j), & \ldots & ,(n-j)  \\
   \end{array}
   \right),
$

\[\begin{array}{cccccccccccccccccccccc}
     \ \ \ \ \ \ \ \ \ &  &  &  &  &  &  &  &\downarrow (i+2)^{nd} column &
  \end{array}
\]

$\nu^{j}_{\sigma^{2}[i] } := \left( \begin{array}{cccccccc}
     (n-j), & (n-j)& ,-j, & \ldots & ,-j, & (n-j), & \ldots & ,(n-j)  \\
   \end{array}
   \right),
$

\[\ldots \ldots \ldots \ldots \ldots \ldots \ldots \ldots \ldots 
\ldots \ldots \ldots \ldots \ldots \ldots \ldots \ldots
\ldots \ldots \ldots \ldots \ldots \ldots \ldots \ldots \ldots
\ldots \ldots \ldots \ldots\]

\[\begin{array}{ccccccccccccccccccccccc}
     \ \ \ \ \ \ \ \ \ \ \ \downarrow (i-2)^{nd} column \ \ \ \ \downarrow (n-2)^{nd} column \ \ \ \ \
  \end{array}
\]

$\nu^{j}_{\sigma^{n-2}[i] } := \left(
\begin{array}{ccccccccccccccc}
     -j, \ \ldots \ , -j , \ (n-j) , & \ldots & , (n-j) , \ -j , \ -j  \\
   \end{array}
   \right),
$

\[\begin{array}{ccccccccccccccccccccccccccccccc}
   \ \ \ \ \ \ \ \ \ \ \ \ \ \ \downarrow (i-1)^{st} column \ \ \ \ \downarrow 
   (n-1)^{st} column \ \ \ \ \ 
  \end{array}
\]

$\nu^{j}_{\sigma^{n-1}[i] } := \left( \begin{array}{cccccccc}
     -j, & \ldots & , -j , \ (n-j), & \ldots & ,(n-j) , \ -j  \\
   \end{array}
   \right),\\
$

where $\sigma^{k}[i]:=\{\sigma^{k}(1), \ldots , \sigma^{k}(i)\}$
for all $1\leq i \leq n-2,$ $1\leq j \leq n-1$ and $0\leq k \leq n-1.$\\

$\it{Remark}:$ \ It is easy to see $(\cite{SA})$ \ that for any 
$1\leq i \leq n-2$ \ and $0\leq t \leq n-1$   we have 
$\nu^{n-i-1}_{\sigma^{t}[i]} = \nu_{\sigma^{t}[i]}.$\\

If $A_{i}$ are some nonempty subsets of $[n]$ for $1\leq i\leq
m$, ${\bf{\mathcal{A}}}=\{A_{1},\ldots,A_{m}\}$, then the
set of the vectors $\sum_{k=1}^{m} e_{i_{k}}$ with $i_{k} \in
A_{k}$ is the base of a polymatroid, called the \emph{transversal
polymatroid presented} by ${\bf{\mathcal{A}}}.$ The \emph{base ring} of
 a
transversal polymatroid presented by ${\bf{\mathcal{A}}}$ 
is the ring 
\[K[{\bf{\mathcal{A}}}]:=K[x_{i_{1}}\cdot\cdot\cdot
x_{i_{m}}\ | \  i_{j}\in A_{j},1\leq j\leq m].\]

\begin{lemma}
Let $n\in \mathbb{N},$ \ $n \geq 3,$ \ $1\leq i\leq n-2$ \ 
and $1\leq j\leq n-1.$ \ We consider the following two cases:

$a)$ \ If $i+j\leq n-1,$ \ then let $A:=\{
\log(x_{j_{1}}\cdot\cdot\cdot x_{j_{n}}) \ | \
j_{k}\in A_{k},\ for \ all \ 1\leq k \leq n
\}\subset \mathbb{N}^{n}$ be the exponent set of generators of
$K-$algebra $K[{\bf{\mathcal{A}}}],$ where
${\bf{\mathcal{A}}}=\{A_{1}=[n],\ldots,A_{i}=[n],
A_{i+1}=[n]\setminus [i],\ldots,A_{i+j}=[n]\setminus
[i],A_{i+j+1}=[n],\ldots,A_{n}=[n]\}$.

$b)$ \ If $i+j\geq n,$ \ then let $A:=\{
\log(x_{j_{1}}\cdot\cdot\cdot x_{j_{n}}) \ | \
j_{k}\in A_{k},\ for \ all \ 1\leq k \leq n
\}\subset \mathbb{N}^{n}$ be the exponent set of generators of
$K-$algebra $K[{\bf{\mathcal{A}}}],$ where
${\bf{\mathcal{A}}}=\{A_{1}=[n]\setminus
[i],\ldots,A_{i+j-n}=[n]\setminus
[i],
A_{i+j-n+1}=[n],\ldots,A_{i}=[n], A_{i+1}=[n]\setminus
[i],\ldots,A_{n}=[n]\setminus
[i]\}$.\\
Then the cone generated by $A$ has the
irreducible representation
\[\mathbb{R}_{+}A= \bigcap_{a\in N}H^{+}_{a},\] where
$N=\{\nu^{j}_{\sigma^{0}[i]}, \ e_{k} \ | \ 1\leq k \leq
n\}$ and  $\{e_{i}\}_{1 \leq i \leq n}$ \ is the
canonical base of $\mathbb{R}^{n}.$
\end{lemma}

\begin{proof}
We denote 
$
J_{k}=\left \{ \begin{array}[pos]{cccccccccccccccc}
(n-j) \ e_{k}+ j \ e_{i+1}  , \ \ \ \ \ \rm{if} \ \ 1\it{\leq k \leq i} \ \ \ \ \ \ \\

\ (n-j) \ e_{1}+j \ e_{k} , \ \ \ \ \  \ \ \rm{if} \ \it{i}+\rm{2}\it{\leq k \leq n} \ \
\end{array} \right .
$ and $J=n \ e_{n}.$ \
Since $A_{t}=[n]$ \ for any $t \in \{1,\ldots, i\}\cup \{i+j+1,
\ldots, n\}$ \ and $A_{r}=[n] \setminus [i]$ \ for any $r \in 
\{i+1,\ldots, i+j\}$ it is easy to see that for any $k \in 
\{1,\ldots, i\}$ \ and  $r \in \{i+2,\ldots, n\}$ \ the set of 
monomials $x^{n-j}_{k} \ x^{j}_{i+1},$ \ $x^{n-j}_{1} \ x^{j}_{r},$ 
\ $x^{n}_{n}$ \ is a subset of the generators of $K-$algebra 
$K[{\bf{\mathcal{A}}}].$ \ Thus one has \[\{J_{1}, \ldots, J_{i},
J_{i+2}, \ldots, J_{n}, J\}\subset A.\]
If we denote by $C$ \ the matrix with the rows the coordinates of 
$\{J_{1}, \ldots, J_{i}, J_{i+2}, \ldots, J_{n}, J\},$ \ 
then by a simple computation we get $|\det\ (C)|=n \ (n-j)^{i} \ 
j^{n-i-1}$ \ for any $1\leq i\leq n-2$ \ and $1\leq j\leq n-1.$ \ 
Thus, we get that the set \[\{J_{1}, \ldots, J_{i}, J_{i+2}, 
\ldots, J_{n}, J\}\] is linearly independent and it 
follows that $\dim \ \mathbb{R}_{+}A = n.$ \ Since $\{J_{1}, 
\ldots, J_{i}, J_{i+2}, \ldots, J_{n}\}$ is linearly independent 
and lie on the hyperplane $H_{\nu^{j}_{\sigma^{0}[i]}}$ we have 
that $\dim(H_{\nu^{j}_{\sigma^{0}[i]}}\cap \mathbb{R}_{+}A)=n-1.$\\
Now we will prove that $\mathbb{R}_{+}A \subset H_{a}^{+}$ \ for 
all $a \in N.$ \ It is enough to show that for all vectors 
$P \in A,$ \ $\langle P,a \rangle \geq 0$ for all $a \in N.$ \ 
Since $\{e_{k}\}_{1 \leq k \leq n}$ \ is  the
canonical base of $\mathbb{R}^{n},$\ we get that 
$\langle P,\ e_{k}\rangle \geq 0$ for any $1\leq k \leq n.$  \ 
Let $P \in A,$ \ $P=\log(x_{j_{1}}\cdot\cdot\cdot x_{j_{n}}).$ \ 
We have two possibilities:\\ 
$a)$ \ If $i+j\leq n-1,$ \ then let $t$ 
 be the number of $j_{k}$ \ such that \ $k \in \{1,\ldots, 
i\}\cup\{i+j+1,\ldots, n\}$ \ and $j_{k} \in [i].$ \ Thus $t \leq n-j.$ \ 
Then $\langle P, \nu^{j}_{\sigma^{0}[i]}\rangle =-t \ j + (n-j-t)\ 
(n-j) + j(n-j)= n(n-j-t) \geq 0.$\\  
$b)$ \ If $i+j\geq n,$ \ 
then let $t$ be the number of $j_{k}$ \ such that \ $i+j-n+1
\leq k \leq i$ \ and $j_{k} \in [i].$ \ Thus $t \leq n-j.$ \ 
Then $\langle P, \nu^{j}_{\sigma^{0}[i]}\rangle=-t \ j + (n-j-t)\ (n-j) 
+ j(n-j)= n(n-j-t) \geq 0.$\\ Thus \[\mathbb{R}_{+}A\subseteq 
\bigcap_{a\in N}H^{+}_{a}.\]Now we will prove the converse inclusion 
$\mathbb{R}_{+}A\supseteq \bigcap_{a\in N}H^{+}_{a}.$\\ It is  
enough to prove that the extremal rays of the cone 
$\bigcap_{a\in N}H^{+}_{a}$ \ are in
${\mathbb{R}_{+}}A.$ \ Any extremal ray of the cone  
$\bigcap_{a\in N}H^{+}_{a}$ \ can be written  as the intersection of 
$n-1$ hyperplanes $H_{a}$,  with $a\in N.$ \ 
There are two possibilities to obtain extremal rays by intersection of 
$n-1$ hyperplanes.\\ 
$First \ case.$\\ Let \ $1\leq i_{1} < \ldots < 
i_{n-1}\leq n$ \ be a sequence of integers and $\{t\}=[n]\setminus \{i_{1},
\ldots,i_{n-1}\}.$ \ The system of equations:
$ \ (*) \ \begin{cases}
z_{i_{1}}=0,\\
\vdots \\
z_{i_{n-1}}=0
\end{cases}$ admits the solution $x \in \mathbb{Z}_{+}^{n},$ \ $
x=\left( \begin{array}[pos]{c}
    x_{1}\\
    \vdots\\
    x_{n}
\end{array} \right)
$ with $ | \ x \ |=n,$ \ $x_{k}=n\cdot\delta_{kt}$ \ for all 
$1\leq k \leq n,$ where $\delta_{kt}$ is Kronecker's symbol.\\
There are two possibilities:\\ $1)$ \ If $1\leq t \leq i,$ \ 
then $H_{\nu^{j}_{\sigma^{0}[i]}}(x) < 0$ \ and thus $x\notin 
\bigcap_{a\in N}H^{+}_{a}.$\\
$2)$ \ If $i+1\leq t \leq n,$ \ then $H_{\nu^{j}_{\sigma^{0}[i]}}(x) > 0$ \ 
and thus $x\in \bigcap_{a\in N}H^{+}_{a}$ and  is an extremal ray.\\ 
Thus, there exist $n-i$ \ sequences \ $1\leq i_{1}< \ldots < i_{n-1}\leq n$ \ 
such that the system of equations $(*)$ \ has a solution 
$x \in \mathbb{Z}_{+}^{n}$ \ with $ | \ x \ |=n$ \ 
and $H_{\nu^{j}_{\sigma^{0}[i]}}(x) > 0.$\\ The extremal rays are: 
$\{n \ e_{k} \ | \ i+1 \leq k \leq n\}.$\\
$Second \ case.$\\ 
Let \ $1\leq i_{1}< \ldots < i_{n-2}\leq n$ \ be a sequence of integers and 
$\{r, \ s\}=[n]\setminus \{i_{1},\ldots,i_{n-2}\},$ \ with $r < s.$\\
Let $x \in \mathbb{Z}_{+}^{n},$ \ with $ | \ x \ |=n,$ \ 
be a solution of the system of equations:\\
$ \ (**) \ \begin{cases}
z_{i_{1}}=0,\\
\vdots \\
z_{i_{n-2}}=0,\\
-j \ z_{1}-\ldots-j \ z_{i}+(n-j)z_{i+1}+\ldots+(n-j)z_{n}=0.
\end{cases}$\\
There are two possibilities:\\
$1)$ \ If $1\leq r \leq i$ and $i+1\leq s
 \leq n,$ then the system of equations $(**)$ admits the solution
$
x=\left( \begin{array}[pos]{c}
    x_{1}\\
    \vdots\\
    x_{n}
\end{array} \right) \in \mathbb{Z}_{+}^{n},$ \ with $ | \ x \ |=n,$ \ 
with $x_{t}=j \ \delta_{ts}+(n-j)\delta_{tr}$ \ for all $1\leq t \leq n.$\\
$2)$ If $1\leq r,s \leq i$ \ or $i+1\leq r,s \leq n,$ \ then there exists 
no solution $x\in \mathbb{Z}_{+}^{n}$ \ for the system of equations $(**)$ 
\ because otherwise $H_{\nu^{j}_{\sigma^{0}[i]}}(x) > 0$ \ or 
$H_{\nu^{j}_{\sigma^{0}[i]}}(x) < 0.$\\
Thus, there exist $i(n-i)$ \ sequences \ $1\leq i_{1}< \ldots < i_{n-2}\leq n$ 
such that the system of equations $(**)$ \ has a solution 
$x \in \mathbb{Z}_{+}^{n}$ \ with $ | \ x \ |=n$,   and the extremal rays are: 
\[\{(n-j)e_{r}+j \ e_{s} \ | \ 1 \leq r \leq i \ {\rm{and}} \ i+1 \leq s \leq n\}.\]\\
In conclusion, there exist $(i+1)(n-i)$ \ extremal rays of the cone 
${\bigcap}_{a\in N}H_{a}^{+}$:\[R:=\{n e_{k} \ | \ i+1 \leq k \leq n\}
\cup\{(n-j)e_{r}+j \ e_{s} \ | \ 1 \leq r \leq i \ {\rm{and}} \ i+1 \leq s \leq n\}.\]
Since  $A=\{\log(x_{j_{1}}\cdot\cdot\cdot x_{j_{n}}) \ | \
j_{k}\in A_{k} \in {\bf{\mathcal{A}}},\ {\rm{for \ all}} \ 1\leq k \leq n
\}\subset \mathbb{N}^{n}$ \ where ${\bf{\mathcal{A}}}=\{A_{1}=[n],\ldots,A_{i}=[n],
A_{i+1}=[n]\setminus [i],\ldots,A_{i+j}=[n]\setminus
[i],A_{i+j+1}=[n],\ldots,A_{n}=[n]\}$ \ or ${\bf{\mathcal{A}}}=\{A_{1}=[n]\setminus
[i],\ldots,A_{i+j-n}=[n]\setminus [i],
A_{i+j-n+1}=[n],\ldots,A_{i}=[n], A_{i+1}=[n]\setminus
[i],\ldots,A_{n}=[n]\setminus
[i]\},$ \ it is easy to see that $R\subset A.$ \ Thus, we have $\mathbb{R}_{+}A= 
\bigcap_{a\in N}H^{+}_{a}.$\\ The representation is
irreducible because if we delete, for some $k,$ \ the hyperplane with
the normal $e_{k}$, then the $k^{'th}$ coordinate
of $\log(x_{j_{1}}\cdot\cdot\cdot x_{j_{n}})$
would be negative, which is impossible; and if we delete the
hyperplane with the normal $\nu^{j}_{\sigma^{0}[i]},$ then the cone
${\mathbb{R_{+}}}A$ would be presented by $A=\{
\log(x_{j_{1}}\cdot\cdot\cdot x_{j_{n}}) \ | \ j_{k}\in [n], \ {\rm{for \ all}} \ 1\leq
k \leq n \}$ which is impossible. \ Thus the representation
${\mathbb{R_{+}}}A= \bigcap_{a\in N}H^{+}_{a}$ is irreducible.
\end{proof}

\begin{lemma}
Let $n\in \mathbb{N},$ \ $n \geq 3,$ \ $t\geq 1,$ \ $1\leq i\leq n-2$ \ and 
$1\leq j\leq n-1.$ \ Let $A=\{ \log(x_{j_{1}}\cdot\cdot\cdot x_{j_{n}}) \ | \
j_{\sigma^{t}{(k)}}\in A_{\sigma^{t}{(k)}}, 1\leq k \leq
n\}\subset \mathbb{N}^{n}$ be the exponent set of generators of
$K-$algebra  $K[{\bf{\mathcal{A}}}],$ where ${\bf{\mathcal{A}}}=\{
A_{\sigma^{t}{(k)}} \ | \ A_{\sigma^{t}{(k)}}=[n],\ for \ k \in [i] \cup 
([n]\setminus [i+j]) \ and \ A_{\sigma^{t}{(k)}}=[n]\setminus \sigma^{t}[i], 
\ for \ i+1 \leq k \leq i+j\},$ \  if $i+j \leq n-1$ or ${\bf{\mathcal{A}}}=\{
A_{\sigma^{t}{(k)}} \ | \ A_{\sigma^{t}{(k)}}=[n]\setminus [i],\ for \ 
k \in [i+j-n] \cup ([n]\setminus [i]) \ and \ A_{\sigma^{t}{(k)}}=[n], \ 
for \ i+j-n+1 \leq k \leq i\},$ \ if $i+j \geq n.$  
Then the cone generated by $A$
has the irreducible representation
\[{\mathbb{R_{+}}}A= \bigcap_{a\in N}H^{+}_{a},\] where 
$N=\{\nu^{j}_{\sigma^{t}[i]},\ e_{k} \ | \ 1\leq k \leq
n\}.$
\end {lemma}
Note that the algebras from Lemmas 2.1.1. and 2.1.2.
are isomorphic.

\begin{lemma}
$a)$ Let $1\leq t \leq n-i-j-1$, $s\geq 2$ and 
$\beta \in {\bf{\mathbb{N}}}^{n}$ be such that 
$H_{\nu^{j}_{\sigma^{0}[i]}}(\beta)= n(n-i-j-t)$ 
and $| \ \beta \ |=sn$. 
Then $\beta_{1}+\ldots+\beta_{i}=(n-j)(s-1)+i+t$ and
$\beta_{i+1}+\ldots+\beta_{n}=n+j(s-1)-i-t$.\\
$b)$ Let $r\geq 2$, $s\geq r$, $1\leq t\leq r(n-j)-i$ and 
$\beta \in {\bf{\mathbb{N}}}^{n}$ be such that 
$H_{\nu^{j}_{\sigma^{0}[i]}}(\beta)= nt$ and $| \ \beta \ |=sn$. 
Then $\beta_{1}+\ldots+\beta_{i}=(n-j)s-t$ and
$\beta_{i+1}+\ldots+\beta_{n}=js+t$.
\end{lemma}
\begin{proof}
$a)$ Let $c=\beta_{1}+\ldots+\beta_{i}$ so that $\beta_{i+1}+\ldots+\beta_{n}=sn-c.$
Then $H_{\nu^{j}_{\sigma^{0}[i]}}(\beta)=-jc+(n-j)(sn-c)=n(sn-c-js)$. Since
$H_{\nu^{j}_{\sigma^{0}[i]}}(\beta)=n(n-i-j-t)$, it follows that
$sn-c-js=n-i-j-t$ and so $c=(n-j)(s-1)+i+t$. 
Thus, $\beta_{1}+\ldots+\beta_{i}=(n-j)(s-1)+i+t$ and
$\beta_{i+1}+\ldots+\beta_{n}=sn-c=n+j(s-1)-i-t$.\\
$b)$ The proof goes as that for $a).$
\end{proof}

\section{The type of base ring associated to transversal polymatroids 
with the cone of dimension $n$ with $n+1$ facets.}
\[\]

The main result of this chapter is the following.

\begin{theorem}
Let $R=K[x_{1},\ldots,x_{n}]$ be a standard graded polynomial ring
over a field $K$ \ and ${\bf{\mathcal{A}}}$ satisfies 
the hypothesis of $Lemma \ 2.1.1.$  Then: \\
$a)$ \ If $i+j\leq n-1,$  then the type of $K[{\bf{\mathcal{A}}}]$ is
\[type(K[{\bf{\mathcal{A}}}])=1+ \sum_{t=1}^{n-i-j-1} \binom{n+i-j+t-1}
{i-1}\binom{n-i+j-t-1}{n-i-1}.\]
$b)$  \ If $i+j\geq n,$ \ then the type of $K[{\bf{\mathcal{A}}}]$ is
\[type(K[{\bf{\mathcal{A}}}])=\sum^{r(n-j) -i}_{t=1}
\binom{r(n-j)-t-1}{i-1}\binom{rj+t-1}{n-i-1},\]
where $r=\left\lceil \frac{i+1}{n-j} \right\rceil$ $(\left\lceil x \right\rceil 
is \ the \ least \ integer \ \geq x).$
\end{theorem}

\begin{proof}
Since $K[{\bf{\mathcal{A}}}]$ is normal (\cite{HH}), 
the canonical module $\omega_{K[{\bf{\mathcal{A}}}]}$ of 
$K[{\bf{\mathcal{A}}}],$ 
with respect to standard grading, can be expressed as an ideal of 
$K[{\bf{\mathcal{A}}}]$ generated by monomials
\[\omega_{K[{\bf{\mathcal{A}}}]}=(\{x^{a}| \ a\in {\mathbb{N}}A\cap 
ri({\mathbb{R_{+}}}A)\}),\] 
where $A$ is the exponent set of the $K-$ algebra $K[{\bf{\mathcal{A}}}]$
and $ri({\mathbb{R_{+}}}A)$ denotes the relative interior of
${\mathbb{R_{+}}}A.$ By Lemma 2.1.1.,  the cone generated by $A$ has the
irreducible representation
\[\mathbb{R}_{+}A= \bigcap_{a\in N}H^{+}_{a},\] where
$N=\{\nu^{j}_{\sigma^{0}[i]}, \ e_{k} \ | \ 1\leq k \leq
n\},$ \  $\{e_{k}\}_{1 \leq k \leq n}$ \ being the
canonical base of $\mathbb{R}^{n}.$\\
$a)$ \ Let $i\in [n-2],$ $j\in [n-1]$ be
such that $i+j\leq n-1,$ \ $u_{t}=n+i-j+t,$ $v_{t}=n-i+j-t$ for 
any $t\in [n-i-j-1].$  We will denote by $M$ the set
\[M=\{\alpha \in \mathbb{N}^{n}\ | \ \alpha_{k}\geq 1, \ |(\alpha_{1},\ldots,
\alpha_{i})|=u_{t}, \ |(\alpha_{i+1},\ldots,\alpha_{n})|=v_{t} 
\ {\rm{for \ any}} \ k\in [n], \ i\in [n-2], \ \]\[ j\in [n-1] \  
{\rm{such \  that}} \ i+j\leq n-1 \ {\rm{and}} \ t\in [n-i-j-1]\}.\]
We will show that \[{\mathbb{N}}A\cap ri({\mathbb{R_{+}}}A)=
((1,\ldots,1)+({\mathbb{N}}A\cap{\mathbb{R_{+}}}A))\cup
\bigcup_{\alpha \in M} ( \{\alpha\}+({\mathbb{N}}A\cap{\mathbb{R_{+}}}A)).\]
Since for any $\alpha \in M,$ \ $H_{\nu^{j}_{\sigma^{0}[i]}}(\alpha)=
-j(n+i-j+t)+(n-j)(n-i+j-t)=n(n-i-j-t) \ > \ 0$ and $H_{\nu^{j}_
{\sigma^{0}[i]}}((1,\ldots,1))=n(n-i-j) > 0,$ \ it follows that 
\[{\mathbb{N}}A\cap ri({\mathbb{R_{+}}}A) \supseteq
((1,\ldots,1)+({\mathbb{N}}A\cap{\mathbb{R_{+}}}A))\cup
\bigcup_{\alpha \in M} ( \{\alpha\}+({\mathbb{N}}A\cap{\mathbb{R_{+}}}A)).\]
Let $\beta \in {\mathbb{N}}A\cap ri({\mathbb{R_{+}}}A),$ then $\beta_{k}\geq 1$
for any $k\in [n].$ Since $H_{\nu^{j}_{\sigma^{0}[i]}}((1,\ldots,1))=n(n-i-j) > 0,$ 
it follows that $(1,\ldots,1)\in ri({\mathbb{R_{+}}}A).$ Let $\gamma 
\in \mathbb{N}^{n}, \ \gamma=\beta - (1,\ldots,1).$ It is clear that
$H_{\nu^{j}_{\sigma^{0}[i]}}(\gamma)=H_{\nu^{j}_{\sigma^{0}[i]}}
(\beta)-n(n-i-j).$
If $H_{\nu^{j}_{\sigma^{0}[i]}}(\beta)\geq n(n-i-j),$ then  
$H_{\nu^{j}_{\sigma^{0}[i]}}(\gamma)\geq 0.$ Thus $\gamma \in 
{\mathbb{N}}A\cap {\mathbb{R_{+}}}A.$ \\
If $H_{\nu^{j}_{\sigma^{0}[i]}}(\beta)< n(n-i-j),$ \ 
then let $1\leq t \leq n-i-j-1$ such that 
$H_{\nu^{j}_{\sigma^{0}[i]}}(\beta)= n(n-i-j-t).$

We claim that for any $\beta \in {\mathbb{N}}A\cap ri({\mathbb{R_{+}}}A)$ 
with $| \ \beta \ |=sn\geq 2n$
and $t\in [n-i-j-1]$ such that  
$H_{\nu^{j}_{\sigma^{0}[i]}}(\beta)= n(n-i-j-t)$ we can find 
$\alpha \in M$ with $H_{\nu^{j}_{\sigma^{0}[i]}}(\alpha)= n(n-i-j-t)$
and $\beta-\alpha \in {\mathbb{N}}A\cap{\mathbb{R_{+}}}A.$\\
We proceed by induction on $s\geq 2.$
If $s=2$, then it is clear that $\beta \in M$. 
Indeed, then $\beta_{1}+\ldots+\beta_{i}=n-a,$ 
$\beta_{i+1}+\ldots+\beta_{n}=n+a$ for some 
$a\in \mathbb{Z}$ and so
$H_{\nu^{j}_{\sigma^{0}[i]}}(\beta)=n(n-2j+a)=
n(n-i-j-t),$ that is
$a=j-i-t$. It follows $\beta \in M.$\\
Suppose $s>2.$ Since $\beta_{1}+\ldots+\beta_{i}= 
(n-j)(s-1)+i+t\geq 2(n-j)+i+t$ by Lemma \ 2.1.3. \ 
we can get $\gamma_{e}$, with $0\leq \gamma_{e}\leq \beta_{e}$ 
for any $1\leq e \leq i$,
such that $| \ (\gamma_{1},\ldots,\gamma_{i}) \ |=n-j.$
Since $\beta_{i+1}+\ldots+\beta_{n}=j(s-1)+(n-i-t)\geq 2j+n-i-t$ 
by Lemma \ 2.1.3. \ 
we can get $\gamma_{e}$, with $0\leq \gamma_{e}\leq \beta_{e}$ 
for any $i+1\leq e \leq n$,
such that $| \ (\gamma_{i+1},\ldots,\gamma_{n}) \ |=j.$
It is clear that $\beta^{'}=\beta-\gamma \in \mathbb{R}_{+}^{n}$, 
\[ | \ (\beta_{1}^{'},\ldots,\beta_{i}^{'}) \ |=
\sum_{e=1}^{i}(\beta_{e}-\gamma_{e})=(n-j)(s-2)+i+t,\]
\[ | \ (\beta_{i+1}^{'},\ldots,\beta_{n}^{'})  \ |=
\sum_{e=i+1}^{n}(\beta_{e}-\gamma_{e})=n+j(s-2)-i-t\] and
\[H_{\nu^{j}_{\sigma^{0}[i]}}(\beta^{'})=
H_{\nu^{j}_{\sigma^{0}[i]}}(\beta)-H_{\nu^{j}_{\sigma^{0}[i]}}(\gamma)=
n(n-i-j-t)-[(-j)(n-j)+(n-j)j]=n(n-i-j-t).\]
So, $\beta^{'}\in {\mathbb{N}}A\cap ri({\mathbb{R_{+}}}A)$ and, since 
$| \ \beta^{'} \ |=n(s-1),$ we get by induction hypothesis that 
\[\beta^{'}\in \bigcup_{\alpha \in M} ( \{\alpha\}+({\mathbb{N}}A\cap{\mathbb{R_{+}}}A)).\]
Since $H_{\nu^{j}_{\sigma^{0}[i]}}(\gamma)=0$ and $\gamma \in \mathbb{R}_{+}^{n},$ 
we get that $ \gamma \in {\mathbb{N}}A \cap {\mathbb{R_{+}}}A$ and so
\[\beta \in \bigcup_{\alpha \in M} ( \{\alpha\}+({\mathbb{N}}A\cap{\mathbb{R_{+}}}A)).\]

Thus
\[{\mathbb{N}}A\cap ri({\mathbb{R_{+}}}A) \subseteq
((1,\ldots,1)+({\mathbb{N}}A\cap{\mathbb{R_{+}}}A))\cup
\bigcup_{\alpha \in M} ( \{\alpha\}+({\mathbb{N}}A\cap{\mathbb{R_{+}}}A)).\]
So, the canonical module $\omega_{K[{\bf{\mathcal{A}}}]}$ of 
$K[{\bf{\mathcal{A}}}],$ 
with respect to standard grading, can be expressed as an ideal of 
$K[{\bf{\mathcal{A}}}],$ generated by monomials
\[\omega_{K[{\bf{\mathcal{A}}}]}=(\{x_{1}\cdot\cdot\cdot x_{n}, 
\ x^{\alpha}| \ \alpha\in M\}) K[{\bf{\mathcal{A}}}].\]
The type of $K[{\bf{\mathcal{A}}}]$ is the minimal number of generators of
the canonical module. So, \ $type(K[{\bf{\mathcal{A}}}])=\#(M)+1,$ 
where $\#(M)$ is the cardinal of $M.$\\
Note that for 
 any $t\in [n-i-j-1],$ the equation $\alpha_{1}+ \ldots +
\alpha_{i}=n+i-j+t$ has $\binom{n+i-j+t-1}{i-1}$ distinct nonnegative
integer solutions with $\alpha_{k}\geq 1,$ for any $k\in [i],$ respectively
$\alpha_{i+1}+ \ldots +
\alpha_{n}=n-i+j-t$ has $\binom{n-i+j-t-1}{n-i-1}$ distinct nonnegative
integer solutions with $\alpha_{k}\geq 1$ for any $k\in [n]\setminus[i].$
Thus, the cardinal of $M$ is \[\#(M)=\sum_{t=1}^{n-i-j-1} \binom{n+i-j+t-1}
{i-1}\binom{n-i+j-t-1}{n-i-1}.\] 

$b)$ \ Let $i\in [n-2],$ $j\in [n-1]$ 
such that $i+j\geq n,$ \ $u_{t}=r(n-j)-t,$ $v_{t}=rj+t$ for 
any $t\in [r(n-j)-i],$ where $r=\left\lceil \frac{i+1}{n-j} \right\rceil.$  
We will denote by $M^{'}$ the set
\[M^{'}=\{\alpha \in \mathbb{N}^{n}\ | \ \alpha_{k}\geq 1, \ |(\alpha_{1},\ldots,
\alpha_{i})|=u_{t}, \ |(\alpha_{i+1},\ldots,\alpha_{n})|=v_{t} 
\ {\rm{for \ any}} \ k\in [n], \ i\in [n-2], \ \]\[ j\in [n-1] \  
 {\rm{such \  that}} \ i+j\geq n \ {\rm{and}} \ t\in [r(n-j)-i]\}.\]
We will show that \[{\mathbb{N}}A\cap ri({\mathbb{R_{+}}}A)=
\bigcup_{\alpha \in M^{'}} ( \{\alpha\}+({\mathbb{N}}A\cap{\mathbb{R_{+}}}A)).\]
Since for any $\alpha \in M^{'},$ \ $H_{\nu^{j}_{\sigma^{0}[i]}}(\alpha)=
-j(r(n-j)-t)+(n-j)(rj+t)=nt \ > \ 0,$ \ it follows that 
\[{\mathbb{N}}A\cap ri({\mathbb{R_{+}}}A) \supseteq
\bigcup_{\alpha \in M^{'}} ( \{\alpha\}+({\mathbb{N}}A\cap{\mathbb{R_{+}}}A)).\]
Let $\beta \in {\mathbb{N}}A\cap ri({\mathbb{R_{+}}}A),$ then $\beta_{k}\geq 1$
for any $k\in [n].$ Since $H_{\nu^{j}_{\sigma^{0}[i]}}((1,\ldots,1))=n(n-i-j)<0,$ 
it follows that $(1,\ldots,1)\notin ri({\mathbb{R_{+}}}A).$ We claim that
$| \ \beta \ |\geq rn.$ Indeed, since $\beta \in {\mathbb{N}}A\cap ri({\mathbb
{R_{+}}}A)$ and $| \ \beta \ | =sn,$ it follows that
\[H_{\nu^{j}_{\sigma^{0}[i]}}(\beta)=-j\sum_{k=1}^{i}\beta_{i}+(n-j)
(sn-\sum_{k=1}^{i}\beta_{i}) > 0\Longleftrightarrow \sum_{k=1}^{i}\beta_{i}
< (n-j)s.\] Hence $i+1\leq s(n-j)$ and so
$r=\left\lceil \frac{i+1}{n-j} \right\rceil \leq s.$\\

We claim that for any $\beta \in {\mathbb{N}}A\cap ri({\mathbb{R_{+}}}A)$
with $| \ \beta \ |=sn \geq rn$ and $t\in [r(n-j)-i]$ such that 
$H_{\nu^{j}_{\sigma^{0}[i]}}(\beta)= nt$ we can find 
$\alpha \in M$ with $H_{\nu^{j}_{\sigma^{0}[i]}}(\alpha)= nt$
such that 
$\beta-\alpha \in {\mathbb{N}}A\cap{\mathbb{R_{+}}}A.$ \\
We proceed by induction on $s\geq r.$
If $s=r$, then it is clear that $\beta \in M^{'}$. 
Indeed, then $\beta_{1}+\ldots+\beta_{i}=n(r-1)-a,$ 
$\beta_{i+1}+\ldots+\beta_{n}=n+a$ for some 
$a\in \mathbb{Z}$ and so
$H_{\nu^{j}_{\sigma^{0}[i]}}(\beta)=n(n-jr+a)=
nt,$ that is
$a=jr-n+t$. It follows $\beta \in M^{'}.$\\
Suppose $s>r$. Since $\beta_{1}+\ldots+\beta_{i}= 
(n-j)s-t\geq (n-j)(s-r)+i,$ by Lemma \ 2.1.3
we can get $\gamma_{e}$, with $0\leq \gamma_{e}\leq \beta_{e}$ 
for any $1\leq e \leq i$,
such that $| \ (\gamma_{1},\ldots,\gamma_{i}) \ |=n-j.$
Since $\beta_{i+1}+\ldots+\beta_{n}=js+t\geq jr+t,$ by Lemma \ 2.1.3.
we can get $\gamma_{e}$, with $0\leq \gamma_{e}\leq \beta_{e}$ 
for any $i+1\leq e \leq n$,
such that $| \ (\gamma_{i+1},\ldots,\gamma_{n}) \ |=j.$
It is clear that $\beta^{'}=\beta-\gamma \in \mathbb{R}_{+}^{n}$, 
\[ | \ (\beta_{1}^{'},\ldots,\beta_{i}^{'}) \ |=
\sum_{e=1}^{i}(\beta_{e}-\gamma_{e})=(n-j)(s-1)-t,\]
\[ | \ (\beta_{i+1}^{'},\ldots,\beta_{n}^{'})  \ |=
\sum_{e=i+1}^{n}(\beta_{e}-\gamma_{e})=j(s-1)+t\] and
\[H_{\nu^{j}_{\sigma^{0}[i]}}(\beta^{'})=
H_{\nu^{j}_{\sigma^{0}[i]}}(\beta)-H_{\nu^{j}_{\sigma^{0}[i]}}(\gamma)=
nt-[(-j)(n-j)+(n-j)j]=nt.\]
So, $\beta^{'}\in {\mathbb{N}}A\cap ri({\mathbb{R_{+}}}A)$ and, since 
$| \ \beta^{'} \ |=n(s-1),$ we get by induction hypothesis that 
\[\beta^{'}\in \bigcup_{\alpha \in M^{'}} ( \{\alpha\}+({\mathbb{N}}A\cap{\mathbb{R_{+}}}A)).\]
Since $H_{\nu^{j}_{\sigma^{0}[i]}}(\gamma)=0$ and $\gamma \in \mathbb{R}_{+}^{n},$
we get that $ \gamma \in {\mathbb{N}}A \cap {\mathbb{R_{+}}}A$ and so
\[\beta \in \bigcup_{\alpha \in M^{'}} ( \{\alpha\}+({\mathbb{N}}A\cap{\mathbb{R_{+}}}A)).\]

Thus
\[{\mathbb{N}}A\cap ri({\mathbb{R_{+}}}A) \subseteq
\bigcup_{\alpha \in M^{'}} ( \{\alpha\}+({\mathbb{N}}A\cap{\mathbb{R_{+}}}A)).\]
So, the canonical module $\omega_{K[{\bf{\mathcal{A}}}]}$ of 
$K[{\bf{\mathcal{A}}}],$ 
with respect to standard grading, can be expressed as an ideal of 
$K[{\bf{\mathcal{A}}}],$ generated by monomials
\[\omega_{K[{\bf{\mathcal{A}}}]}=(\{x^{\alpha}| \ \alpha\in M^{'}\})
 K[{\bf{\mathcal{A}}}].\]
The type of $K[{\bf{\mathcal{A}}}]$ is the minimal number of generators of
the canonical module. So, \ $type(K[{\bf{\mathcal{A}}}])=\#(M^{'}),$ 
where $\#(M^{'})$ is the cardinal of $M^{'}.$\\
Note that for 
 any $t\in [r(n-j)-i],$ the equation $\alpha_{1}+ \ldots +
\alpha_{i}=r(n-j)-t$ has $\binom{r(n-j)-t-1}{i-1}$ distinct nonnegative
integer solutions with $\alpha_{k}\geq 1,$ for any $k\in [i],$ respectively
$\alpha_{i+1}+ \ldots +
\alpha_{n}=rj+t$ has $\binom{rj+t-1}{n-i-1}$ distinct nonnegative
integer solutions with $\alpha_{k}\geq 1$ for any $k\in [n]\setminus[i].$\\
Thus, the cardinal of $M^{'}$ is \[\#(M^{'})=\sum^{r(n-j) -i}_{t=1}
\binom{r(n-j)-t-1}{i-1}\binom{rj+t-1}{n-i-1}.\] 
\end{proof}

\begin{corollary} 
$K[{\bf{\mathcal{A}}}]$ is Gorenstein ring if and only if $i+j=n-1.$
\end{corollary}
\begin{proof}
$``\Leftarrow``$ From Theorem 2.2.1.a) $type(K[{\bf{\mathcal{A}}}])=1$. 
See also $\cite{SA}, \ {\rm{Lemma}} \ 5.1.$\\
$``\Rightarrow ``$ If $i+j\leq n-1,$ since $n+i-j+t-1\geq n+i-j\geq n+1-j\geq 2$ and
$n-i+j-t-1\geq n-i+j-1-(n-i-j-1)=2(j-1)\geq 0,$ it follows that 
$type(K[{\bf{\mathcal{A}}}])=1 \ \Leftrightarrow \ n-i-j-1=0 \ (q.e.d.)$ or 
$n-i-j-1\geq 1$ and $n-i+j-t-1<n-i-1$ for $t\in [n-i-j]$ $\Rightarrow$ $j<1$, which is false.\\
If $i+j\geq n,$ then using Theorem 2.2.1.b) we have $type(K[{\bf{\mathcal{A}}}])=1$
$\Leftrightarrow$ $r(n-j)-i=1,$ $r(n-j)-2=i-1$ and $rj=n-i-1$ 
$\Rightarrow$ $rn=n$ $\Leftrightarrow$ $1=\left\lceil \frac{i+1}{n-j} 
\right\rceil \geq \frac{i+1}{n-j} \geq \frac{n-j+1}{n-j}>1$, which is false.
\end{proof}

Let $S$ be a standard graded $K-$algebra over a field $K$. Recall that
 the $a-$\emph{invariant} of $S,$
denoted $a(S),$ is the degree as a rational function of the Hilbert
 series of $S,$ see for instance
(\cite[p.  99]{V}). If $S$ is Cohen-Macaulay and $\omega_{S}$ is the
 canonical module of $S,$ then 
\[a(S)= {} - \min  \{i\ | \ (\omega_{S})_{i} \neq 0\},
\] 
see \cite[p.  141]{BH} and \cite[Proposition 4.2.3]{V}.
In our situation $S = K[{\bf{\mathcal{A}}}]$ is normal \cite{HH} and
consequently Cohen-Macaulay, thus this formula applies. We have the
 following consequence of Theorem 2.2.1.
\begin{corollary}
The $a-$invariant of $K[{\bf{\mathcal{A}}}]$ is
 $a(K[{\bf{\mathcal{A}}}])= \left \{ \begin{array}[pos]{cccc}
-1 \ , \  if \ \ \  i+j\leq n-1, \\

-r \ , \  if \ \ \  i+j\geq n. \ \ \ \ \ 
\end{array} \right. $
\end{corollary}
\begin{proof}
Let $\{x^{\alpha_{1}}, \ldots, x^{\alpha_{q}}\}$ be  generators  of the
 $K-$algebra $K[{\bf{\mathcal{A}}}].$ Then  $K[{\bf{\mathcal{A}}}]$ is
 a standard graded algebra with the grading
\[K[{\bf{\mathcal{A}}}]_{i}=\sum_{|c|=i}K(x^{\alpha_{1}})^{c_{1}}\cdot
 \cdot \cdot (x^{\alpha_{q}})^{c_{q}}, \ 
\mbox{where} \ |c|=c_{1}+\ldots+c_{q}. 
\] 
If $i+j\leq n-1,$ then  
$ \min  \{i\ | \ (\omega_{K[{\bf{\mathcal{A}}}]})_{i} \neq 0\}=1$ and so
we have $a(K[{\bf{\mathcal{A}}}])=  -1.$\\
If $i+j\geq n,$ then  
$ \min  \{i\ | \ (\omega_{K[{\bf{\mathcal{A}}}]})_{i} \neq 0\}=r.$ 
So, we have
$a(K[{\bf{\mathcal{A}}}])=  -r.$
\end{proof}

\section{Ehrhart function}
We consider  a fixed set of distinct monomials
$F=\{x^{\alpha_{1}},\ldots,x^{\alpha_{r}}\}$ \ in a polynomial ring
$R=K[x_{1},\ldots, x_{n}]$ \ over a field $K.$ \\ 

Let
\[\mathcal{P}=conv(\log(F))\] \ be the convex hull of the set 
$\log(F)=\{\alpha_{1},\ldots,
\alpha_{r}\}.$ \ The $normalized \ Ehrhart \ ring$ \ of $\mathcal{P}$ is the
graded algebra \[A_{\mathcal{P}}=\bigoplus^{\infty}_{i=0}(A_{\mathcal{P}})_{i} 
\subset R[T],\] where the $i-{\rm{th}}$ \ component  is given by
\[(A_{\mathcal{P}})_{i}=\sum_{\alpha \in \mathbb{Z} \ \log(F)\cap \ 
i\mathcal{P}} \ K \ x^{\alpha} \ T^{i}.\]
The $normalized \ Ehrhart \ function$ of $\mathcal{P}$ is defined as
\[E_{\mathcal{P}}(t)= \dim_{K}(A_{\mathcal{P}})_{t}=|\ \mathbb{Z} \ 
{\log(F)}\cap \  t \ \mathcal{P} \ |.\] An important result of
{\cite{V}, Corollary 7.2.45} is the following.
\begin{theorem}
If $K[F]$ is a standard graded subalgebra of $R$ and $h$ \ is the
$Hilbert \ function$ of $K[F],$ \ then:\\
$a)$ \ $h(t)\leq E_{\mathcal{P}}(t)$ \ for all $t\geq 0,$ \ and\\
$b)$ \ $h(t)= E_{\mathcal{P}}(t)$ \ for all $t\geq 0$ \ if and only if $K[F]$
is normal.
\end{theorem}

In this section we will compute the Hilbert function and the Hilbert 
series for the $K-$ algebra $K[{\bf{\mathcal{A}}}],$ where $\bf{\mathcal{A}}$ 
satisfies the hypothesis of Lemma  2.1.1.
\begin{proposition}
In the hypothesis of Lemma $2.1.1$, the Hilbert function of
the $K-$ algebra $K[{\bf{\mathcal{A}}}]$ is
\[h(t)=\sum_{k=0}^{(n-j)t}\binom{k+i-1}{k}\binom{nt-k+n-i-1}{nt-k}.\]
\end{proposition}

\begin{proof}
From {\cite{HH}} we know that the $K-$ algebra $K[{\bf{\mathcal{A}}}]$ 
is normal. Thus, to compute the Hilbert function of $K[{\bf{\mathcal{A}}}]$ 
is equivalent to compute the Ehrhart function of ${\mathcal{P}},$ 
where ${\mathcal{P}} \ = \ conv(A)$ (see  Theorem  2.3.1.). \\
It is clear that $\mathcal{P}=\{ x\in \mathbb{R}^n \ | \ 
x_{i}\geq 0, \ 0\leq x_{1}+\ldots + x_{i}\leq n-j \ {\rm{and}} \ x_{1}+
\ldots + x_{n} = n \}$ and thus $t \ \mathcal{P}=\{ x\in \mathbb{R}^n 
\ | \ x_{i}\geq 0, \ 0\leq x_{1}+\ldots + x_{i}\leq (n-j) \ t \ {\rm{and}} 
\ x_{1}+\ldots + x_{n} = n \ t \}.$

Since for any $0\leq k \leq (n-j) \ t$ the equation $ x_{1}+\ldots + 
x_{i} = k$ has $\binom{k+i-1}{k}$ nonnegative integer solutions and the equation 
$x_{i+1}+\ldots + x_{n} = n \ t -k$ has $\binom{nt-k+n-i-1}{nt-k}$ 
nonnegative integer solutions, we get that \[E_{\mathcal{P}}(t)=|\ \mathbb{Z} 
\ A \ \cap \ t \ \mathcal{P} \ | = \sum_{k=0}^{(n-j)t}\binom{k+i-1}
{k}\binom{nt-k+n-i-1}{nt-k}.\]
\end{proof}

\begin{corollary}
In the hypothesis of $Lemma \ 2.1.1.$, the Hilbert series of the  
$K-$ algebra $K[{\bf{\mathcal{A}}}]$ is
\[H_{K[{\bf{\mathcal{A}}}]}(t) \ = \ \frac{1+h_{1} \ t + \ldots + 
h_{n-r} \ t^{n-r} }{(1-t)^{n}},\] where 
\[h_{j}=\sum_{s=0}^{j}(-1)^{s} \ h(j-s) \ \binom{n}{s},\]
$h(s)$ being the Hilbert function and $r=\left\lceil \frac{i+1}{n-j}\right\rceil.$
\end{corollary}

\begin{proof}
Since the $a-$invariant of $K[{\bf{\mathcal{A}}}]$ is
 $a(K[{\bf{\mathcal{A}}}])=  -r,$ with $r=\left\lceil 
 \frac{i+1}{n-j}\right\rceil,$ it follows that to 
 compute the Hilbert series of
 $K[{\bf{\mathcal{A}}}]$ is necessary to know the first $n-r+1$ values of the
 Hilbert function of $K[{\bf{\mathcal{A}}}]$, $h(i)$ for $0\leq i \leq
 n-r.$
Since $\dim (K[{\bf{\mathcal{A}}}])=n,$ applying $n$ times the
 difference operator  $\Delta$ ({see  \cite{BH}})  on the Hilbert function of
 $K[{\bf{\mathcal{A}}}]$ we get the conclusion.

Let $\Delta^{0}(h)_{j}:=h(j)$ for any $0\leq j \leq n-r.$ For $k\geq 1$
 let $\Delta^{k}(h)_{0}:=1$
and $\Delta^{k}(h)_{j}:=\Delta^{k-1}(h)_{j}-\Delta^{k-1}(h)_{j-1}$
for any $1\leq j \leq n-r.$
We claim that
\[\Delta^{k}(h)_{j}=\sum_{s=0}^{k}(-1)^sh(j-s)\binom{k}{s}
\]
for any $k\geq 1$ and $0\leq j \leq n-r.$ We proceed by induction on
 $k.$

If $k=1,$ then
\[
\Delta^{1}(h)_{j}=\Delta^{0}(h)_{j}-\Delta^{0}(h)_{j-1}=h(j)-h(j-1)
=\sum_{s=0}^{1}(-1)^{s}h(j-s)\binom{1}{s}
\]
for any $1\leq j \leq n-r.$\\
If $k>1,$ then
\[\Delta^{k}(h)_{j}=\Delta^{k-1}(h)_{j}-\Delta^{k-1}
(h)_{j-1}=\sum_{s=0}^{k-1}(-1)^{s}h(j-s)
\binom{k-1}{s}-\sum_{s=0}^{k-1}(-1)^{s} h(j-1-s)
\binom{k-1}{s}\]
\[=h(j)\binom{k-1}{0}+\sum_{s=1}^{k-1}(-1)^{s}h(j-s)
\binom{k-1}{s}-\sum_{s=0}^{k-2}(-1)^{s} h(j-1-s)
\binom{k-1}{s}+(-1)^kh(j-k)\binom{k-1}{k-1}\]
\[=h(j)+\sum_{s=1}^{k-1}(-1)^{s}h(j-s)\left[\binom{k-1}{s}
+\binom{k-1}{s-1}\right]+(-1)^kh(j-k)\binom{k-1}{k-1}\
 \ \ \ \ \ \ \ \ \ \ \ \ \ \ \ \ \ \ \ \ \ \ \ \ \ \ \ \]
\[=h(j)+\sum_{s=1}^{k-1}(-1)^{s}h(j-s)\binom{k}{s}+(-1)^
kh(j-k)\binom{k-1}{k-1}=\sum_{s=0}^{k}(-1)^{s}h(j-s)\binom{k}{s}.\
 \ \ \ \ \ \ \ \ \ \ \ \ \ \ \ \ \ \ \ \ \ \ \ \]
Thus, if $k=n$ it follows that
\[h_{j}=\Delta^{n}(h)_{j}=\sum_{s=0}^{n}(-1)^sh(j-s)\binom{n}{s}=
\sum_{s=0}^{j}(-1)^sh(j-s)\binom{n}{s}\]
for any $1\leq j\leq n-r.$

\end{proof}
Next we will give some examples.
\begin{example}
Let ${\bf{\mathcal{A}}}=\{A_{1},\ldots ,A_{7}\},$
where $A_{1}=A_{2}=A_{3}=A_{6}=A_{7}=[7],$ \ $A_{4}=A_{5}=[7]\setminus [3].$
The cone genereated by $A,$ the exponent set of generators of
$K-$algebra $K[{\bf{\mathcal{A}}}],$ has the
irreducible representation
\[\mathbb{R}_{+}A= H^{+}_{\nu^{2}_{\sigma^{0}[3]}}\cap 
H^{+}_{e_{1}} \cap \ldots \cap H^{+}_{e_{7}}.\] 

The type of $K[{\bf{\mathcal{A}}}]$ is
\[type(K[{\bf{\mathcal{A}}}])=1+ \binom{8}
{2}\binom{4}{3}=113.\]

The Hilbert series of $K[{\bf{\mathcal{A}}}]$ is
\[H_{K[{\bf{\mathcal{A}}}]}(t) \ = \ \frac{1+1561 t + 24795 t^{2} + 57023 t^{3}+
25571 t^{4} + 1673 t^{5} + t^{6}}{(1-t)^{7}}.\]
Note that
$type(K[{\bf{\mathcal{A}}}])=1+h_{5}-h_{1}=113.$
\end{example}

\begin{example}
Let ${\bf{\mathcal{A}}}=\{A_{1},\ldots ,A_{7}\},$
where $A_{3}=A_{4}=[7],$ \ $A_{1}=A_{2}=A_{5}=A_{6}=A_{7}=[7]\setminus [4].$
The cone genereated by $A,$ the exponent set of generators of
$K-$algebra $K[{\bf{\mathcal{A}}}],$ has the
irreducible representation
\[\mathbb{R}_{+}A= H^{+}_{\nu^{5}_{\sigma^{0}[4]}}\cap 
H^{+}_{e_{1}} \cap \ldots \cap H^{+}_{e_{7}}.\] 

The type of $K[{\bf{\mathcal{A}}}]$ is
\[type(K[{\bf{\mathcal{A}}}])=\binom{4}
{3}\binom{15}{2}+\binom{3}
{3}\binom{16}{2}=540.\]

The Hilbert series of $K[{\bf{\mathcal{A}}}]$ is
\[H_{K[{\bf{\mathcal{A}}}]}(t) \ = \ \frac{1+351t + 2835t^{2} + 3297t^{3}+
540t^{4}}{(1-t)^{7}}.\]
\end{example}

Note that
$type(K[{\bf{\mathcal{A}}}])=h_{4}=540.$

We end this section with the following open problem.

{\bf{Open Problem}:}
Let $n\geq 4,$ $A_{i}\subset [n]$ for any $1\leq i \leq n$ and 
$K[{\bf{\mathcal{A}}}]$ be the base ring associated to 
the transversal polymatroid presented by
${\bf{\mathcal{A}}}=\{A_{1},\ldots ,A_{n}\}.$
If the Hilbert series is: 
\[H_{K[{\bf{\mathcal{A}}}]}(t) \ = \ \frac{1+h_{1} \ t + \ldots + 
h_{n-r} \ t^{n-r} }{(1-t)^{n}},\] then we have the following:\\
$1)$ If $r=1,$ then $type(K[{\bf{\mathcal{A}}}])=1+h_{n-2}-h_{1}.$\\
$2)$ If $2\leq r\leq n,$ then $type(K[{\bf{\mathcal{A}}}])=h_{n-r}.$

\chapter{Intersections of base rings associated to transversal polymatroids}
\[\]

The discrete polymatroids and their base rings are studied recently
in many papers  (see \cite{C}, \cite{HH}, \cite{HHV}, \cite{V1}, \cite{V2}).
It is important to give conditions when the base ring associated to
a transversal polymatroid is Gorenstein (see \cite{HH}). In
\cite{SA} we introduced a class of such base rings. In this paper we
note that an intersection of such base rings (introduced in
\cite{SA}) is Gorenstein and give necessary and sufficient
conditions for the intersection of two base rings from \cite{SA} to
be still a base ring of a transversal polymatroid. Also, we compute
the $a$-invariant of those base rings. The results presented were
discovered by extensive computer algebra experiments performed with
{\it{Normaliz}} \ \cite{BK}.

\section{Intersection of cones of dimension $n$ with $n+1$ facets.}
\[\]

Let $r\geq 2$, $1\leq i_{1}, \ldots, i_{r} \leq n-2,$ $0=t_{1}\leq t_{2},
\ldots, t_{r}\leq n-1$ and consider \emph{r} presentations of transversal
polymatroids:
\[{\bf{\mathcal{A}}}_{s}=\{A_{s, k} \ | \ A_{s, \sigma^{t_{2}}(k)}=[n],\ {\rm{if}} \
k\in[i_{2}]\cup \{n\},\ A_{s, \sigma^{t_{2}}(k)}=[n]\setminus
\sigma^{t_{2}}[i_{2}],\ {\rm{if}} \ k\in[n-1]\setminus [i_{2}]\}\]
for any $1\leq s \leq r.$
From \cite{SA} we know that the base rings $K[{\bf{\mathcal{A}}}_{s}]$ are
Gorenstein rings and the cones generated by the exponent vectors of the
monomials defining the base ring associated to a transversal polymatroid
presented by ${\bf{\mathcal{A}}}_{s}$ are
\[{\mathbb{R_{+}}}A_{s}= \bigcap_{a\in N_{s}}H^{+}_{a},\]
\[{\rm{where}} \ N_{s}=\{\nu_{\sigma^{t_{s}}[i_{s}]},\ \nu_{\sigma^
{k}[n-1]} \ | \ 0\leq k \leq n-1\}, \ A_{s}=\{ \log(x_{j_{1}}\cdot\cdot\cdot
x_{j_{n}}) \ | \ j_{k}\in A_{s, k}, 1\leq k \leq n\}\subset \mathbb{N}^{n} \]
for any  $1\leq s \leq r.$\\
We denote by $K[A_{1}\cap \ldots \cap A_{r}],$ the $K- {\rm{algebra}}$ generated
by $x^{\alpha}$ with $\alpha \in A_{1}\cap \ldots \cap A_{r}.$\\
It is clear that one has 
\[{\mathbb{R_{+}}}(A_{1}\cap \ldots \cap A_{r})\subseteq{\mathbb{R_{+}}}A_{1}
\cap \ldots \cap {\mathbb{R_{+}}}A_{r}= \bigcap_{a\in N_{1}\cup \ldots \cup
N_{r}}H^{+}_{a}.\]
Conversely, since \[A_{1}\cap \ldots \cap A_{r}=\{\alpha \in \mathbb{N}^{n}
\ | \ \mid \alpha \mid=n, \ H_{\nu_{\sigma^{t_{s}}[i_{s}]}}(\alpha)\geq 0,
\ {\rm{for \ any}} \ 1\leq s \leq r\},\]
we have that \[{\mathbb{R_{+}}}(A_{1}\cap \ldots \cap A_{r})\supseteq
\bigcap_{a\in N_{1}\cup \ldots \cup N_{r}}H^{+}_{a}\]
and so  \[{\mathbb{R_{+}}}(A_{1}\cap \ldots \cap A_{r})=
\bigcap_{a\in N_{1}\cup \ldots \cup N_{r}}H^{+}_{a}.\]
We claim that the intersection \[\bigcap_{a\in N_{1}\cup \ldots \cup
N_{r}}H^{+}_{a}\] is the irreducible representation of the cone
${\mathbb{R_{+}}}(A_{1}\cap \ldots \cap A_{r}).$\\
We prove by induction on $r\geq 1.$\ If $r=1$, then the intersection
\[\bigcap_{a\in N_{1}}H^{+}_{a}\] is the irreducible representation of the
cone ${\mathbb{R_{+}}}(A_{1})$ (see \cite{SA}, Lemma \ 4.1).\\ If $r>1$ then
we have two cases to study:\\
$1)$ if we delete, for some $0\leq k \leq n-1,$ \ the hyperplane with
the normal $\nu_{\sigma^{k}[n-1]}$, then a coordinate
of a $\log (x_{j_{1}}\cdot\cdot\cdot
 x_{j_{i_{1}}}x_{j_{i_{1}+1}}\cdot\cdot\cdot x_{j_{n-1}}x_{j_{n}})$
would be negative, which is impossible;\\
$2)$ if we delete, for some $1\leq s \leq r$,  the hyperplane with
the normal $\nu_{\sigma^{t_{s}}[i_{s}]},$ then \[\bigcap_{a\in
N_{1}\cup \ldots \cup N_{s-1} \cup N_{s+1}\ldots N_{r}}H^{+}_{a}\]
is by induction the irreducible representation
of the cone ${\mathbb{R_{+}}}(A_{1}\cap \ldots \cap A_{s-1}\cap
A_{s+1}\ldots \cap A_{r})$, which is different from ${\mathbb{R_{+}}}
(A_{1}\cap \ldots \cap A_{r}).$
Hence, the intersection \[\bigcap_{a\in N_{1}\cup \ldots \cup
N_{r}}H^{+}_{a}\] is the irreducible representation of the cone
${\mathbb{R_{+}}}(A_{1}\cap \ldots \cap A_{r}).$
\begin{lemma}
The $K-$ algebra $K[A_{1}\cap \ldots \cap A_{r}]$ is a Gorenstein ring.
\end{lemma}
\begin{proof}
We will show that the canonical module $\omega_{K[A_{1}\cap \ldots
\cap A_{r}]}$ is generated by $(x_{1}\cdot\cdot\cdot x_{n})K[A_{1}
\cap \ldots \cap A_{r}].$
Since the semigroups $\mathbb{N}(A_{t})$ are  normal for any $1\leq
t \leq r$, it follows that $\mathbb{N}(A_{1}\cap \ldots \cap A_{r})$
is normal. Then the $K-$ algebra $K[A_{1}\cap \ldots \cap A_{r}]$
is normal \  (see \cite{BH} Theorem 6.1.4. p. 260) \ and using
the Danilov-Stanley theorem we get that the canonical
module $\omega_{K[A_{1}\cap \ldots \cap A_{r}]}$ is \[\omega_
{K[A_{1}\cap \ldots \cap A_{r}]}=(\{x^{\alpha}
\ | \ \alpha \in \mathbb{N}(A_{1}\cap \ldots \cap A_{r}) \cap
ri(\mathbb{R}_{+}(A_{1}\cap \ldots \cap A_{r}))\}).\]
Let $d_{t}$ be the greatest common divisor of $n \ {\rm{and}} \
i_{t}+1, \ \gcd(n, i_{t}+1)=d_{t},$ for any $1\leq t \leq r.$\\
For any $1\leq s \leq r$, there exist two possibilities for
the equation of the facet $H_{\nu_{\sigma^{t_{s}}[i_{s}]}}$: \\
$1)$ If $i_{s}+t_{s}\leq n,$ then the equation of the facet
$H_{\nu_{\sigma^{t_{s}}[i_{s}]}}$ is
\[H_{\nu_{\sigma^{t_{s}}[i_{s}]}}(y): \ \frac{(i_{s}+1)}{d_{s}}
\sum_{k=1}^{t_{s}}y_{k}-\frac{(n-i_{s}-1)}{d_{s}}\sum_{k=t_{s}+1}^
{t_{s}+i_{s}}y_{k}+\frac{(i_{s}+1)}{d_{s}}\sum_{k=t_{s}+i_{s}+1}^{n}y_{k}
=0. \ \ \ \ \ \ \ \ \ \ \ \ \ \ \ \ \ \]
$2)$ If $i_{s}+t_{s}\ > \ n,$ then the equation of the facet
$H_{\nu_{\sigma^{t_{s}}[i_{s}]}}$ is
\[H_{\nu_{\sigma^{t_{s}}[i_{s}]}}(y): \ -\frac{(n-i_{s}-1)}{d_{s}}
\sum_{k=1}^{i_{s}+t_{s}-n}y_{k}+ \frac{(i_{s}+1)}{d_{s}}\sum_{k=i_{s}+t_{s}
-n+1}^{t_{s}}y_{k}-\frac{(n-i_{s}-1)}{d_{s}}\sum_{k=t_{s}+1}^{n}y_{k} =0.\]
The  relative interior of the cone $\mathbb{R}_{+}(A_{1}\cap \ldots \cap A_{r})$
is \[ri(\mathbb{R}_{+}(A_{1}\cap \ldots \cap A_{r}))=\{y \in \mathbb{R}^{n}
\ | \ y_{k} \ > \ 0, \ H_{\nu_{\sigma^{t_{s}}[i_{s}]}}(y) \ > \ 0
\ {\rm{for \ any}} \ 1\leq k \leq n \ {\rm{and}} \ 1\leq s \leq r\}.\]
We will show that \[\mathbb{N}(A_{1}\cap \ldots \cap A_{r}) \cap
ri(\mathbb{R}_{+}(A_{1}\cap \ldots \cap A_{r}))= \ (1,\ldots, 1) \ + \
(\mathbb{N}(A_{1}\cap \ldots \cap A_{r}) \cap \mathbb{R}_{+}(A_{1}\cap
\ldots \cap A_{r})).\]
It is clear that $ri(\mathbb{R}_{+}(A_{1}\cap \ldots \cap A_{r})) \
\supset \ (1,\ldots, 1) \ + \mathbb{R}_{+}(A_{1}\cap \ldots \cap A_{r}).$\\
If $(\alpha_{1}, \alpha_{2}, \ldots, \alpha_{n})\in \mathbb{N}
(A_{1}\cap \ldots \cap A_{r}) \cap ri(\mathbb{R}_{+}(A_{1}\cap \ldots \cap A_{r})),$
then $\alpha_{k} \geq 1$ for any $1\leq k\leq n$ and for any $1\leq s \leq r$ we
have \[\frac{(i_{s}+1)}{d_{s}}\sum_{k=1}^{t_{s}}\alpha_{k}- \frac{(n-i_{s}-1)}
{d_{s}}\sum_{k=t_{s}+1}^{t_{s}+i_{s}}\alpha_{k}+\frac{(i_{s}+1)}
{d_{s}}\sum_{k=t_{s}+i_{s}+1}^{n}\alpha_{k} \geq 1, \ {\rm{if}} \ i_{s}+t_{s}\leq n
\ \ \ \ \ \ \ \ \ \ \ \ \ \ \ \ \ \] or
\[-\frac{(n-i_{s}-1)}{d_{s}}\sum_{k=1}^{i_{s}+t_{s}-n}\alpha_{k}+ \frac{(i_{s}+1)}
{d_{s}}\sum_{k=i_{s}+t_{s}-n+1}^{t_{s}}\alpha_{k}-\frac{(n-i_{s}-1)}
{d_{s}}\sum_{k=t_{s}+1}^{n}\alpha_{k} \geq 1, \ {\rm{if}} \ i_{s}+t_{s}\ > \ n\] and
\[ \sum_{k=1}^{n}\alpha_{k}=t \ n \ {\rm{for \ some}} \ t\geq 1.\]
We claim that there exists $(\beta_{1}, \beta_{2}, \ldots, \beta_{n})
\in \mathbb{N}(A_{1}\cap \ldots \cap A_{r}) \cap \mathbb{R}_{+}
(A_{1}\cap \ldots \cap A_{r})$ such that $(\alpha_{1}, \alpha_{2}, \ldots, \alpha_{n})
=(\beta_{1}+1, \beta_{2}+1, \ldots, \beta_{n}+1).$
Let $\beta_{k}=\alpha_{k}-1$ for all $1\leq k \leq n.$\\
It is clear that $\beta_{k}\geq 0$ and for any $1\leq s \leq r,$
\[H_{\nu_{\sigma^{t_{s}}[i_{s}]}}(\beta)=H_{\nu_{\sigma^{t_{s}}[i_{s}]}}(\alpha)-
H_{\nu_{\sigma^{t_{s}}[i_{s}]}}(1,\ldots,1)=H_{\nu_{\sigma^{t_{s}}[i_{s}]}}
(\alpha)-\frac{n}{d_{s}}.\]
If $H_{\nu_{\sigma^{t_{s}}[i_{s}]}}(\beta)=j_{s}, \ {\rm{for \ some}} 
\ 1\leq s \leq r \ {\rm{and}}
 \ 1\leq j_{s} \leq \frac{n}{d_{s}}-1,$  then we will get a contradiction.
Indeed, since $n$ divides $\sum_{k=1}^{n}\alpha_{k},$ it follows that $\frac{n}{d_{s}}$
divides $j_{s},$  which is false.\\ Hence, we have
$(\beta_{1}, \beta_{2}, \ldots, \beta_{n})\in \mathbb{N}(A_{1}\cap \ldots \cap A_{r})
\cap \mathbb{R}_{+}(A_{1}\cap \ldots \cap A_{r})$ and $(\alpha_{1}, \alpha_{2},
\ldots, \alpha_{n})\in \mathbb{N}(A_{1}\cap \ldots \cap A_{r}) \cap ri
(\mathbb{R}_{+}(A_{1}\cap \ldots \cap A_{r}))$.\\ Since $\mathbb{N}(A_{1}\cap \ldots
\cap A_{r}) \cap ri(\mathbb{R}_{+}(A_{1}\cap \ldots \cap A_{r}))= \ (1,\ldots, 1) \
+ \ (\mathbb{N}(A_{1}\cap \ldots \cap A_{r}) \cap \mathbb{R}_{+}(A_{1}\cap \ldots
\cap A_{r})),$ we get that $\omega_{K[A_{1}\cap \ldots \cap A_{r}]}= \
(x_{1}\cdot\cdot\cdot x_{n})K[A_{1}\cap \ldots \cap A_{r}].$
\end{proof}
Let $S$ be a standard graded $K-{\rm{algebra}}$ over a field $K$. Recall that the
$a-{\rm{invariant}}$ of $S,$ denoted $a(S),$ is the degree as a rational function of the
Hilbert series of $S,$ see for instance
$(\cite{V}, \ p. \ 99).$ If $S$ is Cohen-Macaulay and $\omega_{S}$ is the
canonical module of $S,$ then \[a(S)= \ - \ \min \ \{i\ | \ (\omega_{S})_{i}
\neq 0\},\] see (\cite{BH}, \ p. \ 141) and (\cite{V}, \ Proposition 4.2.3).
In our situation $S = K[A_{1}\cap \ldots \cap A_{r}]$ is normal  and
consequently Cohen-Macaulay, thus this formula applies. As consequence of
Lemma  3.1.1. we have the following.
\begin{corollary}
The $a-invariant$ of $K[A_{1}\cap \ldots \cap A_{r}]$ is $a(K[A_{1}\cap \ldots
\cap A_{r}])= \ -1.$
\end{corollary}
\begin{proof}
Let $\{x^{\alpha_{1}}, \ldots, x^{\alpha_{q}}\}$ be  the  generators  of the
$K-{\rm{algebra}}$ $K[A_{1}\cap \ldots \cap A_{r}].$ \  $K[A_{1}\cap \ldots
\cap A_{r}]$ is a standard graded algebra with the grading
\[K[A_{1}\cap \ldots
\cap A_{r}]_{i}=\sum_{|c|=i}K(x^{\alpha_{1}})^{c_{1}}\cdot \cdot \cdot
(x^{\alpha_{q}})^{c_{q}}, \ {\rm{where}} \ \mid c \mid =c_{1}+\ldots+c_{q}. \]
Since $\omega_{K[A_{1}\cap \ldots
\cap A_{r}]}= \ (x_{1}\cdot\cdot\cdot x_{n})K[A_{1}\cap \ldots
\cap A_{r}],$ it follows that $ \min \ \{i\ | \ (\omega_{K[A_{1}\cap \ldots
\cap A_{r}]})_{i} \neq 0\}=1$ and so $a(K[A_{1}\cap \ldots
\cap A_{r}])= \ -1.$
\end{proof}
\[\]
\section{When is $K[A\cap B]$ the base ring associated to some transversal polymatroid?}
\[\]

Let $n\geq 2$ and consider two transversal polymatroids presented by
${\bf{\mathcal{A}}}=\{A_{1},\ldots, A_{n}\},$ respectively ${\bf{\mathcal{B}}}=
\{B_{1},\ldots, B_{n}\}.$ Let $A$ and $B$ be the set of exponent vectors
of monomials defining the base rings $K[{\bf{\mathcal{A}}}],$ respectively
$K[{\bf{\mathcal{B}}}],$ and $K[A\cap B]$ the $K- {\rm{algebra}}$ generated
by $x^{\alpha}$ with $\alpha \in A\cap B.$\\
{\bf{Question:}} There exists a transversal polymatroid such that
its base ring is the $K-{\rm{algebra}}$ $K[A\cap B]$?\\
In the following we will give two suggestive examples.\\
{\bf{\emph{Example} $1$}.} \ Let $n=4,$ ${\bf{\mathcal{A}}}=\{A_{1},
A_{2}, A_{3}, A_{4}\},$ ${\bf{\mathcal{B}}}=\{B_{1}, B_{2}, B_{3},
B_{4}\},$ where $A_{1}=A_{4}=B_{2}= B_{3}=\{1, 2, 3, 4\},$
$A_{2}=A_{3}=\{2, 3, 4\},$  $B_{1}=B_{4}=\{1, 3, 4\}$ and
$K[{\bf{\mathcal{A}}}],$ $K[{\bf{\mathcal{B}}}]$ the base rings
associated to transversal polymatroids presented by
${\bf{\mathcal{A}}},$ respectively ${\bf{\mathcal{B}}}.$ It is easy
to see that the generators set of $K[{\bf{\mathcal{A}}}],$
respectively $K[{\bf{\mathcal{B}}}],$ is given by $A=\{y\in
\mathbb{N}^{4} \ | \  \mid y \mid =4, \ 0\leq y_{1}\leq 2, \
y_{k}\geq 0, \ 1\leq k \leq 4\},$ respectively $B=\{y\in
\mathbb{N}^{4} \ | \ \mid y \mid=4, \ 0\leq y_{2}\leq 2, \ y_{k}\geq
0, \ 1\leq k \leq 4\}.$ We show that the $K-$ algebra $K[A\cap B]$
is the base ring of the transversal polymatroid presented by
${\bf{\mathcal{C}}}=\{C_{1}, C_{2}, C_{3}, C_{4}\},$ where
$C_{1}=C_{4}=\{1, 3, 4\}, \ C_{2}=C_{3}=\{2, 3, 4\}.$\\
 Since the
base ring associated to the transversal polymatroid presented by
${\bf{\mathcal{C}}}$ has the exponent set $C=\{y\in \mathbb{N}^{4} \
| \ \mid y \mid=4, \ 0\leq y_{1}\leq 2, \ 0\leq y_{2}\leq 2, \
y_{k}\geq 0, \ 1\leq k \leq 4\},$ it follows that $K[A\cap
B]=K[{\bf{\mathcal{C}}}].$
Thus, in this example $K[A\cap B]$ is the base ring of a transversal polymatroid.\\
{\bf{\emph{Example} $2$}.} \ Let $n=4,$ ${\bf{\mathcal{A}}}=\{A_{1}, A_{2}, A_{3}, A_{4}\},$
${\bf{\mathcal{B}}}=\{B_{1}, B_{2}, B_{3}, B_{4}\}$ where
$A_{1}=A_{2}=A_{4}=B_{1}=B_{2}=B_{3}=\{1, 2, 3, 4\},$ $A_{3}=\{3, 4\},$
$B_{4}=\{1, 4\}$ and  $K[{\bf{\mathcal{A}}}],$ $K[{\bf{\mathcal{B}}}]$
the base rings associated to the transversal polymatroids presented by
${\bf{\mathcal{A}}},$ respectively ${\bf{\mathcal{B}}}.$
It is easy to see that the generators set of $K[{\bf{\mathcal{A}}}],$
respectively $K[{\bf{\mathcal{B}}}]$, is $A=\{y\in \mathbb{N}^{4} \ |
\ \mid y \mid=4, \ 0\leq y_{1}+y_{2}\leq 3, \ y_{k}\geq 0, \ 1\leq k \leq 4\},$
respectively $B=\{y\in \mathbb{N}^{4} \ | \ \mid y \mid=4, \
0\leq y_{2}+y_{3}\leq 3, \ y_{k}\geq 0, \ 1\leq k \leq 4\}.$
We claim that there exists no transversal polymatroid ${\bf{\mathcal{P}}}$
such that the $K-$ algebra $K[A\cap B]$ is its base ring. Suppose,
on the contrary, that ${\bf{\mathcal{P}}}$ is presented by ${\bf{\mathcal{C}}}
=\{C_{1}, C_{2}, C_{3}, C_{4}\}$ with each $C_{k}\subset [4].$ Since $(3,0,1,0),
(3,0,0,1) \in {\bf{\mathcal{P}}}$ and $(3,1,0,0)\notin {\bf{\mathcal{P}}},$
we may assume by changing the numerotation of  $\{C_{i}\}_{i=\overline{1,4}}$
that  $1\in C_{1}, 1\in C_{2}, 1\in C_{4}$ and $C_{3}=\{3, 4\}.$ Since
$(0,3,0,1)\in {\bf{\mathcal{P}}},$ we may assume that $2\in C_{1}, 2\in C_{2},
2\in C_{4}.$ Hence $(0,3,1,0)\in {\bf{\mathcal{P}}},$ a contradiction.\\

Let $1\leq i_{1}, i_{2} \leq n-2,$ $0\leq t_{2}\leq n-1$ and
$\tau \in S_{n-2},$ $\tau=(1,2,\ldots, n-2)$ \ the cycle of length $n-2.$
\ We consider two transversal polymatroids presented by:
\[{\bf{\mathcal{A}}}=\{A_{k} \ | \ A_{k}=[n],\ {\rm{if}} \ k\in[i_{1}]\cup \{n\},
\ A_{k}=[n]\setminus [i_{1}],\ {\rm{if}} \ k\in[n-1]\setminus [i_{1}]\}
\ \ \ \ \ \ \ \ \ \ \ \ \ \ \ \] and
\[{\bf{\mathcal{B}}}=\{B_{k} \ | \ B_{\sigma^{t_{2}}(k)}=[n],\ {\rm{if}} \
k\in[i_{2}]\cup \{n\},\ B_{\sigma^{t_{2}}(k)}=[n]\setminus
\sigma^{t_{2}}[i_{2}],\ {\rm{if}} \ k\in[n-1]\setminus [i_{2}]\}\] such that $A,$
respectively $B$, is the exponent vectors of the monomials defining the
base ring associated to the transversal polymatroid presented by
${\bf{\mathcal{A}}},$ respectively ${\bf{\mathcal{B}}}.$
From \cite{SA} we know that the base rings $K[{\bf{\mathcal{A}}}]$
and $K[{\bf{\mathcal{B}}}]$ are Gorenstein rings and the cones
generated by the exponent vectors of the monomials defining the base rings
associated to the transversal polymatroids presented by ${\bf{\mathcal{A}}}$
and ${\bf{\mathcal{B}}}$ are:
\[{\mathbb{R_{+}}}A= \bigcap_{a\in N_{1}}H^{+}_{a}, \ \ \ \
{\mathbb{R_{+}}}B= \bigcap_{a\in N_{2}}H^{+}_{a},\] where $N_{1}=
\{\nu_{\sigma^{0}[i_{1}]},\ \nu_{\sigma^{k}[n-1]} \ | \ 0\leq k \leq
n-1\}, \ \ N_{2}=\{\nu_{\sigma^{t_{2}}[i_{2}]},\ \nu_{\sigma^{k}[n-1]}
\ | \ 0\leq k \leq
n-1\},$ \[A=\{ \log(x_{j_{1}}\cdot\cdot\cdot x_{j_{n}}) \ | \
j_{k}\in A_{k}, 1\leq k \leq
n\}\subset \mathbb{N}^{n}\ {\rm{and}} \]\[ B=\{ \log(x_{j_{1}}\cdot\cdot\cdot
x_{j_{n}}) \ | \ j_{k}\in B_{k}, 1\leq k \leq
n\}\subset \mathbb{N}^{n}.\]
It is easy to see that $A=\{\alpha \in  \mathbb{N}^{n} \ | \
\ 0\leq \alpha_{1}+\ldots + \alpha_{i_{1}}\leq i_{1}+1 \ {\rm{and}}
\ \mid \alpha \mid=n \} $ and\\
$B=\{\alpha \in  \mathbb{N}^{n} \ | \ \ 0\leq
\alpha_{t_{2}+1}+\ldots + \alpha_{t_{2}+i_{2}}\leq i_{2}+1 \ {\rm{and}} \
\mid \alpha \mid=n \}, \ {\rm{if}} \ i_{2}+t_{2} \leq n $ \\ or\\
\[B=\{\alpha \in  \mathbb{N}^{n} \ | \ 0\leq
\sum_{s=1}^{i_{2}+t_{2}-n}\alpha_{s}+ \sum_{s=t_{2}+1}^{n}\alpha_{s}
\leq i_{2}+1 \ {\rm{and}} \ \mid \alpha \mid=n \}, \ {\rm{if}} \ i_{2}+t_{2} \geq n .
\ \ \ \ \ \ \ \ \ \ \ \ \ \ \ \ \ \ \ \ \]

For any base ring $K[{\bf{\mathcal{A}}}]$ of a transversal polymatroid
presented by ${\bf{\mathcal{A}}}=\{A_{1},\ldots ,A_{n}\}$ we
associate a $(n\times n)$ square tiled by closed unit subsquares,
called {\bf{boxes}}, colored with two colors,
$''white''$ and $''black'',$ as follows:
the box of coordinate $(i,j)$ is $''white''$ if $j\in A_{i},$
otherwise the box is $''black''.$ We will call this square  the
{\bf{polymatroidal diagram}} associated to the presentation
${\bf{\mathcal{A}}}=\{A_{1},\ldots ,A_{n}\}.$

Next we give necesary and sufficient conditions such that the
$K-{\rm{algebra}}$ $K[A\cap B]$ is the base ring associated to some
transversal polymatroid.

\begin{theorem}
Let $1\leq i_{1}, i_{2} \leq n-2,$ $0\leq t_{2}\leq n-1.$ We
consider two presentations of transversal poymatroids presented by:
${\bf{\mathcal{A}}}=\{A_{k} \ | \ A_{k}=[n],\ if \ k\in[i_{1}] \cup
\{n\},\ A_{k}=[n]\setminus [i_{1}],\ if \ k\in[n-1]\setminus
[i_{1}]\}$ and ${\bf{\mathcal{B}}}=\{B_{k} \ | \
B_{\sigma^{t_{2}}(k)}=[n],\ if \ k\in[i_{2}]\cup \{n\},\
B_{\sigma^{t_{2}}(k)}=[n]\setminus \sigma^{t_{2}}[i_{2}],\ if \
k\in[n-1]\setminus [i_{2}]\}$ such that $A,$ respectively $B,$ is the set of
exponent vectors of the monomials defining the base ring
associated to the transversal polymatroid presented by
${\bf{\mathcal{A}}},$ respectively ${\bf{\mathcal{B}}}.$\\
Then, the $K-algebra$ $K[A\cap B]$ is the base ring associated to a
transversal polymatroid if and only if one of the
following conditions holds:\\
$a)$ \ \ $i_{1}=1;$ \\
$b)$ \ \ $i_{1}\geq 2$ and $t_{2}=0;$\\
$c)$ \ \ $i_{1}\geq 2$ and $t_{2}=i_{1};$\\
$d)$ \ \ $i_{1}\geq 2,$  $1\leq t_{2} \leq i_{1}-1$ and
$i_{2}\in \{1,\ldots,i_{1}-t_{2}\}\cup \{n-t_{2},\ldots,n-2\};$\\
$e)$ \ \ $i_{1}\geq 2,$  $i_{1}+1\leq t_{2} \leq n-1$ and
$i_{2}\in \{1,\ldots,n-t_{2}\}
\cup \{n-t_{2}+i_{1},\ldots,n-2\}.$
\end{theorem}
The proof follows from the following three lemmas.
\begin{lemma}
Let $A$ and $B$ be as above. If $i_{1}\geq 2,$ $1\leq t_{2}
\leq i_{1}-1,$ then  the $K-algebra$ $K[A\cap B]$ is the base
ring associated to some transversal polymatroid if and only
if $i_{2}\in \{1,\ldots,i_{1}-t_{2}\}\cup \{n-t_{2},\ldots,n-2\}.$
\end{lemma}
\begin{proof}
$''\Leftarrow''$
Let $i_{2}\in \{1,\ldots,i_{1}-t_{2}\}\cup \{n-t_{2},\ldots,n-2\}.$
We will prove that there exists a transversal polymatroid ${\bf{\mathcal{P}}}$
presented by  ${\bf{\mathcal{C}}}=\{C_{1},\ldots,C_{n}\}$ such
that the base ring associated to ${\bf{\mathcal{P}}}$ is $K[A\cap B].$\\
We have two cases to study.\\
$\bf Case \ 1.$ If $i_{2}+t_{2}\leq i_{1},$ then let ${\bf{\mathcal{P}}}$
be the transversal polymatroid presented by  ${\bf{\mathcal{C}}}=
\{C_{1},\ldots,C_{n}\},$ where
\begin{align*}
&C_{1}=\ldots=C_{i_{2}}=C_{n}=[n], \\
&C_{i_{2}+1}=\ldots =C_{i_{1}}=[n]\setminus \sigma^{t_{2}}[i_{2}],\\
&C_{i_{1}+1}=\ldots =C_{n-1}=[n]\setminus [i_{1}]. 
\end{align*}
The associated polymatroidal diagram is the following.

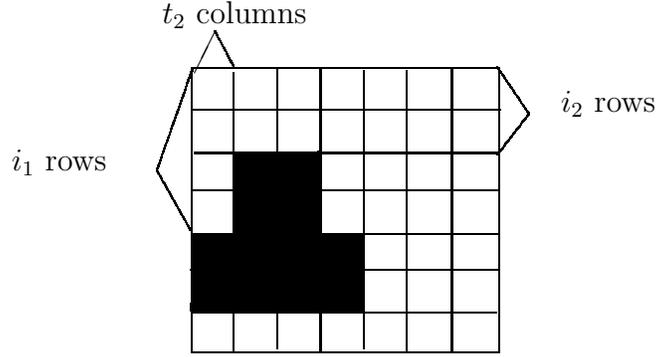
\unitlength 1mm 
\linethickness{0.4pt}
\ifx\plotpoint\undefined\newsavebox{\plotpoint}\fi 
\begin{picture}(0,67)(-3,0)
\put(54,18.25){\framebox(40.75,37.75)[cc]{}}
\put(59.5,55.5){\line(0,-1){37}}
\put(65.25,56){\line(0,-1){37.25}}
\put(71,56){\line(0,-1){37.5}}
\put(76.75,55.75){\line(0,-1){37}}
\put(82.5,55.75){\line(0,-1){37.5}}
\put(88.5,56){\line(0,-1){37.75}}
\put(54,50.5){\line(1,0){41}}
\put(54.25,44.75){\line(1,0){40.5}}
\put(54,39.75){\line(1,0){40.75}}
\put(54,34){\line(1,0){40.75}}
\put(53.75,29.25){\line(1,0){41}}
\put(54,23.5){\line(1,0){40.5}}
\put(53.75,23.75){\rule{.5\unitlength}{1\unitlength}}
\put(54,23.75){\rule{22.75\unitlength}{10\unitlength}}
\put(59.5,34){\rule{11.5\unitlength}{10.5\unitlength}}
\multiput(54,56)(-.03368794,-.09574468){141}{\line(0,-1){.09574468}}
\multiput(49.25,42.5)(.03358209,-.05970149){134}{\line(0,-1){.05970149}}
\put(54.25,55.5){\line(1,2){2.75}}
\multiput(57,61)(.0333333,-.0633333){75}{\line(0,-1){.0633333}}
\multiput(70.75,45)(.03125,-.03125){8}{\line(0,-1){.03125}}
\multiput(94.5,56.25)(.03373016,-.04960317){126}{\line(0,-1){.04960317}}
\multiput(98.75,50)(-.03373016,-.04365079){126}{\line(0,-1){.04365079}}
\put(103,50.25){$i_{2} \ {\rm{rows}}$}
\put(30,42.5){$i_{1} \ {\rm{rows}}$}
\put(50,62){$t_{2} \ {\rm{columns}}$}
\end{picture}

It is easy to see that the base ring associated to the transversal
polymatroid ${\bf{\mathcal{P}}}$ presented by ${\bf{\mathcal{C}}}$
is generated by the following set of monomials
\[\{x_{t_{2}+1},\ldots,x_{t_{2}+i_{2}}\}^{i_{2}+1-k}\{x_{1},\ldots,
x_{t_{2}},x_{t_{2}+i_{2}+1},\ldots,x_{i_{1}}\}^{i_{1}-i_{2}+k-s}
\{x_{i_{1}+1},\ldots,x_{n}\}^{n-1-i_{1}+ s}\] for any
$0\leq k \leq i_{2}+1$ and $0\leq s \leq i_{1}-i_{2}+k.$
If $x^{\alpha}\in K[{\bf{\mathcal{C}}}],$
$\alpha=(\alpha_{1},\ldots,\alpha_{n})\in {\mathbb{N}}^{n},$
then there exist $0\leq k \leq i_{2}+1$ and $0\leq s \leq i_{1}-i_{2}+k$
such that
\[\alpha_{t_{2}+1}+\ldots + \alpha_{t_{2}+i_{2}}= i_{2}+1-k \ {\rm{and}} \
\alpha_{1}+\ldots + \alpha_{i_{1}}= i_{1}+1-s\]
and thus, $K[{\bf{\mathcal{C}}}]\subset K[A\cap B].$\\
Conversely, \ if $\alpha \in A\cap B$ then $\alpha_{t_{2}+1}+
\ldots + \alpha_{t_{2}+i_{2}}\leq i_{2}+1 \ {\rm{and}} \ \alpha_{1}+\ldots +
\alpha_{i_{1}}\leq i_{1}+1$
and thus there exist $0\leq k \leq i_{2}+1$ and $0\leq s \leq i_{1}-i_{2}+k$
such that
\[x^{\alpha}\in \{x_{t_{2}+1},\ldots,x_{t_{2}+i_{2}}\}^{i_{2}+1-k}
\{x_{1},\ldots,x_{t_{2}}, x_{t_{2}+i_{2}+1},\ldots,x_{i_{1}}\}^
{i_{1}-i_{2}+k-s}\{x_{i_{1}+1},\ldots,x_{n}\}^{n-1-i_{1}+ s}.\]
Thus, $K[{\bf{\mathcal{C}}}]\supset K[A\cap B]$ and so $K[{\bf{\mathcal{C}}}]= K[A\cap B].$\\
${\bf Case} \ 2.$ If $i_{2}+t_{2}>i_{1},$ then it follows that $i_{2}\geq n-t_{2}$
and  $n-i_{2}\leq t_{2}\ < \ i_{1}.$
Let ${\bf{\mathcal{P}}}$ be the transversal polymatroid presented by
${\bf{\mathcal{C}}}=\{C_{1},\ldots,C_{n}\},$ where
\begin{align*}
&C_{1}=\ldots=C_{n-i_{2}-1}=[n]\setminus \sigma^{t_{2}}[i_{2}],\\
&C_{n-i_{2}}=\ldots =C_{i_{1}}=C_{n}=[n], \\
&C_{i_{1}+1}=\ldots =C_{n-1}=[n]\setminus [i_{1}]. \\
\end{align*}
The associated polymatroidal diagram is the following.

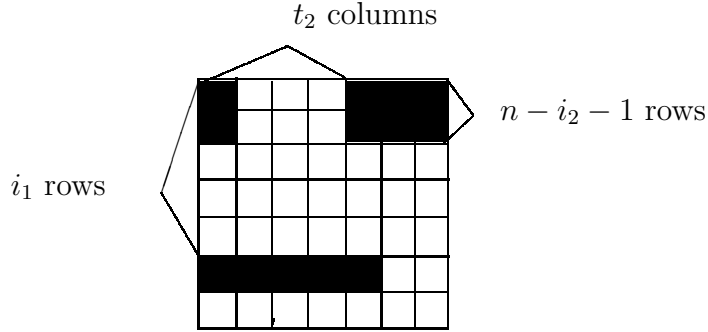
\unitlength 1mm 
\linethickness{0.4pt}
\ifx\plotpoint\undefined\newsavebox{\plotpoint}\fi 
\begin{picture}(60,70)(-15,0)
\put(45,22.25){\framebox(33,33)[cc]{}}
\put(50,55.5){\line(0,-1){33.25}}
\put(54.75,55){\line(0,-1){32.25}}
\put(59.5,55.25){\line(0,-1){32.25}}
\multiput(55,23.5)(-.033333,-.1){15}{\line(0,-1){.1}}
\put(59.5,23.5){\line(0,-1){1.5}}
\put(64.5,55.25){\line(0,-1){32.75}}
\put(69.25,55){\line(0,-1){32.75}}
\put(73.75,55){\line(0,-1){32.5}}
\put(45,51.25){\line(1,0){32.75}}
\put(44.75,46.75){\line(1,0){33.25}}
\put(45,42){\line(1,0){33.25}}
\put(45.25,37){\line(1,0){32.75}}
\put(45.25,31.75){\line(1,0){33}}
\put(45,27){\line(1,0){33}}
\put(45.25,27.25){\rule{24\unitlength}{4.25\unitlength}}
\put(45.25,46.75){\rule{4.75\unitlength}{8.25\unitlength}}
\put(64.75,47){\rule{13.25\unitlength}{8\unitlength}}
\put(45,55.25){\line(-1,-3){5}}
\multiput(40,40.25)(.03355705,-.05704698){149}{\line(0,-1){.05704698}}
\multiput(45,54.75)(.07885906,.03355705){149}{\line(1,0){.07885906}}
\multiput(56.75,59.75)(.06150794,-.03373016){126}{\line(1,0){.06150794}}
\multiput(78,55.25)(.03365385,-.04567308){104}{\line(0,-1){.04567308}}
\multiput(81.5,50.5)(-.03571429,-.03348214){112}{\line(-1,0){.03571429}}
\put(57.5,62.75){$t_{2} \ {\rm{columns}}$}
\put(85,50){$n-i_{2}-1 \ {\rm{rows}}$}
\put(20,40){$i_{1} \ {\rm{rows}}$}
\end{picture}

It is easy to see that the base ring associated to the transversal polymatroid
${\bf{\mathcal{P}}}$ presented by ${\bf{\mathcal{C}}}$
is generated by the following set of monomials
\[\{x_{i_{2}+t_{2}-n+1},\ldots,x_{t_{2}}\}^{i_{1}+1-k}
\{x_{1},\ldots, x_{i_{2}+t_{2}-n}, x_{t_{2}+1},\ldots,x_{i_{1}}\}^{k-s}
\{x_{i_{1}+1},\ldots,x_{n}\}^{n-1-i_{1}+ s}\]
for any $0\leq k \leq i_{1}+i_{2}-n+2$ and $0\leq s \leq k.$
Since $i_{2}+t_{2}\geq n$ and $0\leq s \leq k \leq i_{1}+i_{2}-n+2$, 
it follows that for any $x^{\alpha}\in K[{\bf{\mathcal{C}}}]$
we have $\alpha_{1}+\ldots+\alpha_{i_{1}}=i_{1}+1-s\leq i_{1}+1$,
$\alpha_{1}+\ldots +\alpha_{i_{2}+t_{2}-n}+\alpha_{t_{2}+1}+
\ldots+\alpha_{n}=n-1-i_{1}+k \leq i_{2}+1$
and thus, $K[{\bf{\mathcal{C}}}]\subset K[A\cap B].$\\
Conversely, \ if $\alpha \in A\cap B$ then $\alpha_{1}+\ldots+
\alpha_{i_{1}}\leq i_{1}+1$,
$\alpha_{1}+\ldots +\alpha_{i_{2}+t_{2}-n}+\alpha_{t_{2}+1}+\ldots+
\alpha_{n}\leq i_{2}+1$
and thus there exist $0\leq k \leq i_{1}+i_{2}-n+2$ and
$0\leq s \leq k$ such that \[x^{\alpha}\in \{x_{i_{2}+t_{2}-n+1},
\ldots,x_{t_{2}}\}^{i_{1}+1-k}
\{x_{1},\ldots, x_{i_{2}+t_{2}-n}, x_{t_{2}+1},\ldots,x_{i_{1}}\}^{k-s}
\{x_{i_{1}+1},\ldots,x_{n}\}^{n-1-i_{1}+ s}.\]
Thus, $K[{\bf{\mathcal{C}}}]\supset K[A\cap B]$ and so
$K[{\bf{\mathcal{C}}}]= K[A\cap B].$

$''\Rightarrow''$ \ Now suppose that there exists a transversal
polymatroid ${\bf{\mathcal{P}}}$ given by
${\bf{\mathcal{C}}}=\{C_{1},\ldots,C_{n}\}$ such that its associated
base ring is $K[A\cap B]$. We will prove that
$i_{2}\in \{1,\ldots,i_{1}-t_{2}\}\cup \{n-t_{2},\ldots,n-2\}.$\\
Suppose, on the contrary, that $i_{1}+1-t_{2}\leq i_{2}\leq n-t_{2}-1.$
We have two cases to study:\\
${\bf Case} \ 1'.$ If $n-i_{1}-1\leq i_{2}+1,$ then since
$(i_{1}+1)e_{1}+(n-i_{1}-1)e_{k}\in {\bf{\mathcal{P}}}$
for any $i_{1}+1\leq k \leq n$ and
$(i_{1}+1)e_{1}+ e_{s}+(n-i_{1}-2)e_{k}\notin {\bf{\mathcal{P}}}$
for any $2\leq s \leq i_{1}$ and $i_{1}+1\leq k \leq n,$
we may assume $1\in C_{1},\ldots, 1\in C_{i_{1}}, 1\in C_{n}$
and $C_{i_{1}+1}=\ldots = C_{n-1}=[n]\setminus[i_{1}].$
If $i_{1}\leq i_{2},$ then since $(i_{1}+1)e_{t_{2}+1}+
(n-i_{1}-1)e_{t_{2}+i_{2}+1}\in {\bf{\mathcal{P}}},$
we may assume $t_{2}+1\in C_{1},\ldots, t_{2}+1\in C_{i_{1}},
t_{2}+1\in C_{n}.$ Then $(i_{1}+1)e_{t_{2}+1}+(n-i_{1}-1)
e_{i_{1}+1}\in {\bf{\mathcal{P}}},$ which is false.
If $i_{1}>i_{2},$ then since $(i_{2}+1)e_{t_{2}+1}+(n-i_{1}-1)
e_{t_{2}+i_{2}+1}\in {\bf{\mathcal{P}}},$ we may assume
$t_{2}+1\in C_{1},\ldots, t_{2}+1\in C_{i_{2}}, t_{2}+1\in C_{n}.$
Then $(i_{2}+1)e_{t_{2}+1}+(i_{1}-i_{2})e_{1}+(n-i_{1}-1)
e_{i_{1}+1}\in {\bf{\mathcal{P}}},$ which is false.\\
${\bf Case} \ 2'.$ If $n-i_{1}-1>i_{2}+1,$ then since $(i_{1}+1)e_{1}+
(i_{2}+1)e_{i_{1}+1}+(n-i_{1}-i_{2}-2)e_{k}\in {\bf{\mathcal{P}}}$
for any $t_{2}+i_{2}+1\leq k\leq n$ and $(i_{1}+1)e_{1}+
e_{s}+(i_{2}+1)e_{i_{1}+1}+(n-i_{1}-i_{2}-3)e_{k}\notin
{\bf{\mathcal{P}}}$ for any
$1\leq s \leq i_{1}$ and $t_{2}+i_{2}+1\leq k \leq n,$
we may assume $1\in C_{1},\ldots, 1\in C_{i_{1}}, 1\in C_{n},$
$C_{i_{1}+1}=\ldots = C_{i_{1}+i_{2}+1}=[n]\setminus[i_{1}]$ and
$C_{i_{1}+i_{2}+2}=\ldots=C_{n-1}=[n]\setminus[t_{2}+i_{2}].$
If $i_{1}\leq i_{2},$ then since
$(i_{1}+1)e_{t_{2}+1}+(n-i_{1}-1)e_{t_{2}+i_{2}+1}\in
{\bf{\mathcal{P}}},$ we may assume $t_{2}+1\in C_{1},\ldots,
t_{2}+1\in C_{i_{1}}, t_{2}+1\in C_{n}.$
Then $(i_{1}+1)e_{t_{2}+1}+(i_{2}+1)e_{i_{1}+1}+
(n-i_{1}-i_{2}-2)e_{t_{2}+i_{2}+1}\in
{\bf{\mathcal{P}}},$ which is false. If $i_{1}>i_{2},$
then since $(i_{2}+1)e_{t_{2}+1}+(i_{1}-i_{2})e_{1}+
(n-i_{1}-1)e_{t_{2}+i_{2}+1}\in {\bf{\mathcal{P}}},$
we may assume $t_{2}+1\in C_{1},\ldots,
t_{2}+1\in C_{i_{2}}, t_{2}+1\in C_{n}.$ Then
$(i_{1}-i_{2})e_{1}+(i_{2}+1)e_{t_{2}+1}+(i_{2}+1)e_{t_{2}+
i_{2}}+(n-i_{1}-i_{2}-2)e_{t_{2}+i_{2}+1}\in
{\bf{\mathcal{P}}},$ which is false.
\end{proof}
\begin{lemma}
Let $A$ and $B$ be as above. If $i_{1}\geq 2,$
$i_{1}+1\leq t_{2} \leq n-1,$  then  the
$K-algebra$ $K[A\cap B]$ is the base ring associated to some
transversal polymatroid if and only if
$i_{2}\in \{1,\ldots,n-t_{2}\}
\cup \{n-t_{2}+i_{1},\ldots,n-2\}.$
\end{lemma}
\begin{proof}
$''\Leftarrow''$ \
Let $i_{2}\in \{1,\ldots,n-t_{2}\}
\cup \{n-t_{2}+i_{1},\ldots,n-2\}.$ We will prove that
there exists a transversal polymatroid ${\bf{\mathcal{P}}}$
presented by  ${\bf{\mathcal{C}}}=\{C_{1},\ldots,C_{n}\}$
such that its associated base ring is $K[A\cap B].$
We distinguish three cases to study:\\
${\bf Case} \ 1.$ If $i_{2}+t_{2}\leq n$ and $i_{1}+1+i_{2}\neq n,$
then let ${\bf{\mathcal{P}}}$ be the transversal polymatroid
presented by  ${\bf{\mathcal{C}}}=\{C_{1},\ldots,C_{n}\},$ where
\begin{align*}
&C_{1}=\ldots =C_{i_{1}}=C_{n}=[n]\setminus \sigma^{t_{2}}[i_{2}],\\
&C_{i_{1}+1}=\ldots =C_{i_{1}+i_{2}+1}=[n]\setminus [i_{1}],\\
&C_{i_{1}+i_{2}+2}=\ldots =C_{n-1}=[n]
\setminus ([i_{1}]\cup \sigma^{t_{2}}[i_{2}]).
\end{align*}
The associated polymatroidal diagram is the following.

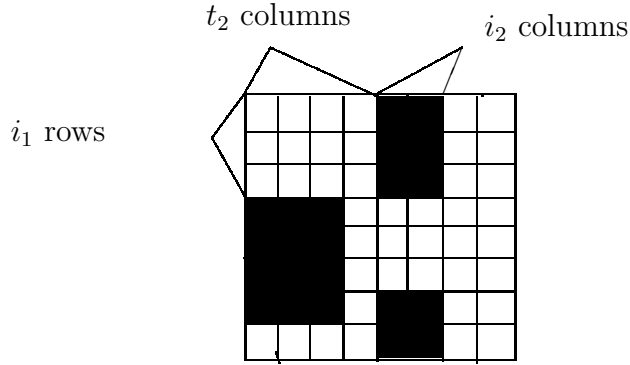
\unitlength 1mm 
\linethickness{0.4pt}
\ifx\plotpoint\undefined\newsavebox{\plotpoint}\fi 
\begin{picture}(80.25,55.75)(10,10)
\put(66.25,15.25){\framebox(35.75,35.25)[cc]{}}
\put(83.75,50.5){\line(0,-1){34.75}}
\put(74.75,50.25){\line(0,-1){35}}
\put(70.5,50.25){\line(0,-1){34.5}}
\multiput(70.25,16.25)(.033333,-.1){15}{\line(0,-1){.1}}
\put(83.75,16.25){\line(0,-1){1.25}}
\put(79.25,50.25){\line(0,-1){35}}
\put(92.5,50.25){\line(0,-1){35.25}}
\put(87.75,50.5){\line(0,-1){35}}
\multiput(97.75,50.25)(-.03125,.03125){8}{\line(0,1){.03125}}
\put(97,50.25){\line(0,-1){35.5}}
\put(66.5,33){\line(1,0){35.75}}
\put(66.25,41.25){\line(1,0){35.75}}
\put(66.25,45.5){\line(1,0){36}}
\put(66.25,36.75){\line(1,0){35.75}}
\put(66,28.75){\line(1,0){36}}
\put(66.25,20){\line(1,0){36}}
\put(66.25,24.25){\line(1,0){36}}
\put(66.5,20){\rule{12.5\unitlength}{16.5\unitlength}}
\put(83.75,15.5){\rule{8.75\unitlength}{8.5\unitlength}}
\multiput(83.5,50.5)(.06182796,.03360215){186}{\line(1,0){.06182796}}
\put(95,56.75){\line(-2,-5){2.5}}
\multiput(83.75,50.25)(-.0738342,.03367876){193}{\line(-1,0){.0738342}}
\multiput(69.5,56.75)(-.03365385,-.06009615){104}{\line(0,-1){.06009615}}
\multiput(66,50.5)(-.03373016,-.04563492){126}{\line(0,-1){.04563492}}
\multiput(61.75,44.75)(.03358209,-.05970149){134}{\line(0,-1){.05970149}}
\put(98,58.25){$i_{2} \ {\rm{columns}}$}
\put(61,59.75){$t_{2} \ {\rm{columns}}$}
\put(35,44.25){$i_{1} \ {\rm{rows}}$}
\put(83.75,36.75){\rule{8.75\unitlength}{13.5\unitlength}}
\end{picture}

It is easy to see that the base ring $K[{\bf{\mathcal{C}}}]$
associated to the transversal polymatroid ${\bf{\mathcal{P}}}$
presented by ${\bf{\mathcal{C}}}$ is generated by the
following set of monomials
\[\{x_{1},\ldots,x_{i_{1}}\}^{i_{1}+1-k}
\{x_{t_{2}+1},\ldots,x_{t_{2}+i_{2}}\}^{i_{2}+1-s}
\{x_{i_{1}+1},\ldots,x_{t_{2}},x_{t_{2}+i_{2}+1},\ldots,x_{n}\}
^{n-i_{1}-i_{2}-2+k+s}\] for any $0\leq k\leq i_{1}+1$ and
$0\leq s\leq i_{2}+1.$
If $x^{\alpha}\in K[{\bf{\mathcal{C}}}],$
$\alpha=(\alpha_{1},\ldots,\alpha_{n})\in {\mathbb{N}}^{n},$
then there exist $0\leq k \leq i_{1}+1$ and $0\leq s \leq i_{2}+1$
such that
\[\alpha_{t_{2}+1}+\ldots + \alpha_{t_{2}+i_{2}}= i_{2}+1-s
\ , \ \ \alpha_{1}+\ldots + \alpha_{i_{1}}= i_{1}+1-k\]
and thus, $K[{\bf{\mathcal{C}}}]\subset K[A\cap B].$
Conversely, \ if $\alpha \in A\cap B$ then
$\alpha_{t_{2}+1}+\ldots + \alpha_{t_{2}+i_{2}}\leq i_{2}+1
\ , \ \ \alpha_{1}+\ldots + \alpha_{i_{1}}\leq i_{1}+1;$
thus there exist $0\leq k \leq i_{1}+1$ and $0\leq s
\leq i_{2}+1$ such that
\[\alpha_{t_{2}+1}+\ldots + \alpha_{t_{2}+i_{2}}= i_{2}+1-s
\ , \ \ \alpha_{1}+\ldots + \alpha_{i_{1}}= i_{1}+1-k\] and
since $ | \ \alpha \ | = n$ it follows that
 \[x^{\alpha}\in \{x_{1},\ldots,
x_{i_{1}}\}^{i_{1}+1-k}
\{x_{t_{2}+1},\ldots,x_{t_{2}+i_{2}}\}^{i_{2}+1-s}
\{x_{i_{1}+1},\ldots,x_{t_{2}},x_{t_{2}+i_{2}+1},\ldots,x_{n}\}
^{n-i_{1}-i_{2}-2+k+s}.\]
Thus, $K[{\bf{\mathcal{C}}}]\supset K[A\cap B]$ and so
$K[{\bf{\mathcal{C}}}]= K[A\cap B].$\\
${\bf Case} \ 2.$ If $i_{2}+t_{2}\leq n$ and $i_{1}+1+i_{2}=n,$ then
$t_{2}=i_{1}+1$ and let ${\bf{\mathcal{P}}}$
be the transversal polymatroid presented by
${\bf{\mathcal{C}}}=\{C_{1},\ldots,C_{n}\},$ where
\begin{align*}
&C_{1}=\ldots =C_{i_{1}}=[n]\setminus \sigma^{t_{2}}[i_{2}],\\
&C_{i_{1}+1}= \ldots =C_{n-1}=[n]\setminus [i_{1}],\\
&C_{n}=[n].
\end{align*}
The associated polymatroidal diagram is the following.

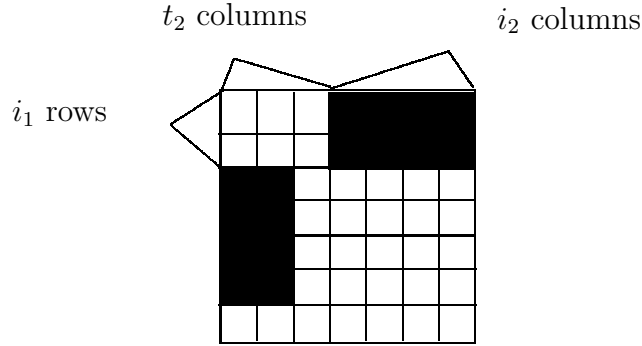
\unitlength 1mm 
\linethickness{0.4pt}
\ifx\plotpoint\undefined\newsavebox{\plotpoint}\fi 
\begin{picture}(89.75,49.75)(-7,10)
\put(49,12.5){\framebox(33.75,33.5)[cc]{}}
\put(63.5,45.75){\line(0,-1){32.75}}
\put(73.25,45.75){\line(0,-1){33}}
\put(68.25,45.75){\line(0,-1){33}}
\put(78,45.75){\line(0,-1){33}}
\put(53.75,46){\line(0,-1){33.25}}
\put(58.75,45.75){\line(0,-1){33.25}}
\put(63.5,13.5){\line(0,-1){1}}
\put(49,31.5){\line(1,0){33.75}}
\put(49,22.25){\line(1,0){34}}
\put(49.25,26.75){\line(1,0){33.5}}
\put(48.75,40.25){\line(1,0){34}}
\put(48.75,35.75){\line(1,0){33.75}}
\put(48.75,17.5){\line(1,0){34.25}}
\put(49,17.5){\rule{9.5\unitlength}{18\unitlength}}
\put(63.25,35.5){\rule{19.5\unitlength}{10.25\unitlength}}
\multiput(63.5,46.25)(.1057047,.03355705){149}{\line(1,0){.1057047}}
\multiput(79.25,51.25)(.03350515,-.04896907){97}{\line(0,-1){.04896907}}
\multiput(49,45.75)(.0336538,.0865385){52}{\line(0,1){.0865385}}
\multiput(50.75,50.25)(.11383929,-.03348214){112}{\line(1,0){.11383929}}
\multiput(49,35.75)(-.03947368,.03362573){171}{\line(-1,0){.03947368}}
\multiput(42.25,41.5)(.05223881,.03358209){134}{\line(1,0){.05223881}}
\put(21.25,42){$i_{1} \ {\rm{rows}}$}
\put(41.25,54.75){$t_{2} \ {\rm{columns}}$}
\put(85.75,54.25){$i_{2} \ {\rm{columns}}$}
\end{picture}

It is easy to see that the base ring $K[{\bf{\mathcal{C}}}]$
associated to the transversal polymatroid ${\bf{\mathcal{P}}}$
presented by ${\bf{\mathcal{C}}}$ is generated by the
following set of monomials
\[\{x_{1},\ldots,x_{i_{1}}\}^{i_{1}+1-k}x_{i_{1}+1}^{n-i_{1}-1+k-s}
\{x_{i_{1}+2},\ldots,x_{n}\}^{s}\] for any
$0\leq k\leq i_{1}+1$ and $0\leq s\leq n-i_{1}.$
If $x^{\alpha}\in K[{\bf{\mathcal{C}}}],$ $\alpha=
(\alpha_{1},\ldots,\alpha_{n})\in {\mathbb{N}}^{n},$
then there exist $0\leq k \leq i_{1}+1$ and
$0\leq s \leq i_{2}+1(=n-i_{1})$ such that
\[\alpha_{t_{2}+1}+\ldots + \alpha_{t_{2}+i_{2}}=
\alpha_{i_{1}+2}+\ldots+x_{n}=s\leq i_{2}+1 \ , \  \
\alpha_{1}+\ldots + \alpha_{i_{1}}= i_{1}+1-k\] and thus,
$K[{\bf{\mathcal{C}}}]\subset K[A\cap B].$
Conversely, \ if $\alpha \in A\cap B$ then
$\alpha_{t_{2}+1}+\ldots + \alpha_{t_{2}+i_{2}}=
\alpha_{i_{1}+2}+\ldots+x_{n}\leq i_{2}+1 \ {\rm{and}} \
\alpha_{1}+\ldots + \alpha_{i_{1}}\leq i_{1}+1;$
thus there exist $0\leq k \leq i_{1}+1$ and
$0\leq s \leq i_{2}+1$ such that \[x^{\alpha}\in
\{x_{1},\ldots,x_{i_{1}}\}^{i_{1}+1-k}x_{i_{1}+1}^
{n-i_{1}-1+k-s}\{x_{i_{1}+2},\ldots,x_{n}\}^{s}.\]
Thus, $K[{\bf{\mathcal{C}}}]\supset K[A\cap B]$ and
so $K[{\bf{\mathcal{C}}}]= K[A\cap B].$\\
${\bf Case} \ 3.$ If $i_{2}+t_{2}>n,$ then let ${\bf{\mathcal{P}}}$
be the transversal polymatroid presented by
${\bf{\mathcal{C}}}=\{C_{1},\ldots,C_{n}\},$ where
\begin{align*}
&C_{1}=\ldots =C_{i_{1}}=C_{n}=[n],\\
&C_{i_{1}+1}=
\ldots =C_{i_{1}+n-i_{2}-1}=[n]\setminus \sigma^{t_{2}}[i_{2}],\\
&C_{i_{1}+n-i_{2}}=\ldots=C_{n-1}=[n]\setminus[i_{1}].
\end{align*}

The associated polymatroidal diagram is the following:

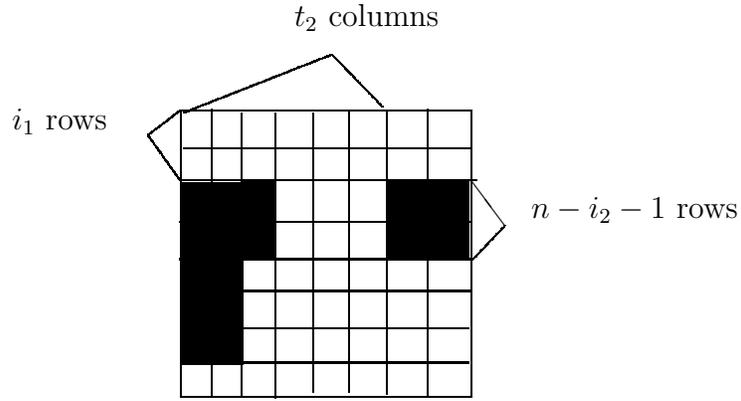
\unitlength 1mm 
\linethickness{0.4pt}
\ifx\plotpoint\undefined\newsavebox{\plotpoint}\fi 
\begin{picture}(0,30)(-20,45)
\put(32.5,14.25){\framebox(38.5,38)[cc]{}}
\put(36.5,52.5){\line(0,-1){38.25}}
\put(40.5,52){\line(0,-1){37.75}}
\put(45,52){\line(0,-1){38}}
\put(50,52){\line(0,-1){37.25}}
\put(54.75,52.25){\line(0,-1){37.5}}
\put(59.75,52.25){\line(0,-1){37.75}}
\put(65.25,52.5){\line(0,-1){38.25}}
\put(32.75,47.25){\line(1,0){38.25}}
\put(32.5,43){\line(1,0){39.25}}
\put(32.25,37.5){\line(1,0){38.75}}
\put(32.25,32.5){\line(1,0){38.75}}
\put(32.5,28.25){\line(1,0){38.5}}
\put(32.5,23.25){\line(1,0){38.25}}
\put(32.5,18.75){\line(1,0){38.25}}
\put(32.5,18.5){\rule{8.25\unitlength}{24.25\unitlength}}
\put(40.5,32.75){\rule{4.25\unitlength}{10\unitlength}}
\put(59.75,32.5){\rule{11\unitlength}{10.25\unitlength}}
\multiput(32.25,43)(-.03373016,.04761905){126}{\line(0,1){.04761905}}
\multiput(28,49)(.04381443,.03350515){97}{\line(1,0){.04381443}}
\multiput(32.75,52)(.08586957,.03369565){230}{\line(1,0){.08586957}}
\multiput(52.5,59.75)(.03358209,-.03482587){201}{\line(0,-1){.03482587}}
\put(71,43){\line(3,-4){4.5}}
\multiput(75.5,37)(-.03358209,-.03544776){134}{\line(0,-1){.03544776}}
\put(47.5,63.5){$t_{2} \ {\rm{columns}}$}
\put(10,49.75){$i_{1} \ {\rm{rows}}$}
\put(79,38){$n-i_{2}-1 \ {\rm{rows}}$}
\end{picture}
\[\] \[\] \[\] \[\] \[\] \[\] 
Since $i_{2}+t_{2}>n,$ it follows that $i_{2}+t_{2}\geq n+i_{1}$
and so $i_{2}-i_{1}\geq n-t_{2}\geq 1.$
It is easy to see that the base ring $K[{\bf{\mathcal{C}}}]$
associated to the transversal polymatroid ${\bf{\mathcal{P}}}$
presented by ${\bf{\mathcal{C}}}$ is generated by the
following set of monomials
\[\{x_{1},\ldots,x_{i_{1}}\}^{i_{1}+1-k}
\{x_{i_{1}+1},\ldots,x_{i_{2}+t_{2}-n},x_{t_{2}+1},\ldots,x_{n}\}
^{i_{2}-i_{1}+k-s}
\{x_{i_{2}+t_{2}-n+1},\ldots,x_{t_{2}}\}^{n-i_{2}-1+s}\] for any
$0\leq k \leq i_{1}+1$ and $0\leq s \leq i_{2}-i_{1}+k.$ If
$x^{\alpha}\in K[{\bf{\mathcal{C}}}],$
$\alpha=(\alpha_{1},\ldots,\alpha_{n})\in {\mathbb{N}}^{n},$ then
there exist $0\leq k \leq i_{1}+1$ and $0\leq s \leq i_{2}-i_{1}+k$
such that
\[\alpha_{1}+\ldots + \alpha_{i_{1}}= i_{1}+1-k \ , \
\ \alpha_{1}+\ldots+\alpha_{i_{2}+t_{2}-n}+
\alpha_{t_{2}+1}+\ldots+\alpha_{n}= i_{2}+1-s\] and thus,
$K[{\bf{\mathcal{C}}}]\subset K[A\cap B].$ Conversely, \ if $\alpha
\in A\cap B$ then $\alpha_{1}+ \ldots + \alpha_{i_{1}}\leq i_{1}+1 \
{\rm{and}} \ \alpha_{1}+ \ldots+\alpha_{i_{2}+t_{2}-n}+\alpha_{t_{2}+1}+
\ldots+\alpha_{n}\leq i_{2}+1,$ then there exist $0\leq k \leq
i_{1}+1$ and $0\leq s \leq i_{2}-i_{1}+k$ such that
\[\alpha_{1}+\ldots + \alpha_{i_{1}}= i_{1}+1-k \ , \ \
\alpha_{1}+\ldots+\alpha_{i_{2}+t_{2}-n}+
\alpha_{t_{2}+1}+\ldots+\alpha_{n}= i_{2}+1-s\] and since $|\ \alpha
\ |=n$ it follows that
 \[x^{\alpha}\in \{x_{1},\ldots,x_{i_{1}}\}^{i_{1}+1-k}
\{x_{i_{1}+1},\ldots,x_{i_{2}+t_{2}-n},x_{t_{2}+1},\ldots,x_{n}\}
^{i_{2}-i_{1}+k-s}
\{x_{i_{2}+t_{2}-n+1},\ldots,x_{t_{2}}\}^{n-i_{2}-1+s}.\]
Thus, $K[{\bf{\mathcal{C}}}]\supset K[A\cap B]$
and so $K[{\bf{\mathcal{C}}}]= K[A\cap B].$

$''\Rightarrow''$ \ Now suppose that there exists a transversal
polymatroid ${\bf{\mathcal{P}}}$ presented by
${\bf{\mathcal{C}}}=\{C_{1},\ldots,C_{n}\}$ such that its associated
base ring  is $K[A\cap B]$. We will prove that $i_{2}\in
\{1,\ldots,n-t_{2}\} \cup \{n-t_{2}+i_{1},\ldots,n-2\}.$ Suppose, on
the contrary, that $i_{2}\in \{n-t_{2}+1,\ldots,n-t_{2}+i_{1}-1\}.$
We may assume that $1\notin C_{i_{2}+t_{2}-n+1}, \ldots, 1\notin
C_{i_{1}}, 1\notin C_{i_{1}+1}, \ldots, 1\notin C_{n-1}.$ If $i_{1}<
i_{2},$ then $x_{1}^{i_{1}+1}x_{i_{1}+1}^{n-i_{1}-1}\in K[A\cap B]$.
But $x_{1}^{i_{1}+1}x_{i_{1}+1}^{n-i_{1}-1} \notin
K[{\bf{\mathcal{C}}}]$ because the maximal power of $x_1$ in a
minimal generator of $K[{\bf{\mathcal{C}}}]$ is $\leq i_2+t_2-n+1\leq i_1$, which is
false. If $i_{1}\geq i_{2},$ then
$x_{1}^{i_{2}+1}x_{i_{1}+1}^{n-i_{2}-1}\in K[A\cap B]$. But
$x_{1}^{i_{2}+1}x_{i_{1}+1}^{n-i_{2}-1} \notin
K[{\bf{\mathcal{C}}}]$ because the maximal power of $x_1$ in a
minimal generator of $K[{\bf{\mathcal{C}}}]$ is $\leq i_2+t_2-n+1\leq i_2$, which is
false. Thus, $i_{2}\in \{1,\ldots,n-t_{2}\} \cup
\{n-t_{2}+i_{1},\ldots,n-2\}.$
\end{proof}

\begin{lemma}
Let $A$ and $B$ as above. If $i_{1}=1$ and  $0\leq t_{2}
\leq n-1,$ or $i_{1}\geq 2$ and $t_{2}=i_{1},$ or $i_{1}\geq 2$ and
$t_{2}=0,$   then  the $K-algebra$ $K[A\cap B]$ is the base ring
associated to some transversal polymatroid.
\end{lemma}
\begin{proof}
We have three cases to study:\\
${\bf Case} \ 1.$ \ $i_{1}=1$ and  $0\leq t_{2} \leq n-1.$ \
Then we distinguish five subcases:\\
${\bf Subcase} \ 1.a.$ If $t_{2}=0,$ then we find a transversal polymatroid
${\bf{\mathcal{P}}}$
like in the Subcase $3.b.$ when $i_{1}=1.$\\
${\bf Subcase} \ 1.b.$ If $t_{2}>0 \ {\rm{and}} \ t_{2}+i_{2} \leq n \ {\rm{with}} \
i_{2}\neq n-2, \ n-3,$ then
let ${\bf{\mathcal{P}}}$ be the transversal polymatroid
presented by  ${\bf{\mathcal{C}}}=\{C_{1},\ldots,C_{n}\},$
where 
\begin{align*}
&C_{1}=C_{n}=[n]\setminus \sigma^{t_{2}}[i_{2}],\\
&C_{2}=\ldots =C_{i_{2}+2}=[n]\setminus[1], \\
&C_{i_{2}+3}=\ldots=C_{n-1}=[n]\setminus \{\{1\}\cup
\sigma^{t_{2}}[i_{2}]\}.
\end{align*}
It is easy to see that the polymatroid ${\bf{\mathcal{P}}}$
is the same as in Lemma  3.2.3.
when $i_{2}+t_{2}\leq n \ {\rm{and}} \ i_{1}+1+i_{2}\neq n.$
Thus $K[A\cap B]=K[\bf{\mathcal{C}}].$\\
${\bf Subcase} \ 1.c.$ If $t_{2}>0 \ {\rm{and}} \ t_{2}+i_{2} \leq n \
{\rm{with}} \ i_{2}= n-2,$ then
let ${\bf{\mathcal{P}}}$ be the transversal polymatroid
presented by
${\bf{\mathcal{C}}}=\{C_{1},\ldots,C_{n}\},$ where
\begin{align*}
&C_{1}=[n]\setminus \sigma^{t_{2}}[n-2]\ , \ C_{n}=[n], \ \ \ \ \ \ \ \ \ \ \\
&C_{2}=\ldots =C_{n-1}=[n]\setminus[1].
\end{align*}
It is easy to see that the polymatroid ${\bf{\mathcal{P}}}$
is the same as in Lemma  3.2.3.
when $i_{2}+t_{2}\leq n \ {\rm{and}} \ i_{1}+1+i_{2} = n.$
Thus $K[A\cap B]=K[\bf{\mathcal{C}}].$\\
${\bf Subcase} \ 1.d.$ If $t_{2}>0 \ {\rm{and}} \ t_{2}+i_{2} \leq n \ {\rm{with}}
\ i_{2}= n-3,$ then let ${\bf{\mathcal{P}}}$ be the transversal
polymatroid presented by
${\bf{\mathcal{C}}}=\{C_{1},\ldots,C_{n}\},$ where 
\begin{align*}
&C_{1}=C_{n}=[n]\setminus \sigma^{t_2}[n-3],\ \ \ \ \ \ \ \ \ \ \ \ \ \ \ \ \\
&C_{2}=\ldots =C_{n-1}=[n]\setminus[1].
\end{align*}
It is easy to see that the polymatroid
${\bf{\mathcal{P}}}$ is the same as in Lemma 3.2.3. when
$i_{2}+t_{2}\leq n \ {\rm{and}} \ i_{1}+1+i_{2}\not = n.$
Thus $K[A\cap B]=K[\bf{\mathcal{C}}].$\\
${\bf Subcase} \ 1.e.$ If $t_{2}>0$ and $t_{2}+i_{2}>n,$ then
let ${\bf{\mathcal{P}}}$ be the transversal polymatroid
presented by 
${\bf{\mathcal{C}}}=\{C_{1},\ldots,C_{n}\},$ where
\begin{align*}
&C_{1}=C_{n}=[n],\ \ \ \ \ \ \ \ \ \ \ \ \ \ \ \ \ \ \ \ \ \ \ \ \ \ \ \ \ \ \\
&C_{2}=\ldots =C_{n-i_{2}}=[n]\setminus \sigma^{t_{2}}
[i_{2}],\\
&C_{n-i_{2}+1}=\ldots=C_{n-1}=[n]\setminus \{1\}.
\end{align*}
It is easy to see that the polymatroid ${\bf{\mathcal{P}}}$
is the same as in Lemma 3.2.3.
when $i_{2}+t_{2}> n \ {\rm{and}} \ i_{1}=1.$ Thus $K[A\cap B]=K[\bf{\mathcal{C}}].$\\
${\bf Case} \ 2.$ \ $i_{1}\geq 2$ and  $t_{2}=i_{1}.$\
Then we distinguish three subcases:\\
${\bf Subcase} \ 2.a.$ If $i_{2}+t_{2}<n-1,$ then
let ${\bf{\mathcal{P}}}$ be the transversal polymatroid presented by \\
${\bf{\mathcal{C}}}=\{C_{1},\ldots,C_{n}\},$ where
\begin{align*}
& \ \ \ C_{1}=\ldots=C_{i_{1}}=C_{n}=[n]\setminus \sigma^{t_{2}}[i_{2}],\\
& \ \ \ C_{i_{1}+1}=\ldots =C_{i_{1}+i_{2}+1}=[n]\setminus [i_{1}],\\
& \ \ \ C_{i_{1}+i_{2}+2}=\ldots=C_{n-1}=[n]\setminus[i_{1}+i_{2}]. 
\end{align*}
It is easy to see that the polymatroid ${\bf{\mathcal{P}}}$
is the same as in Lemma 3.2.3.
when $i_{2}+t_{2}\leq n \ {\rm{and}} \ i_{1}+1 +i_{2}\neq n.$
Thus $K[A\cap B]=K[\bf{\mathcal{C}}].$\\
${\bf Subcase} \ 2.b.$ If $i_{2}+t_{2}=n-1,$ then
let ${\bf{\mathcal{P}}}$ be the transversal polymatroid presented by
${\bf{\mathcal{C}}}=\{C_{1},\ldots,C_{n}\},$ where
\begin{align*}
&C_{1}=\ldots=C_{i_{1}}=[n]\setminus \sigma^{t_{2}}[i_{2}], \ \\
&C_{i_{1}+1}=\ldots =C_{n-1}=[n]\setminus [i_{1}], \ \\
&C_{n}=[n]. \ 
\end{align*}
It is easy to see that the polymatroid ${\bf{\mathcal{P}}}$
is the same as in Lemma 3.2.3.
when $i_{2}+t_{2}\leq n \ {\rm{and}} \ i_{1}+1 +i_{2}= n.$
Thus $K[A\cap B]=K[\bf{\mathcal{C}}].$\\
${\bf Subcase} \ 2.c.$ If $i_{2}+t_{2}\geq n,$ then
let ${\bf{\mathcal{P}}}$ be the transversal polymatroid presented by
${\bf{\mathcal{C}}}=\{C_{1},\ldots,C_{n}\},$ where
\begin{align*}
& \ \ \ \ \ \ C_{1}=\ldots=C_{n-i_{2}-1}=[n]\setminus
\sigma^{t_{2}}[i_{2}],\\
& \ \ \ \ \ \ C_{n-i_{2}}=\ldots =C_{i_{1}}=C_{n}=[n],\\
& \ \ \ \ \ \ C_{i_{1}+1}=\ldots=C_{n-1}=[n]\setminus[i_{1}].
\end{align*}
It is easy to see that the polymatroid ${\bf{\mathcal{P}}}$
is the same as in Lemma 3.2.2.
when $i_{2}+t_{2}>i_{1}.$ Thus $K[A\cap B]=K[\bf{\mathcal{C}}].$\\
${\bf Case} \ 3.$ \ $i_{1}\geq 2$ and  $t_{2}=0.$\ 
Then we distinguish two subcases:\\
${\bf Subcase} \ 3.a.$ If $i_{2}\leq i_{1},$ then
let ${\bf{\mathcal{P}}}$ be the transversal polymatroid presented by
${\bf{\mathcal{C}}}=\{C_{1},\ldots,C_{n}\},$ where
\begin{align*}
& \ \ \ \ C_{1}=\ldots=C_{i_{2}}=C_{n}=[n],\\
& \ \ \ \ C_{i_{2}+1}=\ldots =C_{i_{1}}=[n]\setminus [i_{2}], \\
& \ \ \ \ C_{i_{1}+1}=\ldots=C_{n-1}=[n]\setminus[i_{1}].
\end{align*}
It is easy to see that the polymatroid ${\bf{\mathcal{P}}}$ is
the same as in Lemma 3.2.2.
when $i_{2}+t_{2}\leq i_{1}.$ Thus $K[A\cap B]=K[\bf{\mathcal{C}}].$\\
${\bf Subcase} \ 3.b.$ If $i_{2}> i_{1},$ then
let ${\bf{\mathcal{P}}}$ be the transversal polymatroid presented by
${\bf{\mathcal{C}}}=\{C_{1},\ldots,C_{n}\},$ where
\begin{align*}
& \ \ \ \ \ \ C_{1}=\ldots=C_{i_{1}}=C_{n}=[n], \\
& \ \ \ \ \ \ C_{i_{1}+1}=\ldots =C_{i_{2}}=[n]\setminus [i_{1}],\\
& \ \ \ \ \ \ C_{i_{2}+1}=\ldots=C_{n-1}=[n]\setminus[i_{2}].
\end{align*}
The associated polymatroidal diagram is the following.

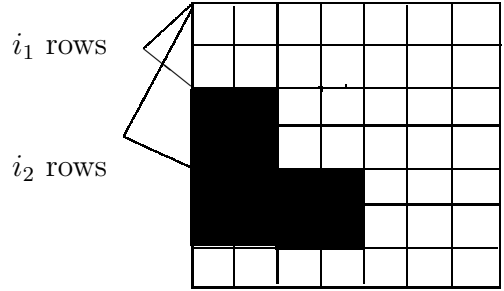
\unitlength 1mm 
\linethickness{0.4pt}
\ifx\plotpoint\undefined\newsavebox{\plotpoint}\fi 
\begin{picture}(0,70)(0,0)
\put(54,18.25){\framebox(40.75,37.75)[cc]{}}
\put(59.5,55.5){\line(0,-1){37}}
\put(65.25,56){\line(0,-1){37.25}}
\put(71,56){\line(0,-1){37.5}}
\put(76.75,55.75){\line(0,-1){37}}
\put(82.5,55.75){\line(0,-1){37.5}}
\put(88.5,56){\line(0,-1){37.75}}
\put(54,50.5){\line(1,0){41}}
\put(54.25,44.75){\line(1,0){40.5}}
\put(54,39.75){\line(1,0){40.75}}
\put(54,34){\line(1,0){40.75}}
\put(53.75,29.25){\line(1,0){41}}
\put(54,23.5){\line(1,0){40.5}}
\put(53.75,23.75){\rule{.5\unitlength}{1\unitlength}}
\multiput(70.75,45)(.03125,-.03125){8}{\line(0,-1){.03125}}
\put(74.25,45){\rule{.25\unitlength}{.25\unitlength}}
\put(53.75,23.75){\rule{11.25\unitlength}{20.75\unitlength}}
\multiput(54,56)(-.03651685,-.03370787){178}{\line(-1,0){.03651685}}
\put(47.5,50){\line(5,-4){6.25}}
\put(30,49.25){$i_{1} \ {\rm{rows}}$}
\put(71,44.25){\rule{.25\unitlength}{.75\unitlength}}
\multiput(54,55.5)(-.033636364,-.062727273){275}{\line(0,-1){.062727273}}
\multiput(44.75,38.25)(.07352941,-.03361345){119}{\line(1,0){.07352941}}
\put(30,33){$i_{2} \ {\rm{rows}}$}
\put(65,23.25){\rule{11.5\unitlength}{10.5\unitlength}}
\end{picture}

It is easy to see that the base ring $K[{\bf{\mathcal{C}}}]$
associated to the transversal polymatroid ${\bf{\mathcal{P}}}$
presented by ${\bf{\mathcal{C}}}$ is generated by the following
set of monomials
\[\{x_{1},\ldots,x_{i_{1}}\}^{i_{1}+1-k}
\{x_{i_{1}+1},\ldots,x_{i_{2}}\}^{i_{2}-i_{1}+k-s}
\{x_{i_{2}+1},\ldots,x_{n}\}^{n-i_{2}+s-1}\] for any $0\leq k\leq
i_{1}+1$ and $0\leq s\leq i_{2}-i_{1}+k.$ If $x^{\alpha}\in
K[{\bf{\mathcal{C}}}],$ $\alpha=(\alpha_{1},\ldots,\alpha_{n})\in
{\mathbb{N}}^{n},$ then there exist $0\leq k \leq i_{1}+1$ and
$0\leq s \leq i_{2}-i_{1}+k$ such that
\[\alpha_{1}+\ldots + \alpha_{i_{2}}= i_{2}+1-s \ {\rm{and}} \
\alpha_{1}+\ldots + \alpha_{i_{1}}= i_{1}+1-k\] and thus,
$K[{\bf{\mathcal{C}}}]\subset K[A\cap B].$ Conversely, \ if $\alpha
\in A\cap B$ then $\alpha_{1}+ \ldots + \alpha_{i_{2}}\leq i_{2}+1 \
{\rm{and}} \ \alpha_{1}+\ldots + \alpha_{i_{1}}\leq i_{1}+1$ and so there
exist $0\leq k \leq i_{1}+1$ and $0\leq s \leq i_{2}-i_{1}+k$ such
that
\[\alpha_{1}+\ldots + \alpha_{i_{2}}= i_{2}+1-s \ {\rm{and}} \
\alpha_{1}+\ldots + \alpha_{i_{1}}= i_{1}+1-k\] and since
$| \ \alpha \ |=n$ it follows that
\[\{x_{1},\ldots,x_{i_{1}}\}^{i_{1}+1-k}
\{x_{i_{1}+1},\ldots,x_{i_{2}}\}^{i_{2}-i_{1}+k-s}
\{x_{i_{2}+1},\ldots,x_{n}\}^{n-i_{2}+s-1}.\] Thus,
$K[{\bf{\mathcal{C}}}]\supset K[A\cap B]$
and so $K[{\bf{\mathcal{C}}}]= K[A\cap B].$\\
\end{proof}

\chapter{A remark on the Hilbert series of transversal polymatroids}
\[\]

In this chapter we study when the transversal polymatroids presented by\\
${\bf{\mathcal{A}}}=\{A_1, A_2,\ldots, A_m\},$ with all sets $A_{i}$ having two
elements,  
have the base ring $K[{\bf{\mathcal{A}}}]$ Gorenstein. Using Worpitzky
identity, we prove that the numerator of Hilbert series has the coefficients
Eulerian numbers and from \cite{BE} it follows that the Hilbert series is unimodal.

\section{Segre product and the base ring associated to a transversal polymatroid.}

Let $K$ be an infinite field, $n$ and $m$ be positive integers,
$A_{i}$ be some subsets of $[n]$ for $1\leq i\leq m$, 
${\bf{\mathcal{A}}}=\{A_{1},A_{2},\ldots,A_{m}\}.$ Let
\[K[{\bf{\mathcal{A}}}]=K[x_{i_{1}}x_{i_{2}}\ldots x_{i_{m}} \ | \ i_{j}\in A_{j},1\leq
j\leq m]\]and
\[C=K[x_{i}y_{j} \ | \ i\in A_{j},1\leq j\leq m].\]
Obviously $C\subseteq S$, where S is the Segre product of polynomial
rings in $n$, respectively $m,$ indeterminates
\[S:=K[x_{1},x_{2},\ldots ,x_{n}]*K[y_{1},y_{2},\ldots
,y_{m}]=K[x_{i}y_{j} \ | \ 1\leq i\leq n,1\leq j\leq m]. \]

We consider the variables $t_{ij},$ $1\leq i\leq n,$ $1\leq j\leq
m,$\ and we define \[T=K[t_{ij} \ | \ 1\leq i\leq n , 1\leq j\leq m ], \]
\[T({\bf{\mathcal{A}}})= K[t_{ij} \ | \ 1\leq j\leq m,i\in A_j]. \]
We also consider the presentations $\phi:T\longrightarrow S$ and $\phi^{^{\prime}}:
T({\bf{\mathcal{A}}}
)\longrightarrow C$ defined by $t_{ij}\longrightarrow x_{i}y_{j}.$
By \cite[Proposition 9.1.2]{V} we know that $\ker(\phi)$ is the ideal 
$I_{2}(t)$ of the $2$--minors of the $n\times m$ matrix $t=(t_{ij})$
via the map $\phi$. The algebras $C$, $T({\bf{\mathcal{A}}})$, $S$ and $T$ are
$\mathbb{Z}^{m}$ -graded by setting $\deg(x_{i}y_{j})=
\deg(t_{ij})=e_{j}\in \mathbb{Z}^{m}$
where $e_{j}$, $1\leq j\leq m,$ denote the vectors of the canonical basis of 
${\mathbb{Z}}^{m}.$
By \cite[Propositions 4.11 and 8.11]{S} or \cite[Proposition
8.1.10]{V} we know that the cycles of the complete bipartite graph
$K_{n,m}$ give a universal Gr\"{o}bner basis of $I_{2}(t).$
A cycle of the complete bipartite graph is described by a pair
$(I,J)$ of sequences of integers, say
\[I=i_{1},i_{2},\ldots,i_{s},\ J=j_{1},j_{2},\ldots,j_{s},\] with
$2\leq s\leq \min(m,n)$, $1\leq i_{k}\leq m$, $1\leq j_{k}\leq n,$ and
such that the $i_{k}$ are distinct and the $j_{k}$ are distinct.
Associated with any such a pair we have a polynomial
$F_{(I,J)}=t_{i_{1}j_{1}}\ldots t_{i_{s}j_{s}}-t_{i_{2}j_{1}}\ldots
t_{i_{s}j_{s-1}}t_{i_{1}j_{s}}$ which is in $I_{2}(t)$.

For a ${\mathbb{Z}}^{m}$-graded algebra $E$ we denote by $E_{\Delta }$
the direct sum of the graded components of degree $(a,a,\dots ,a)\in
\mathbb{Z}^{m}.$ Similarly, for a $\mathbb{Z}^{m}-$graded $E-$module $M$,
we denote by $M_{\Delta }$ the direct sum of the graded components
of $M$ of degree $(a,a,\dots, a)\in
\mathbb{Z}^{m}.$ Clearly $E_{\Delta }$ is a $\mathbb{Z}-$graded algebra and $%
M_{\Delta }$ is a $\mathbb{Z}-$graded $E_{\Delta }$ module. Furthermore $%
-_{\Delta }$ is exact as a functor on the category of $\mathbb
{Z}^{m}-$ graded $E-$modules with maps of degree $0.$ Now $C_{\Delta
}$ is the $K$-algebra
generated by the elements $x_{i_{1}}y_{1}\dots x_{i_{m}}y_{m}$ with $%
i_{j}\in A_{j}$. Therefore $K[{\bf{\mathcal{A}}}]$ is isomorphic 
to the algebra $C_{\Delta }$ and we have the presentation

\[0\longrightarrow J\longrightarrow T({\bf{\mathcal{A}}})_{\Delta }
\longrightarrow\ K[{\bf{\mathcal{A}}}]\longrightarrow 0,\]
where $J=I_{2}(t)\cap T({\bf{\mathcal{A}}})_{\Delta } .$

$T({\bf{\mathcal{A}}})_{\Delta }$ is the $K-$algebra generated by the monomials $%
t_{1i_{1}}t_{2i_{2}}\ldots t_{mi_{m}},$ with $i_{k}\in A_{k}$, that is, $T(%
{\bf{\mathcal{A}}})_{\Delta }$ is the Segre product $T_{1}*T_{2}*\ldots *T_{m}$
of the
polynomial rings $T_{i}=K[t_{ij} \ | \ j\in A_i].$ Now we consider the variables $%
s_{\alpha}$ with $\alpha \in A :=A_{1}\times A_{2}\times \ldots
\times A_{m}.$ Then we get the presentation of the Segre product
$T({\bf{\mathcal{A}}})_{\Delta }$ as
a quotient of $K[s_{\alpha} \ | \ \alpha \in A]$ 
by mapping $s_(j_{1},\ldots ,j_{m})$ to $%
t_{1j_{1}}t_{2j_{2}}\ldots t_{mj_{m}}.$

From \cite{H} the defining ideal of $T({\bf{\mathcal{A}}})_{\Delta }$ is
generated by the so-called Hibi relations

\[
s_{\alpha }s_{\beta }-s_{({\alpha }\vee {\beta })}s_{{(\alpha }\wedge {\beta }%
)},
\]
where
\[
{\alpha }\vee {\beta }=(\max(\alpha_{1},\beta_{1}),\ldots,
\max(\alpha_{m},\beta_{m})),
\]
and
\[
{\alpha }\wedge {\beta }=(\min(\alpha_{1},\beta_{1}),\ldots,
\min(\alpha_{m},\beta_{m})).
\]

\begin{example}
Let $n=3$ and $\mathcal{A}=\{A_{1}=\{1,2\},A_{2}=\{2,3\},A_{3}=\{3,4\}\}.$ 
Then $C$
is the quotient of $K[t_{11}, t_{12}, t_{22}, t_{23}, t_{33},
t_{34}]$ by zero ideal ($J=0$ because we don't have cycles) and then
$K[{\bf{\mathcal{A}}}]=K[x_{1}x_{2}x_{3},x_{1}x_{2}x_{4},x_{1}x_{3}^{2},
x_{1}x_{3}x_{4},x_{2}^{2}x_{3},x_{2}^{2}x_{4},x_{2}x_{3}^{2},x_{2}x_{3}x_{4}]$
is the quotient of
$K[s_{123},s_{124},s_{133},s_{134},s_{223},s_{224},s_{233},s_{234}]$
modulo the ideal generated by the Hibi relations:
\begin{center}
{\em $s_{123}s_{134}-s_{124}s_{133}$ ,
$s_{123}s_{224}-s_{124}s_{223},$ }

{\em $s_{123}s_{234}-s_{124}s_{233}$ ,
$s_{123}s_{233}-s_{133}s_{223},$ }

{\em $s_{123}s_{234}-s_{133}s_{224}$ ,
$s_{123}s_{234}-s_{134}s_{223},$ }

{\em $s_{124}s_{234}-s_{134}s_{224}$ ,
$s_{133}s_{234}-s_{134}s_{233},$ }

{\em $s_{223}s_{234}-s_{224}s_{233}.$ }
\end{center}
Since $K[t_{11}, t_{12}]*K[t_{22}, t_{23}]*K[t_{33}, t_{34}]$ is a
Gorenstein ring ({\cite[Example 7.4]{HH}}), $B_{3}$ is a
Gorenstein ring .
\end{example}
\begin{example}
Let $n=3$ and $\mathcal{A}=\{A_{1}=\{1,2\},A_{2}=\{2,3\},A_{3}=\{3,1\}\}.$ Then $C$
is the quotient of $K[t_{11}, t_{12}, t_{22}, t_{23}, t_{33},
t_{31}]$ by the polynomial $t_{11}t_{22}t_{33}-t_{12}t_{23}t_{31}$ 
(we have one 6-cycle) and then
$K[{\bf{\mathcal{A}}}]=K[x_{1}x_{2}x_{3},x_{1}^{2}x_{2},x_{1}x_{3}^{2},
x_{1}^{2}x_{3},x_{2}^{2}x_{3},x_{1}x_{2}^{2},x_{2}x_{3}^{2}]$
is the quotient of
$K[s_{123},s_{121},s_{133},s_{131},s_{223},s_{221},s_{233},s_{231}]$
modulo the ideal generated by the Hibi relations:
\begin{center}
{\em $s_{221}s_{233}-s_{223}s_{231}$,
$s_{131}s_{233}-s_{133}s_{231},$}

{\em $s_{121}s_{233}-s_{123}s_{231}$,
$s_{131}s_{221}-s_{121}s_{231},$}

{\em $s_{133}s_{221}-s_{123}s_{231}$,
$s_{131}s_{223}-s_{123}s_{231},$}

{\em $s_{133}s_{223}-s_{233}s_{231}$,
$s_{121}s_{223}-s_{221}s_{231},$}

{\em $s_{121}s_{133}-s_{131}s_{231},$}
\end{center}
and by the linear relation {\em $$s_{123}-s_{231}.$$}
 Since $K[t_{11}, t_{12}]*K[t_{22}, t_{23}]*K[t_{33}, t_{31}]$ is a Gorenstein ring and
 $t_{11}t_{22}t_{33}-t_{12}t_{23}t_{31}$ is a regular element in $K[t_{11}, t_{12}]*
 K[t_{22}, t_{23}]*K[t_{33}, t_{31}]$, $\frac{K[t_{11}, t_{12}]*K[t_{22}, t_{23}]
 *K[t_{33}, t_{31}]}{(t_{11}t_{22}t_{33}-t_{12}t_{23}t_{31})}$
 $ \cong K[{\bf{\mathcal{A}}}]$ is a Gorenstein ring.
\end{example}

\section{Hilbert series}

\begin{definition}
Let $R=K[x_{1},x_{2},\ldots ,x_{n}]$ be a polynomial ring over a
field $K$. If $M$ is a finitely generated $\mathbb{N}-$graded $R$-module,
the numerical function
\[
H(M,-):{\mathbb{N}\longrightarrow \mathbb{N}}
\]
with $H(M,n)=\dim_{K}(M_{n}),$ for all $n\in \mathbb{N},$ is the Hilbert
function and \[H_{M}(t)=\sum_{n\in \mathbb{N}}H(M,n)t^{n}\] is the
Hilbert series of $M.$
\end{definition}
\begin{definition}
A sequence $(h_{i})_{i\geq 0}$ is \emph{log-concave} if 
$h_{i}^{2}\geq h_{i-1}h_{i+1}$ for all $i\geq 1$.
\end{definition}
\begin{definition}
A sequence $(h_{i})_{i\geq 0}$ is \emph{unimodal} if there exists an
index $j\geq 2$ such that $h_{i}\leq h_{i+1}$ for $i\leq j-1$ and 
$h_{i}\geq h_{i+1}$ for $i\geq j$.
\end{definition}

Log-concavity is easily shown to imply unimodality.

Let $n$, $m$ be  positive integers, $A_{i}$ be some subsets of $[n]$
such that $|A_{i}|=l$ for $1\leq i\leq m$, 
$\mathcal{A}=\{A_{1},A_{2},\ldots,A_{m}\},$
\[B_{m}=K[x_{i_{1}}x_{i_{2}}\ldots x_{i_{m}} \ | \ i_{j}\in A_{j},1\leq
j\leq m]\]and
\[C=K[x_{i}y_{j} \ | \ i\in A_{j},1\leq j\leq m].\]

From section above we know that $B_{m}$ is isomorphic to the algebra
$C_{\Delta }$ and we have the presentation

\[0\longrightarrow J\longrightarrow T(\mathcal{A})_{\Delta }\longrightarrow\ B_{m}\longrightarrow 0,\]
where $J=I_{2}(t)\cap T(\mathcal{A})_{\Delta } .$

Now we are interested when $J=(0)$.

\begin{remark}
If $J=(0)$ then $B_{m}$ is isomorphic to the algebra 
$T(\mathcal{A})_{\Delta }$.
\end{remark}
The ideal $J$ is zero if and only if when the bipartite graph
presented by $\mathcal{A}$  \ $( $V$_{1}=\{1,2,\ldots,m \}, V_{2}=A_{1}\cup
A_{2}\cup \ldots \cup A_{m}$ and the edges from $V_{1}$ to $V_{2}$ 
are the following:  join
$i\in\ V_{1}$ with $i_{j}\in \ V_{2} \Leftrightarrow\ i_{j} \in
A_{i} )$ has no cycles.
If $|A_{i}|=l$ for $1\leq i\leq m$, $|A_{i}\cap A_{i+1}| \leq 1$
\ and\ $A_{j}\cap A_{i}=\varnothing$ for $2\leq i\leq m$,
 $j< i-1$ then the bipartite graph
presented by $\mathcal{A}$ has no cycles , thus the ideal $J$ is
zero.\\
Since $J=(0)$, then $B_{m}$ is the Segre product of $m$ polynomial
rings, each of them in $l$ indeterminates, that is, $B_{m}$ is a
Gorenstein ring (see \cite[Example 7.4]{HH});
$\dim_{K}(B_{m})_{i}={\binom{i +l-1}{i}}^{m}.$
In the case $m=2$ it is known (see \cite[proposition 9.1.3]{V}) that
the Hilbert series of $B_{2}$ is
\[
H_{B_{2}}(t)=\frac{\sum_{k=0}^{l-1}{\binom{l-1}{k}}^{2}t^{k}}{(1-t)^{2l-1}}
; H(B_{2},i)=\dim_{K}(B_{2})_{i}={\binom{i +l-1}{i}}^{2}.\]
It results that the Krull dimension of $
B_{2} $ is $\dim_{K}{B_{2}}=2l-1$ and the number of generators of
the defining ideal of $B_{2}$ (the number of Hibi-relations of
$B_{2}$) is
\[
\mu=\binom{H(B_{2},1)+1}{2}-H(B_{2},2)=\binom{l^{2}+1}{2}-\binom{l+1}{2}^2=\binom{l}{2}^2.
\]

\begin{proposition}
We have the following relation between the Hilbert series of $B_{m+1}$ and
$B_{m}$

\[
H_{B_{m+1}}(t)=\frac{1}{(l-1)!}\frac{d^{(l-1)}}{dt^{l-1}}%
(t^{l-1}H_{B_{m}}(t)).
\]
\end{proposition}
\begin{proof}
Since 
\begin{align*}
&H_{B_{m}}(t)=\sum_{i\geq 0}{\binom{i+l-1}{i}^{m}}t^{i}, \ 
{\rm{we \ have}}\\
&\frac{1}{(l-1)!}\frac{d^{(l-1)}}{dt^{l-1}}(t^{l-1}H_{B_{m}}(t))=
\frac{1}{(l-1)!}\frac{d^{(l-1)}}{dt^{l-1}}(t^{l-1} \sum_{i\geq
0}{\binom{i +l-1}{i}}^{m}t^{i})\\
&=\frac{1}{(l-1)!}\frac{d^{(l-2)}}{dt^{l-2}}{(\frac{d}{dt}}(t^{l-1}
\sum_{i\geq 0}{\binom{i
+l-1}{i}}^{m}{t^{i}))}\\
&=\frac{1}{(l-1)!}\frac{d^{(l-2)}}{dt^{l-2}}{((l-1)t^{l-2}}
\sum_{i\geq 0}{\binom{i+l-1}{i}}^{m}t^{i}+t^{l-2}\sum_{i\geq
0}{\ {i} {\binom{i
+l-1}{i}}^{m}}t^{i})\\
&=\frac{1}{(l-1)!}\frac{d^{(l-2)}}{dt^{l-2}}{(t^{l-2}\sum_{i\geq
0}{\binom{i +l-1}{i}}}^{m}
{(i+l-1)}t^{i})\\
&=\frac{1}{(l-1)!}\frac{d^{(l-3)}}{dt^{l-3}}{(\frac{d}{dt}}(t^{l-2}
\sum_{i\geq 0}{\binom{i +l-1}{i}}^{m}{(i+l-1)}{t^{i}))}
\end{align*}
\begin{align*}
&=\frac{1}{(l-1)!}\frac{d^{(l-3)}}{dt^{l-3}}{((l-2)t^{l-3}\sum_{i
\geq 0}{\binom{i +l-1}{i}}^{m}}(i+l-1)t^{i}+ \\
&+ t^{l-3}\sum_{i\geq 0}{\ {i} {\binom{i +l-1}{i}}^{m}}
{(i+l-1)}{t^{i})}\\
&=\frac{1}{(l-1)!}\frac{d^{(l-3)}}{dt^{l-3}}{(t^{l-3}\sum_{i\geq
0}{\binom{i +l-1}{i}}^{m}{(i+l-1)}{(i+l-2)}t^{i})=\cdots}\\
&=\frac{1}{(l-1)!}\sum_{i\geq 0}{\binom{i
+l-1}{i}}^{m}{(i+l-1)}{(i+l-2)}\cdots{(i+2)}{(i+1)}t^{i}\\
&=\sum_{i\geq 0}{\binom{i +l-1}{i}}^{m+1}t^{i} =\
H_{B_{m+1}}(t).
\end{align*}
\end{proof}

\begin{definition}({\cite{RV}})
Let $A(t):=\sum_{i}{a_{i}t^{i}}$ and $B(t):=\sum_{i}{b_{i}t^{i}}$ be
two power series in $\mathbb {Z}[[t]]$. Then we define the Hadamard
product of $A$ and $B$ and we denote it by
$Had(A,B):=\sum_{i}(a_{i}b_{i})t^{i}.$
\end{definition}
\begin{definition}({\cite{RV}})
Let $A(t)$ be the Hilbert series of a standard $K-$algebra $S$.Then
we denote by $ri(A)$(or $ri(S)$) the regularity index of $A$ (or of
$S$), i.e. the first integer  $r$ such that for every $s\geq r$ the
Hilbert function of $S$ takes the same values as the Hilbert
polynomial of $S$.
\end{definition}
\begin{remark}
$ri(S)=a(S)+1,$ where $a(S)$ is the $a-$invariant of $S$.
\end{remark}
\begin{proposition}$({\cite{RV}})$
Let $A(t):=\frac{P(t)}{(1-t)^{a}}$ and
$B(t):=\frac{Q(t)}{(1-t)^{b}}$, where $p:=deg(P)$, $q:=deg(Q)$,
$P(1)\neq 0$, $Q(1)\neq 0$, and assume that $A(t)$ and $B(t)$ are
the Hilbert series of standard $K-$algebras. Then:
\begin{enumerate}
\item[1)]$ri(A)=p-a+1$ and $ri(B)=q-b+1$;
\item[2)]$ri(Had(A,B))\leq \max(ri(A),ri(B))$;
\item[3)]$Had(A,B)=\frac{R(t)}{(1-t)^{a+b-1}}$ with $R(1)\neq 0$;
\item[4)]$\deg(R)\leq \max(ri(A),ri(B))$ + $(a+b-1)-1$.
\end{enumerate}
\end{proposition}
\begin{theorem}$({\cite{RV}})$
Let $S_{1}$ and $S_{2}$ be two standard $K-$algebras and assume that
we know their Hilbert series, $H_{S_{1}}$ and $H_{S_{2}}$. Then the
Hilbert series of the Segre product of $S_{1}$ and $S_{2}$ is
$H_{S_{1}*S_{2}}$ = $Had(H_{S_{1}},H_{S_{2}})$.
\end{theorem}
\begin{definition}
Let $R=K[x_{1},x_{2},\ldots ,x_{n}]$ be a polynomial ring over a
field $K$, $M$ is a finitely generated $\mathbb{N}-$ graded $R$-module.
The ${\it difference\ operator\ {\Delta }}$ on the set of numerical
functions $H(M,-)$ is \[({\Delta }H(M,-))(n)=H(M,n+1)-H(M,n),\]
where $H(M,-)$ is the Hilbert function of $M$.

The $m-times\ iterated\ {\Delta}\ operator$ $\ (''m-difference\ of\
H(M,n)'')$  will be denoted by ${\Delta}^{m}$.
\end{definition}
\begin{proposition}
If $|A_{i}|=2$ for $1\leq i\leq m$, $|A_{i}\cap A_{i+1}|\leq 1$\
and \ $A_{j}\cap A_{i}=\varnothing$ for $1\leq i\leq m-1$,
 $1\leq j< i-1$ then the Hilbert series of $B_{m}$ is
\[
H_{B_{m}}(t)=\frac{\sum_{k=0}^{m-1}{\ A(m,k+1)}t^{k}}{(1-t)^{m+1}},
\] where
\[
A(m,k)\ =\ kA(m-1,k)+(m-k+1)A(m-1,k-1),
\]
with $A(m,1)=A(m,m)=1$\ and $2\leq k\leq m-1.$
\end{proposition}
\begin{proof}
We know that  \ $B_{m}=\ T_{1}*T_{2}*\ldots *T_{m}$,\ where\
$T_{i}=\ K[t_{ij} \ | \ j\in A_{i}]$ is the Segre product of $m$ polynomial
rings in two indeterminates and $\dim_{K}(B_{m})_{i}={\binom{\imath
+2-1}{i}}^{m}=(i+1)^{m}$.

We will show that $B_{m}$ has Krull dimension dim $B_{m}= m+1$ and
the Hilbert series,
$H_{B_{m}}(t)=\frac{R(t)}{(1-t)^{m+1}}$ with $\deg(R) \leq {m-1}.$\\
We proceed by induction on $m\geq 1$. If $m=1$ it is  clear.
Suppose $m \geq 2$.
For every $1\leq i \leq m$ we have $ri(H_{T_{i}})=-1$, thus
$ri(H_{B_{m}})=-1$.
Since $B_{m+1}=B_{m}*T_{m+1},$ we have
\[H_{B_{m+1}}(t)=Had(H_{B_{m}},T_{m+1})=\frac{R(t)}{(1-t)^{(m+1)+
2-1}}=\frac{R(t)}{(1-t)^{m+2}};\] \[\deg(R) \leq
\max(ri(H_{B_{m}}),ri(H_{T_{m+1}}))+ ((m+1)+ 2-1)-1=m.\]

Now we will find the coefficients $r_{i}'s$ of the Hilbert series 
$H_{B_{m}}(t)=\frac{R(t)}{(1-t)^{m+1}},$ where 
$R(t):=\sum_{k=0}^{m-1}{r_{i}t^{i}}$. We may compute the first $m$ values of
$H(B_{m},i)$. Then it suffices to take the $(m+1)^{st}$ difference
of these first $m$ values and we get the required $r_{i}'s$. For
this it suffices to go backward in the algorithm which determines
the numerators of the Hilbert series and to obtain
$H(B_{m},i)=dim_{K}(B_{m})_{i}$ for all $i$.

We define
\[A_{0}(m,k)=r_{k}=A(m,k), \]
\[A_{i}(m,1)=1, A_{i}(m,k)=A_{i}(m,k-1)+A_{i-1}(m,k),\] for $i\geq 1$ and $2\leq
k\leq m.$

For $m\geq 2$ and $2\leq k\leq m$ fixed we want to prove that
\[ A_{t}(m,k)=\sum_{s=1}^{k}{A(m,s)\binom{t+k-s-1}{k-s}}\]
for any $t\geq 1$ .\\
We proceed by induction on $t\geq 1$.\\
{\it Case $t=1$}.
Since for any $m\geq 2$ and $2\leq k\leq m$ fixed we have
\begin{align*}
&A_{1}(m,k)=A_{1}(m,k-1)+A(m,k),\\
&A_{1}(m,k-1)=A_{1}(m,k-2)+A(m,k-1),\\
&A_{1}(m,k-2)=A_{1}(m,k-3)+A(m,k-2),\\
&\ldots\ldots\ldots\ldots\ldots\ldots\ldots\ldots\\
&A_{1}(m,3)=A_{1}(m,2)+A(m,3),\\
&A_{1}(m,2)=A_{1}(m,1)+A(m,2),\\
&A_{1}(m,1)=1=A(m,1),\\
\end{align*}
we obtain \[A_{1}(m,k)=\sum_{s=1}^{k}{A(m,s)}.\]
{\it Case $t>1$}.\\
From
\begin{align*}
&A_{t}(m,k+1)=A_{t}(m,k)+A_{t-1}(m,k+1),\\
&A_{t-1}(m,k+1)=A_{t-1}(m,k)+A_{t-2}(m,k+1),\\
&A_{t-2}(m,k+1)=A_{t-2}(m,k)+A_{t-3}(m,k+1),\\
\end{align*}
\begin{align*}
&\ldots\ldots\ldots\ldots\ldots\ldots\ldots\ldots\ldots\\
&A_{3}(m,k+1)=A_{3}(m,k)+A_{2}(m,k+1),\\
&A_{2}(m,k+1)=A_{2}(m,k)+A_{1}(m,k+1),\\
&A_{1}(m,k+1)=A_{1}(m,k)+A(m,k+1),\\
\end{align*}
we obtain \[A_{t}(m,k+1)=\sum_{j=1}^{t}{A_{j}(m,k)}+A(m,k+1).\]
For $t>1$
\begin{align*}
&A_{t}(m,k+1)=\sum_{j=1}^{t}{A_{j}(m,k)}+A(m,k+1) \\
&=\sum_{j=1}^{t}\left(\sum_{s=1}^{k}A(m,s)\binom{j+k-s-1}{k-s}\right)+A(m,k+1)\\
&=\sum_{s=1}^{k}\left(\sum_{j=1}^{t}\binom{j+k-s-1}{k-s}\right)A(m,s)+A(m,k+1)\\
&=\sum_{s=1}^{k}\binom{t+k-s}{k-s+1}A(m,s)+A(m,k+1)\\
&=\sum_{s=1}^{k+1}{A(m,s)\binom{t+k-s}{k-s+1}},\\
\end{align*}
since \[\sum_{j=1}^{t}\binom{j+k-s-1}{k-s}=\binom{t+k-s}{k-s+1}.\]
Now we want to prove that $A_{m+1}(m,k)={k}^m.$\\
 From \cite{C1} or \cite{W} we mention the Worpitzky identity
$$k^{m}=\sum_{s=1}^{m}{\ A(m,s)}{\binom{k+s-1}{m}}.$$
We know that
\[
A_{m+1}(m,k)=\sum_{s=1}^{k}{A(m,s)\binom{m+k-s}{k-s}}=\sum_{s=1}^{k}{A(m,s)\binom{m+k-s}{m}}.
\]
Thus
\begin{align*}
&k^{m}=\sum_{s=1}^{m}{\ A(m,s)}\binom{k+s-1}{m}=\
A(m,m)\binom{k+m-1}{m}+\ A(m,m-1)\\
&\times \binom{k+m-1-1}{m}+\ldots +\ A(m,m-k+2)\binom{m+1}{m}+\
A(m,m-k+1)\binom{m}{m}\\
&=A(m,1)\binom{k+m-1}{m}+ \ A(m,2)\binom{k+m-2}{m}+\ldots +\
A(m,k-1)\binom{k+m-k+1}{m}\\
&+\ A(m,k)\binom{k+m-k}{m}=\sum_{s=1}^{k}{A(m,s)}\binom{m+k-s}{m}=\
A_{m+1}(m,k).\\
\end{align*}

Thus we have $r_{k}=A(m,k+1)$ for $0\leq k\leq m-1 .$
\end{proof}

\begin{corollary}
The sequence in $k$, $A(m,k)$ with $1\leq k\leq m$,  is symmetric
for any $m\geq 2.$
\end{corollary}
\begin{proof}
If $m=2$ then $A(2,1)$=$A(2,2)$=1.\\
If $m>2$ then $A(m,k)\ =\ k\ A(m-1,k)+(m-k+1)\ A(m-1,k-1)=\
kA(m-1,m-k)+(m-k+1)A(m-1,m-k+1)=\ A(m,m-k+1).$
\end{proof}
\begin{corollary}
The number of generators of the defining ideal of $B_{m}$ $(the \
number \ of \ Hibi-relations \ of \ B_{m})$ is
\[
\mu=\binom{H(B_{m},1)+1}{2}-H(B_{m},2)=\binom{2^{m}+1}{2}-3^{m}=2^{2m-1}+2^{m-1}-3^{m}.\]
\end{corollary}

\begin{corollary}
The h-vector of the Hilbert series associated to the transversal
polymatroid presented by $\mathcal{A}=\{A_1, A_2,\ldots, A_m\},$ such
that $|A_i|=2$ for $1\leq i\leq m$, $|A_i\cap A_{i+1}|\leq 1$ and\
$A_{j}\cap A_{i}={\varnothing}$ for $1\leq i\leq m-1$, $1\leq
j< i-1$, is unimodal .
\end{corollary}
\begin{proof}
From \cite{BE},\ we know that $A(m,k)$ is a log-concave sequence in
$k,$ for all $m$,\ thus it is unimodal.
\end{proof}

\vspace{2mm} \noindent {\footnotesize
\begin{minipage}[b]{10cm}
Alin \c{S}tefan, Assistant Professor\\
"Petroleum and Gas" University of Ploie\c{s}ti\\
Ploie\c{s}ti, Romania\\
E-mail:nastefan@upg-ploiesti.ro
\end{minipage}}


\begin{thebibliography}{99}

\bibitem{BE}  M.Bona, R.Ehrenborg, {\it{Combinatorial Proof of the Log-Concavity of
the Numbers of Permutations with $k$ Runs}}, Journal of Combinatorial
Theory, Series A 90, (2000), 293--303.

\bibitem{B} A. Br{\o}ndsted, {\it Introduction to Convex Polytopes},
 Graduate Texts in Mathematics 90, Springer-Verlag, 1983.
\bibitem{BH}  W. Bruns, J. Herzog, {\it Cohen-Macaulay rings,} Revised
Edition, Cambridge, 1997.

\bibitem{BG} W. Bruns, J. Gubeladze, {\it Polytopes, rings and
K-theory,} preprint.

\bibitem{BK} W. Bruns, R. Koch, \textit{Normaliz--a program for
 computing
normalizations of affine semigroups}, 1998. Available via anonymous
ftp from: ftp.mathematik.Uni-Osnabrueck.DE/pub/osm/kommalg/software.

\bibitem{C} A. Conca, {\it Linear Spaces, Transversal Polymatroids and
ASL Domains}, Journal of Algebraic Combinatorics, Volume 25,  Issue 1, (2007) 
25--41.

\bibitem{E1}  D.\ Eisenbud {\it Commutative algebra with a view toward 
algebraic geometry}, Springer, 1994.

\bibitem{GPS} G.-M.Greuel, G.Pfister, H.Sch\"{o}nemann. SINGULAR
2.0. {\it{A Computer Algebra System for Polynomial Computations}}. Centre
for Computer Algebra, University of
Kaiserslautern(2001). http://www.singular.uni-kl.de.

\bibitem{H} T. Hibi, {\it Algebraic Combinatorics on Convex Polytopes}, Carslaw
Publications, Glebe, N.S.W., Australia, 1992.

\bibitem{E} J. Edmonds,  {\it{Submodular functions, matroids, and certain
polyedra}}, in {\it Combinatorial Structures and Their Applications}, (R.
Guy, H. Hanani, N. Sauer, J. Schonheim, Eds.), Gordon and
Breach, New York, 1970.

\bibitem{C1}  L. Comtet. {\it{Advanced Combinatorics: The Art of Finite and
Infinite Expansions}}, rev. enl. ed. Dordrecht, Netherlands: Reidel,
1974.

\bibitem{RV}  L. Robbiano, G. Valla, {\it Hilbert-Poincar$\acute{e}$ Series of
Bigraded Algebra,} Bollettino U.M.I.(8), (1998), 521--540.

\bibitem{RW}  V. Reiner and V. Welker, {\it On the Charney-Davis and
Neggers-Stanley conjectures.} Avaible at www.math.umn.edu/~reiner.

\bibitem{O} J. Oxley, {\it Matroid Theory}, Oxford University Press,
Oxford, 1992.

\bibitem{H} T. Hibi, {\it Algebraic Combinatorics on Convex Polytopes},
Carslaw Publication, Glebe, N.S.W., Australia, 1992.

\bibitem{HH} J. Herzog, T. Hibi, {\it Discrete polymatroids},   J.
 Algebraic Combin.,  {16}(2002), no. 3, 239--268.
 
\bibitem{HN} E. De Negri, T Hibi, {\it Gorenstein algebras of Veronese type},
J. Algebra 193, (1997), 629--639.
 
\bibitem{HHV} J. Herzog, T. Hibi, M. Vl\u{a}doiu,  {\it{Ideals of fiber type
and polymatroids}}, Osaka J. Math. 42 (2005), 807--829.

\bibitem{S1}  R.\ P.\ Stanley, {\it Combinatorics and Commutative Algebra}, 
Birkh\"auser, 1983.

\bibitem{S}  B. Sturmfels, {\it Gr\"{o}bner bases and convex polytopes,}
Amer. Math. Soc., 1996.

\bibitem{SA} A. \c{S}tefan,  {\it A class of transversal polymatroids with
Gorenstein base ring,} Bull. Math. Soc. Sci. Math. Roumanie Tome 51(99)
No. 1, (2008), 67--79.

\bibitem{SA1} A. \c{S}tefan, {\it Intersections of base rings associated to 
transversal polymatroids}, to appear in Bull. Math. Soc. Sci. Math. Roumanie
Tome 51(99) No. 4, (2008), http://arXiv.org/pdf/:0805.2729.

\bibitem{SA2} A. \c{S}tefan, {\it The type of the base ring associated to 
a transversal polymatroid},  submitted, http://arXiv.org/pdf/:0807.2371.

\bibitem{SA3} A. \c{S}tefan, {\it A remark on the Hilbert series of 
transversal polymatroids} , Analele \c{S}tiin\c{t}ifice ale Universita\c{t}ii 
''Ovidius'' Constan\c{t}a Seria Matematica volumul XIV (2006), fascicola 2, 85--96.

\bibitem{SA4} A. \c{S}tefan, {\it The Facets Cone Associated To Some Classes
 Of  Transversal Polymatroids}, Analele \c{S}tiin\c{t}ifice ale Universita\c{t}ii 
''Ovidius'' Constan\c{t}a Seria Matematica volumul XV (2007), fascicola 1, 139--158.

\bibitem{SA5} A. \c{S}tefan, {\it Some examples of transversal polymatroids 
with Gorenstein base ring}, Preprint 2006.

\bibitem{MS} E. Miller, B.  Sturmfels, {\it Combinatorial commutative
algebra}, Graduate Texts in Mathematics 227, Springer-Verlag,
New-York, 2005.

\bibitem{V}  R. Villarreal, {\it Monomial Algebras,} Marcel Dekker,
 New-York, 2001.
 
\bibitem{V1} M. Vl\u{a}doiu, {\it{Discrete polymatroids}},  An. \c{S}t. Univ.
Ovidius, Constan\c ta, 14 (2006), 89--112.

\bibitem{V2} M. Vl\u{a}doiu, {\it{Equidimensional and unmixed ideals of
Veronese type}}, to appear in Communications in Alg., arXiv:math.
AC/0611326.

\bibitem{Z} G. M. Ziegler, {\it{Lecture on Polytopes}}, Graduate Texts
in Mathematics 152, Springer-Verlag, New-York, 1995.

\bibitem{W2} R. Webster, {\it{Convexity}}, Oxford University Press,
 Oxford, 1994.

\bibitem{W1} D. Welsh, {\it{Matroid Theory}}, Academic Press, London,
1976.

\bibitem{W3} N. White, {\it{The basis monomial ring of a matroid}}, Adv. in Math. 
24(1977) 292--294.

\bibitem{W} J. Worpitzky, {\it{Studien\"{u}ber die Bernoullischen und Eulerschen Zahlen.}}
J. reine angew. Math. 94, (1883), 203--232.

 
\end{thebibliography}
\end{document}